\pdfoutput=1
\documentclass[11pt,reqno]{amsart}
\usepackage{amsthm,amsmath,amssymb,booktabs,graphicx,array,xcolor}
\usepackage[margin=1in]{geometry}
\usepackage[utf8]{inputenc}
\usepackage[T1]{fontenc}
\PassOptionsToPackage{hyphens}{url}
\usepackage{float}
\usepackage{tikz}
\usetikzlibrary{arrows.meta,backgrounds}
\usepackage[hidelinks]{hyperref}
\usepackage{url}
\numberwithin{equation}{section}
\newtheorem{theorem}{Theorem}[section]
\newtheorem{proposition}[theorem]{Proposition}
\newtheorem{lemma}[theorem]{Lemma}
\newtheorem{corollary}[theorem]{Corollary}
\theoremstyle{definition}
\newtheorem{definition}[theorem]{Definition}
\newtheorem{example}[theorem]{Example}
\newtheorem{remark}[theorem]{Remark}
\newtheorem{question}{Question}
\newcommand{\F}{\mathbb F}
\newcommand{\Fbar}{\overline{\mathbb F}}
\newcommand{\cC}{\mathcal C}
\newcommand{\CL}{\mathcal C\mathcal L}
\newcommand{\cS}{\mathcal S}
\newcommand{\ad}{\operatorname{ad}}
\newcommand{\Ker}{\operatorname{Ker}}
\newcommand{\Ann}{\operatorname{Ann}}
\newcommand{\diag}{\operatorname{diag}}
\newcommand{\tr}{\operatorname{tr}}
\newcommand{\GL}{\operatorname{GL}}
\allowdisplaybreaks
\hypersetup{pdftitle={Generalized Centralizers and Polynomial Equivalence of Matrices: Answers to Questions of Dolinar, Guterman, Kuzma, and Markova},
pdfauthor={Hechun Zhang and Chengyi Zhu},
pdfsubject={MSC2020 15A27; 15A21, 15A24, 05C25},
pdfkeywords={Generalized centralizer, double centralizer theorem, polynomial equivalence,
commutativity up to a factor, anticommuting matrices, higher commutators,
directed graph, twins}}
\title[Generalized Centralizers and Polynomial Equivalence]{Generalized
Centralizers and Polynomial Equivalence of Matrices: Answers to Questions
of Dolinar, Guterman, Kuzma, and Markova}
\author[Hechun Zhang]{Hechun Zhang}
\address{Department of Mathematical Sciences, Tsinghua University,
Beijing, 100084, P. R. China}
\email{hczhang@tsinghua.edu.cn}
\author[Chengyi Zhu]{Chengyi Zhu}
\address{Department of Mathematics, University of Wisconsin--Madison,
Madison, WI 53706, USA}
\email{czhu284@wisc.edu}
\thanks{Hechun Zhang was partially supported by National Natural Science
Foundation of China grants No. 12031007 and No. 11971255.}
\thanks{\textbf{Disclosure statement.} No potential conflict of interest
was reported by the author(s).}
\date{}
\begin{document}

\begin{abstract}
For $n\times n$ matrices $A$ and $B$ over a field and a
scalar $\omega\ne0$, let $\mathcal C_\omega(A)=\{X:AX=\omega XA\}$. If
$\omega$ is a root of unity of order $q\ge2$, then
$\mathcal C_{\omega^j}(A)=\mathcal C_{\omega^j}(B)$ for all integers $j$ if and
only if each of $A$ and $B$ is a polynomial in the other with exponents $\equiv1\pmod q$. Over algebraically closed fields, for $\omega\ne1$, all $A$, and scalars
$\varepsilon\ne0$, we describe all $Y$ with $XY=\varepsilon YX$ for all
$X\in\mathcal C_\omega(A)$. Let $a(\mu)$ be the size of the largest Jordan block of $A$ at
$\mu$, or $0$ if there is none. If $\varepsilon$ is a power of $\omega$, all such $Y$ are polynomials
in $A$ if and only if $a(\omega\mu)=a(\mu)$ for all $\mu\ne0$, or this fails
only at some $\nu$ and at $\omega^{-1}\nu$, where $A$ has a single largest Jordan
block at $\nu$ and $a(\omega\nu)=a(\omega^{-1}\nu)=a(\nu)-1$. This answers two questions of Dolinar, Guterman, Kuzma,
and Markova. If $\omega\ne1$ has finite order, then for such $A$,
$\mathcal C_\omega(A)=\mathcal C_\omega(B)$ exactly when $A$ and $B$ are
polynomials in each other of this form. For all other $A$ some $B$
that is not a polynomial in $A$ satisfies
$\mathcal C_\omega(A)=\mathcal C_\omega(B)$ if the field has more than
$n^2+1$ elements. In the $\omega$-commuting digraph of Raja and Vaezpour,
with an edge from $A$ to $B$ if $AB=\omega BA$, vertices with edges to
and from the same vertices are twins. Over fields with more than
$2n^2+1$ elements we determine the vertices $A$ whose twins are all
polynomials in $A$, and describe
these twins. We find
all $A$, $B$, and $k\ge2$ for which the kernel of $(X\mapsto[A,X])^k$ is
the set of $X$ with $[X,[X,\ldots,[X,B]\ldots]]=0$ ($k$ brackets), where
$[X,Y]=XY-YX$. Apart from scalar $B$, this happens only in characteristic $\ne2$, for
$k=2$ and $2\times2$ matrices that generate the same quadratic extension.
\end{abstract}
\maketitle

\setcounter{tocdepth}{1}
\tableofcontents

\section{Introduction}
\label{intro}

Let $\F$ be a field, let $M_n(\F)$ be the algebra of $n\times n$ matrices
over $\F$, and let $A,B\in M_n(\F)$. For $\omega\in\F$ the generalized
centralizer of $A$ is
\[
 \cC_\omega(A)=\{X\in M_n(\F):AX=\omega XA\}.
\]
For $\omega=1$ it is the centralizer $\cC(A)$ of $A$, and for $\omega=-1$
it is the set $\CL(A)$ of the matrices that anticommute with $A$. We
call $\CL(A)$ the clifforder of $A$. Over an arbitrary field, this paper
studies how much information about $A$ the generalized centralizers
contain.

We write $\Fbar$ for an algebraic closure of $\F$,
$\F^\times=\F\setminus\{0\}$, and $\Fbar^\times=\Fbar\setminus\{0\}$. The
order of $\omega\in\F^\times$ is the least integer $q\ge1$ with
$\omega^q=1$, and it is infinite if no such integer exists. An element of
finite order $q$ is a root of unity of order $q$, and the powers of
$\omega$ are the elements $\omega^i$ with integers $i$. We write
$I$ and $O$ for identity and zero matrices. A scalar matrix is a scalar
multiple of $I$, and a matrix $A$ is nilpotent if $A^k=O$ for some
integer $k\ge1$. We write $\F[x]$ for the ring of polynomials in a
variable $x$ with coefficients in $\F$, $\deg f$ for the degree of a
polynomial $f$, $m_A$ for the minimal polynomial of $A$, and
$J_n(\lambda)$ for the Jordan block of size $n$ with eigenvalue
$\lambda$. The Jordan blocks of $A$ at $\mu\in\Fbar$ are the blocks with
eigenvalue $\mu$ in the Jordan form of $A$ over $\Fbar$, and the
generalized eigenspace of $A$ at $\mu$ is the kernel of $(A-\mu I)^n$ in
the space $\Fbar^n$ of column vectors over $\Fbar$.

The information that the centralizer contains about $A$ is described by
a classical theorem. The double centralizer
theorem states that the set $\cC(\cC(A))$ of the matrices that commute
with all elements of $\cC(A)$ is the algebra $\F[A]$ of polynomials in
$A$.
Jacobson~\cite[p.~113]{Jacobson1953} and Suprunenko and
Tyshkevich~\cite[p.~33]{SuprunenkoTyshkevich1968} proved it over an
arbitrary field. Robinson~\cite[pp.~112--113]{Robinson1976} described its
history: Clifford tried to prove in 1875 that every matrix commuting with
$A$ is a polynomial in $A$, Sylvester showed in 1884 that this is false in
general, and Robinson attributed the double centralizer theorem to Autonne
(1910). A consequence is that $\cC(A)\subseteq\cC(B)$ holds if and only
if $B\in\F[A]$, so that two matrices have the same centralizer if and only
if each of them is a polynomial in the other. Dolinar, Guterman, Kuzma,
and Oblak~\cite[Definition~1.3 and Lemma~2.4]{DGKO2013} called such
matrices $\cC$\nobreakdash-equivalent. We call them polynomially
equivalent. Thompson~\cite{Thompson1969} determined when $A$ and a
polynomial $f(A)$ are polynomially equivalent.

The equation $AX=\omega XA$ has an equally long history. Holtz,
Mehrmann, and Schneider~\cite[Section~8]{HMS2004} traced it back to
Cayley's memoir of 1858 for $\omega=-1$ and to the work of Cecioni,
Cherubino, and Kurosaki for general $\omega$. Kurosaki~\cite{Kurosaki1941}
computed the dimension of the solution space over an arbitrary field and
noted that, for an invertible matrix $A$, the equation has an invertible
solution exactly when $A$ is similar to $\omega A$. Potter~\cite{Potter1950}
proved that $(A+B)^q=A^q+B^q$ for matrices with $AB=\omega BA$, where
$\omega$ is a root of unity of order $q$, and he described the
eigenvalues of $A+B$ and of $AB$. Holtz, Mehrmann, and
Schneider~\cite{HMS2004} revisited this theorem and its proofs, and they
derived it for elements of a ring that satisfy the same relation with an
element $\omega$ that commutes with them, provided that the subring
generated by $1$ and $\omega$ is an integral
domain~\cite[Theorem~2 and Corollary~3]{HMS2004}.

Dolinar, Guterman, Kuzma, and Markova~\cite{DGKM2019} asked
whether a statement of the same kind as the double centralizer theorem
holds for the generalized centralizers. For a set $\cS\subseteq M_n(\F)$ and
$\varepsilon\in\F$ let
\[
 \cC_\varepsilon(\cS)=\{Y\in M_n(\F):XY=\varepsilon YX\text{ for all }
 X\in\cS\}.
\]
They called $\cC_\omega(\cC_\omega(A))$ the double generalized centralizer
of $A$ and, for $\omega\ne0$, $\cC_{\omega^{-1}}(\cC_\omega(A))$ the
symmetrized double generalized centralizer. They determined both sets for nilpotent matrices~\cite[Theorem~11]{DGKM2019} and the second set
for certain diagonal matrices~\cite[Theorem~12]{DGKM2019}, and they
computed the first set for several invertible
matrices~\cite[Examples~13--15]{DGKM2019}. They also described the
general case in terms of the nilpotent and the invertible part of
$A$~\cite[Theorem~17]{DGKM2019}, and over algebraically closed fields they classified the matrices with
minimal and maximal generalized centralizers~\cite{DGKM2018}. Their paper~\cite{DGKM2019} ends with four questions. We state them in our
notation and with the assumption that the field is algebraically closed,
which Dolinar et~al.~\cite[Section~2]{DGKM2019} made throughout their
paper. We use the name double generalized centralizers for all the sets
$\cC_\varepsilon(\cC_\omega(A))$ with $\varepsilon\ne0$. The
spectrum of a matrix is the set of its eigenvalues, and $A\oplus B$ is the
block diagonal matrix with the diagonal blocks $A$ and $B$. For sets
$\cS$ and $\cS'$ of matrices, $\cS\oplus\cS'$ is the set of the matrices
$X\oplus X'$ with $X\in\cS$ and $X'\in\cS'$. In Question~\ref{Q2} the
letter $B$ denotes a matrix in $M_m(\F)$ for some $m\ge1$. We write
$A\oplus B$ in place of the letter that they used for this matrix, and
$M_m(\F)$ in place of their $M_n(\F)$. In Question~\ref{Q3} we write $q$ for the order
of $\omega$, which they denoted by $t+1$.

\begin{question}[Dolinar, Guterman, Kuzma, and Markova
{\cite[Question~22]{DGKM2019}}]
\label{Q1}
Let $\F$ be algebraically closed, let $\omega\in\F$ be a root of unity,
and let $A\in M_n(\F)$ be a matrix whose spectrum is closed under
multiplication by $\omega$. Under which conditions on the sizes of the
Jordan blocks of $A$ does the inclusion
$\cC_\omega(\cC_\omega(A))\subseteq\F[A]$ hold?
\end{question}

\begin{question}[Dolinar, Guterman, Kuzma, and Markova
{\cite[Question~23]{DGKM2019}}]
\label{Q2}
Let $\F$, $\omega$, and $A$ be as in Question~\ref{Q1}, and suppose that
$\cC_\omega(\cC_\omega(A))\subseteq\F[A]$. Let $B\in M_m(\F)$ be a matrix
whose spectrum is disjoint from that of $A$ and such that, for every
eigenvalue $\beta$ of $B$, the element $\omega\beta$ is not an eigenvalue
of $B$. Does the inclusion
$\cC_\omega(\cC_\omega(A\oplus B))\subseteq\F[A]\oplus M_m(\F)$ hold?
\end{question}

\begin{question}[Dolinar, Guterman, Kuzma, and Markova
{\cite[Question~24]{DGKM2019}}]
\label{Q3}
Let $\F$ be algebraically closed, let $\omega\in\F$ be a root of unity of
order $q\ge2$, and let $A\in M_n(\F)$ be a direct sum of Jordan blocks.
Suppose that the spectrum of $A$ is
$\{\lambda,\omega\lambda,\ldots,\omega^{q-1}\lambda\}$ for some
$\lambda\ne0$. Write $A=D+N$ with $D$ diagonal and $N$ nilpotent, and
let $\F D$ be the set of the scalar multiples of $D$. The sum
$\F D+\F[N]$ of the subspaces $\F D$ and $\F[N]$ of $M_n(\F)$ is direct.
Under which conditions on the sizes of the Jordan blocks of $A$ does the
equality $\cC_\omega(\cC_\omega(A))=\F D+\F[N]$ hold?
\end{question}

The condition $q\ge2$ in Question~\ref{Q3} is part of the question,
because Dolinar et~al.~\cite[Section~1]{DGKM2019} defined roots of unity
of order $k$ only for $k>1$. For $q=1$ the matrix $D$ would be a
scalar matrix, and the sum $\F D+\F[N]$ would not be direct.

\begin{question}[Dolinar, Guterman, Kuzma, and Markova
{\cite[Question~25]{DGKM2019}}]
\label{Q4}
Let $\F$ be algebraically closed, and let $\omega\in\F$. For which
matrices $A$ does the equality $\cC_\omega(\cC_\omega(A))=\F[A]$ hold?
\end{question}

To the best of our knowledge, these questions have not been answered.

Our first result extends the description of the matrices with the same
centralizer to the generalized centralizers for all powers of a root of
unity. For an integer $q\ge2$ we say that $A$ and $B$ are
$q$\nobreakdash-polynomially equivalent if each of them is a polynomial in
the other in which only powers $x^{qi+1}$ with integers $i\ge0$ occur,
that is, $B\in A\F[A^q]$ and $A\in B\F[B^q]$, where for an integer
$j\ge0$ we write $A^j\F[A^q]$ for the set of the matrices $A^jh(A^q)$
with $h\in\F[x]$, and define $B^j\F[B^q]$ in the same way. If $q=2$, this means that each of $A$ and $B$ is an odd
polynomial in the other, and we call $A$ and $B$ odd polynomially
equivalent.

\begin{theorem}
\label{family}
Let $\omega\in\F^\times$ be a root of unity of order $q\ge2$, and let
$A,B\in M_n(\F)$. Then $\cC_{\omega^j}(A)\subseteq\cC_{\omega^j}(B)$ for
all integers $j$ with $0\le j<q$ if and only if $B\in A\F[A^q]$.
Consequently, $\cC_{\omega^j}(A)=\cC_{\omega^j}(B)$ for all these $j$ if
and
only if $A$ and $B$ are $q$\nobreakdash-polynomially equivalent.
\end{theorem}

For $q=2$ and a field of characteristic
different from $2$, Theorem~\ref{family} says that $\cC(A)=\cC(B)$ and
$\CL(A)=\CL(B)$ hold together if and only if $A$ and $B$ are odd
polynomially equivalent (see Corollary~\ref{oddcor}). It suffices to
assume the inclusions for $0\le j\le q/2$, and none of them can be
omitted (see Remark~\ref{allj}). The proof
applies the double centralizer theorem to the block diagonal matrix with
the diagonal blocks $A,\omega A,\ldots,\omega^{q-1}A$; it is in
Section~\ref{famsec}, after Lemma~\ref{cover}.

A single generalized centralizer contains less information. By the
definitions, for $\omega\ne0$ the inclusion
$\cC_\omega(A)\subseteq\cC_\omega(B)$ holds if and only if
$B\in\cC_{\omega^{-1}}(\cC_\omega(A))$, and if $\omega$ has finite order
$q$ and $j$ is an integer, then the powers $A^m$ with $m\ge0$ and
$m\equiv j\pmod q$ lie in
$\cC_{\omega^{-j}}(\cC_\omega(A))$ (see Lemma~\ref{basic}(c)). For $n\ge2$
and $j$ not divisible by $q$ this is a proposition of Dolinar
et~al.~\cite[Proposition~9]{DGKM2019}. Our
second result determines the matrices $A$ for which these sets consist of
polynomials in $A$. For $\mu\in\Fbar^\times$ let $a(\mu)$ be the size of
the largest Jordan block of $A$ at $\mu$, and let $a(\mu)=0$ if $\mu$ is
not an eigenvalue of $A$.

\begin{theorem}
\label{mainA}
Let $\F$ be a field, let $\omega\in\F\setminus\{0,1\}$ have order $q$,
let $A\in M_n(\F)$, and let $\varepsilon=\omega^{-j}$ with an integer
$j$. The following statements are equivalent.
\begin{enumerate}
\item[(i)] We have $\cC_\varepsilon(\cC_\omega(A))\subseteq\F[A]$.
\item[(ii)] One of the following two conditions holds.
\begin{enumerate}
\item[(a)] For every $\mu\in\Fbar^\times$ we have
$a(\omega\mu)=a(\mu)$.
\item[(b)] There is an element $\nu\in\Fbar^\times$ such that $A$ has
exactly one Jordan block of the largest size $a(\nu)$ at $\nu$,
$a(\omega\nu)=a(\omega^{-1}\nu)=a(\nu)-1$, and $a(\omega\mu)=a(\mu)$ for
all $\mu\in\Fbar^\times\setminus\{\nu,\omega^{-1}\nu\}$.
\end{enumerate}
\end{enumerate}
If $q$ is finite, $0\le j<q$, and these statements hold, then
$\cC_{\omega^{-j}}(\cC_\omega(A))=A^j\F[A^q]$.
\end{theorem}

Condition (a) says that the largest Jordan blocks of $A$ at $\mu$ and at
$\omega\mu$ have the same size for every nonzero $\mu$. It is equivalent
to $m_A(\omega x)=\omega^{\deg m_A}m_A(x)$ for the minimal polynomial
$m_A$ of $A$, and then the equality
$\cC_{\omega^{-j}}(\cC_\omega(A))=A^j\F[A^q]$ follows from a theorem of
Thompson~\cite[Theorem~1]{Thompson1965} (see
Corollary~\ref{thompcor}). Every nilpotent matrix satisfies (a), and so
does every matrix $A$ that is similar to $\omega A$. In (b) with $a(\nu)=1$
the element $\nu$ is a simple eigenvalue of $A$, and $\omega\nu$ and
$\omega^{-1}\nu$ are not eigenvalues. Condition (b) can also be stated
over $\F$, and then $\nu\in\F$ (see Remark~\ref{abfield}). The proof of
Theorem~\ref{mainA} is in Section~\ref{dgcsec}, after Lemma~\ref{dimpow}. It
rests on Theorem~\ref{structure}, which describes
$\cC_\varepsilon(\cC_\omega(A))$ for every matrix $A$ over an
algebraically closed field and every $\varepsilon\in\F^\times$. If $\varepsilon\in\F^\times$ is not a power of $\omega$, the inclusion
$\cC_\varepsilon(\cC_\omega(A))\subseteq\F[A]$ holds if and only if $A$
satisfies (a), or $A$ satisfies (b) and in addition either $a(\nu)=1$ or
$A$ has a Jordan block of size $a(\nu)-1$ at $\nu$ (see
Proposition~\ref{nonpow}).

For $\omega\ne1$, Theorem~\ref{mainA} answers Question~\ref{Q1} (see
Corollary~\ref{q22}), and the answer to Question~\ref{Q2} is positive
exactly for the matrices that satisfy (a), and therefore negative in
general (see Corollary~\ref{q23} and Example~\ref{q23ex}). For $\omega=1$
the inclusion of Question~\ref{Q1} holds for every matrix by the double
centralizer theorem, and no matrix $B$ satisfies the assumption of
Question~\ref{Q2}. As stated, the equality of Question~\ref{Q3} holds for
no matrix (see Remark~\ref{q24rem}). The equality of Question~\ref{Q4}
holds for every matrix if $\omega=1$. If $\omega\ne1$, it holds for no
matrix of size $n\ge2$ and, for $n=1$, exactly for $A\ne O$ (see
Remark~\ref{q25}). Proposition~\ref{q24} gives bases of
$\cC_\omega(\cC_\omega(A))$ and $\cC_{\omega^{-1}}(\cC_\omega(A))$ in terms
of $D$ and $N$ for the matrices of Question~\ref{Q3} whose largest Jordan
blocks at all eigenvalues have the same size.

Our third result concerns the equality of generalized centralizers. We
call $A$ and $B$ $\omega$\nobreakdash-equivalent if
$\cC_\omega(A)=\cC_\omega(B)$.

\begin{theorem}
\label{mainB}
Let $\F$ be a field, let $\omega\in\F\setminus\{0,1\}$ have order $q$,
and let $A\in M_n(\F)$.
\begin{enumerate}
\item[(a)] Suppose that $A$ satisfies condition (ii) of
Theorem~\ref{mainA}, and let $B\in M_n(\F)$. If
$\cC_\omega(A)\subseteq\cC_\omega(B)$, then $B$ is a polynomial in $A$. If
$q$ is finite, then $\cC_\omega(A)\subseteq\cC_\omega(B)$ if and only if
$B\in A\F[A^q]$, and $\cC_\omega(A)=\cC_\omega(B)$ if and only if $A$
and $B$ are $q$\nobreakdash-polynomially equivalent.
\item[(b)] Suppose that $A$ does not satisfy condition (ii) of
Theorem~\ref{mainA} and that $\F$ has more than $n^2+1$ elements. Then
there is a matrix $B\in M_n(\F)$ with $\cC_\omega(A)=\cC_\omega(B)$ that
is not a polynomial in $A$.
\end{enumerate}
\end{theorem}

For $\omega=-1$ and a field of characteristic different from $2$ with
more than $n^2+1$ elements, Theorem~\ref{mainB} describes the matrices
$A$ for which every $B\in M_n(\F)$ with $\CL(A)=\CL(B)$ is odd
polynomially equivalent to $A$. Over every field of characteristic different
from $2$, these include the balanced matrices, that is, the matrices $A$
that are similar to $-A$, and the matrix $J_3(\lambda)\oplus(-J_4(\lambda))$
with $\lambda\ne0$ (see Corollary~\ref{clcor} and Example~\ref{clex}). If $\omega=-1$ and $A\in M_n(\mathbb R)$ is the adjacency matrix of a
graph, that is, the
matrix with the entry $1$ at each pair of joined vertices and $0$
elsewhere, then condition (ii) is a condition on the eigenvalues of $A$ and on the
multiplicity of one of them (see Example~\ref{adjex}). Theorem~1 of Thompson~\cite{Thompson1965} gives the inclusion statements
of Theorem~\ref{mainB}(a) for the matrices $A$ that satisfy condition (a)
of Theorem~\ref{mainA} (see Corollary~\ref{thompcor}). For such $A$ the
matrices $A$ and $\omega A$ have the same minimal polynomial (see
Lemma~\ref{balroots}), so that Theorem~2 of
Thompson~\cite{Thompson1969} applies to the equations
$AX=X(\omega A)$ and $BX=X(\omega B)$. It shows that
$\cC_\omega(A)=\cC_\omega(B)$ if and only if $B=f(A)$ for a polynomial
$f$ with $\omega f(A)=f(\omega A)$ that is injective on the roots of
$m_A$ and whose derivative is nonzero at the multiple roots of $m_A$.
For nilpotent $A$ over an
algebraically closed field, the inclusion statements also follow from a
theorem of Dolinar et~al.~\cite[Theorem~11]{DGKM2019}. The proof
of Theorem~\ref{mainB} is in Section~\ref{dgcsec}, after
Remark~\ref{q24rem}. In Proposition~\ref{rankprop} we also show that,
for $\omega\notin\{0,1,-1\}$, a matrix $B$ with
$\cC_\omega(A)=\cC_{\omega^{-1}}(B)$ exists if and only if the matrices $A$
and $A^2$ have the same rank.

The nonzero elements of the generalized centralizers describe the edges
of a directed graph. Raja and Vaezpour~\cite{RajaVaezpour2011} defined a
directed graph $\Gamma_\omega(M_n(\F))$ whose vertices are the matrices
$A$ that are not scalar matrices and for which $A$ and $\omega A$ have a
common eigenvalue. It has an edge from $A$ to $B$ if $AB=\omega BA$, and
following Markova~\cite{Markova2026} we call it the
$\omega$\nobreakdash-commuting digraph. A vertex $B$ is a successor of a
vertex $A$ if there is an edge from $A$ to $B$, and a predecessor of $A$ if
there is an edge from $B$ to $A$. For $\omega\notin\{0,1\}$ the
successors of a vertex $A$ are the nonzero elements of $\cC_\omega(A)$,
and its predecessors are the nonzero elements of $\cC_{\omega^{-1}}(A)$
(see Lemma~\ref{vertex}). Thus Theorem~\ref{mainB} describes the vertices
with the same successors as $A$ (see Corollary~\ref{nbhd}). Two vertices
are twins if they have the same successors and the same predecessors. By
Theorem~\ref{twins}, every twin of $A$ is a polynomial in $A$ if and only
if there is at most one $\nu\in\Fbar^\times$ with $a(\nu)>a(\omega\nu)$
and $a(\nu)>a(\omega^{-1}\nu)$, and at such an element $\nu$ the matrix
$A$ has exactly one Jordan block of size $a(\nu)$ and
$\max\{a(\omega\nu),a(\omega^{-1}\nu)\}=a(\nu)-1$. We call this condition
(T). The necessity is proved for fields $\F$ with more than $2n^2+1$
elements. For $\omega=-1$ condition (T) is equivalent to condition (ii) of Theorem~\ref{mainA}, and for
$\omega^2\ne1$ it is weaker, so that all twins of a vertex $A$ can be
polynomials in $A$ although some vertex with the same successors as $A$ is
not (see Remark~\ref{twinrem}). The proof of Theorem~\ref{twins} is in
Section~\ref{graphsec}, after Proposition~\ref{Dprop}.

Under condition (T), Theorem~\ref{twinsdesc} describes the twins of $A$
over every field. They are the matrices $f(A)$ for the polynomials
$f\in\F[x]$ with the following three properties: the monic greatest
common divisor of $m_A(x)$ and $m_A(\omega x)$ divides
$f(\omega x)-\omega f(x)$, an equation $f(\lambda)=\omega f(\mu)$ for
eigenvalues $\lambda$ and $\mu$ of $A$ in $\Fbar$ holds only if
$\lambda=\omega\mu$, and the derivative $f'$ is nonzero at the multiple
roots of $m_A$. These twins are polynomially equivalent to $A$. If
$\omega$ has finite order $q\le4$, they are exactly the matrices that are
$q$\nobreakdash-polynomially equivalent to $A$ (see
Corollary~\ref{twinsq}), and if $\omega$ has finite order $q\ge5$, this
fails for some vertices (see Example~\ref{twinex}).

Let $\ad_A$ be the linear
map $X\mapsto AX-XA$ of $M_n(\F)$, and for $k\ge1$ let $\Ker\ad_A^k$ be the
kernel of its $k$th power. For $B\in M_n(\F)$ let $\Ann_k(B)$ be the set
of the matrices $X$ with $\ad_X^k(B)=O$. Then $\Ker\ad_A=\cC(A)$ and
$\Ann_1(B)=\cC(B)$. For $k\ge2$ we call $\ad_X^k(B)$ a higher commutator
of $X$ and $B$. With the commutator $[X,Y]=XY-YX$ it is
$[X,[X,\ldots,[X,B]\ldots]]$ with $k$ brackets. Over algebraically closed fields with a condition on the characteristic,
Marcus and Khan~\cite[Theorem~2]{MarcusKhan1960} for $k=2$ and
Smiley~\cite{Smiley1961} for every $k$ proved that the inclusion
$\Ker\ad_A^k\subseteq\Ann_k(B)$ implies $B\in\F[A]$.
Robinson~\cite{Robinson1965} proved this over every field (see
Theorem~\ref{robinson}). For $k\ge2$
the set $\Ann_k(B)$ is in general not a subspace, and the equality
$\Ker\ad_A^k=\Ann_k(B)$ is a much stronger condition.

\begin{theorem}
\label{annthm}
Let $n\ge2$ and $k\ge2$ be integers, and let $A,B\in M_n(\F)$. Then
$\Ker\ad_A^k=\Ann_k(B)$ if and only if one of the following two
conditions holds.
\begin{enumerate}
\item[(i)] The matrix $B$ is a scalar matrix and $\ad_A^k=0$.
\item[(ii)] We have $n=2$ and $k=2$, the characteristic of $\F$ is not
$2$, the characteristic polynomial of $B$ is irreducible over $\F$, and
$\F[A]=\F[B]$.
\end{enumerate}
\end{theorem}

The proof of Theorem~\ref{annthm} is in Section~\ref{highsec}, after
Example~\ref{c2ex}.

In coding theory the spaces $\cC_\omega(A)$ are called twisted centralizer
codes. Alahmadi, Glasby, Praeger, Sol\'e, and Yildiz~\cite{AGPSY2017}
investigated these codes, and Alahmadi, Glasby, and
Praeger~\cite{AGP2017} studied their dimensions. Holtz, Mehrmann, and Schneider~\cite[Theorem~5]{HMS2004}
described the Jordan form of an invertible complex matrix $A$ for which $AB=\omega BA$
has an invertible solution $B$. A theorem of
Thompson~\cite[Theorem~1]{Thompson1965} implies a description of the
double generalized centralizers of all matrices that satisfy condition (a)
of Theorem~\ref{mainA} (see Corollary~\ref{thompcor}). This description contains Theorems~11 and~12 of Dolinar
et~al.~\cite{DGKM2019}, and for $\omega$ of
order $3$ it shows that the double generalized centralizer displayed in
the fourth item of their Example~13 is not correct.

Raja and Vaezpour~\cite{RajaVaezpour2011} studied whether the
$\omega$\nobreakdash-commuting digraph and its subgraphs on several sets
of matrices are strongly connected, that is, whether every vertex
can be reached from every other vertex along edges, and they studied the
diameter, which is the largest number of edges needed for this.
Markova~\cite{MarkovaSEMR2026,Markova2026} recently determined its
diameter for $n\ge3$ and $\omega\notin\{0,1,-1\}$ over algebraically
closed fields and studied its subgraph on the upper triangular matrices.
She used the description of $\cC_\omega(A)$ for matrices in Jordan form
due to Dolinar, Guterman, Kuzma, and
Markova~\cite[Proposition~7]{DGKM2019}, which Lemma~\ref{decomp}
restates. Dolinar and Kuzma~\cite[Corollary~2.9]{DolinarKuzma2010} proved
for complex matrices that if every matrix $X$ with $XA=\xi AX$ for some
nonzero $\xi$ also satisfies $XB=\eta BX$ for some nonzero $\eta$, then $B$
is a polynomial in $A$. For a fixed $\omega\in\F\setminus\{0,1\}$,
Theorem~\ref{twins}
gives a condition on the Jordan form of $A$ under which every twin of $A$
in the $\omega$\nobreakdash-commuting digraph is a polynomial in $A$, and
it shows that this condition is also necessary over fields with more than
$2n^2+1$ elements.

The classical results on higher commutators describe the inclusion
$\Ker\ad_A^k\subseteq\Ann_k(B)$. For $k=2$ and an algebraically closed
field of characteristic $0$, Marcus and
Khan~\cite[Theorem~2]{MarcusKhan1960} showed that the matrices $B$ with
this property form a subspace of $\F[A]$ whose dimension is the number of
distinct eigenvalues of $A$. Robinson~\cite[Theorem~5.5]{Robinson1971}
determined the set of these matrices for every $k$ and every field: it is
$\F[A^{p^e}]$ in characteristic $p$, where $p^e$ is the least power of $p$
with $p^e\ge k$, and in characteristic $0$ it is $\F[A]$ for $k=1$, while for $k\ge2$ it
consists of the matrices in $\F[A]$ that are diagonalizable over $\Fbar$
(see Theorem~\ref{robeq}). Theorem~\ref{annthm} shows that the
corresponding equality holds only in the two cases of that theorem. By Lemma~\ref{polyeq} and
Theorem~\ref{robinson}, two matrices $A$ and $B$ are polynomially
equivalent if and only if
$\Ker\ad_A^{k_1}=\Ann_{k_1}(B)$ and $\Ker\ad_B^{k_2}=\Ann_{k_2}(A)$ for
some integers $k_1,k_2\ge1$, and by Theorem~\ref{annthm}, for
$k_1,k_2\ge2$ this happens only for scalar
matrices and, if the characteristic of $\F$ is not $2$, for pairs of
$2\times2$ matrices that generate the same quadratic field extension of
$\F$ (see Corollary~\ref{annbyann}).

To the best of our knowledge, Theorems~\ref{family}, \ref{annthm},
\ref{twins}, and~\ref{twinsdesc}, Theorem~\ref{mainA} except for the
sufficiency of condition (a), Theorem~\ref{mainB} except for part (a)
for the matrices that satisfy condition (a),
Propositions~\ref{rankprop}, \ref{nonpow}, and~\ref{Dprop}, and
Corollary~\ref{twinsq} are new. Theorem~\ref{structure} is new as well,
and it extends Theorem~17 of Dolinar et~al.~\cite{DGKM2019}.

The proofs use only linear algebra. The main tool is
Theorem~\ref{structure}. For every $\mu\in\Fbar$, every element of
$\cC_\omega(A)$ maps the generalized eigenspace of $A$ at $\mu$ into the generalized eigenspace at
$\omega\mu$ (see Lemma~\ref{decomp}). The restriction $X$ of such an
element to the first space satisfies $N'X=\omega XN$, where $N$ and $N'$
are the restrictions of $A-\mu I$ and $A-\omega\mu I$ to the two spaces.
For the space of the linear maps $X$ with $N'X=\omega XN$,
Lemma~\ref{intertw} determines the intersection of the kernels, the sum
of the images, and the pairs of maps $Y$, $Y'$ with $XY=Y'X$ for all $X$
in this space. The necessity of the conditions of Theorem~\ref{mainA} comes from
the space $\mathcal J(A)$ of the matrices $Z$ with $ZX=XZ=O$ for all
$X\in\cC_\omega(A)$. This space lies in every double generalized
centralizer, and it consists of polynomials in $A$ only under these
conditions (see Corollary~\ref{annih} and Lemma~\ref{Jpoly}). For their sufficiency we combine Thompson's
theorem with a count of dimensions. Theorem~\ref{twins} rests on
Proposition~\ref{Dprop}, which applies Theorem~\ref{structure} to
$\omega$ and to $\omega^{-1}$ and describes the intersection
$\mathcal D(A)$ of $\cC_{\omega^{-1}}(\cC_\omega(A))$ and
$\cC_\omega(\cC_{\omega^{-1}}(A))$. This space contains all twins of $A$. For Theorem~\ref{twinsdesc},
Lemma~\ref{Dpoly} describes the polynomials in $A$ that lie in this
space, and for a polynomial $f$ such that $f(A)$ is a twin of $A$, maps of
rank one in $\cC_\omega(f(A))$ show that $f$ has the other two
properties. The proof of
Theorem~\ref{annthm} uses the fact that $\Ann_k(B)$ contains every matrix
$Z$ with $Z^2=O$ if $k\ge3$, or if $k=2$ and the characteristic of $\F$
is $2$. It also uses the fact that $\Ann_k(B)$ contains a matrix that is
diagonalizable over $\Fbar$ only if this matrix commutes with $B$.

The paper is organized as follows. Section~\ref{prelim} collects
notation and elementary lemmas, and it states known results such as a
criterion of Thompson for polynomial equivalence. We prove
Theorem~\ref{family} in Section~\ref{famsec}. Section~\ref{simsec} treats
the matrices $A$ whose minimal polynomial does not change when $A$ is
replaced by $\omega A$. In Section~\ref{dgcsec} we describe the double
generalized centralizers of all matrices and prove Theorems~\ref{mainA}
and~\ref{mainB}. Section~\ref{graphsec} translates
these results into statements about the $\omega$\nobreakdash-commuting
digraph and proves Theorems~\ref{twins} and~\ref{twinsdesc} on its
twins. Section~\ref{highsec} treats
higher commutators, and Section~\ref{concl} contains open questions.
Figure~\ref{depfig} shows the structure of the paper: an arrow from one
result to another means that the proof of the second uses the first.
The figure leaves out lemmas, examples, remarks, the results of
Section~\ref{prelim}, and Theorem~\ref{kurosaki}, which only a lemma uses.
It also leaves out Propositions~\ref{q24} and~\ref{anticomm} and
Corollaries~\ref{singular}, \ref{nbhd}, \ref{twinsq}, and~\ref{samekernel},
which no other proof uses.

\begin{figure}[H]
\centering
\begin{tikzpicture}[
  res/.style={draw,rounded corners=2pt,align=center,font=\footnotesize,
    minimum width=28.5mm,minimum height=9mm,inner xsep=1.5pt,inner ysep=2pt},
  main/.style={res,line width=1.1pt},
  arr/.style={-{Stealth[length=2mm]},line width=0.5pt},
  show background rectangle,inner frame sep=1.5mm,
  background rectangle/.style={draw=black!45,rounded corners=3pt},
  x=33mm,y=14.5mm]
\node[main] (family)    at (0,0)     {Theorem~\ref{family}\\ \emph{all powers of $\omega$}};
\node[res]  (oddcor)    at (1,0)     {Corollary~\ref{oddcor}\\ \emph{$\cC$ and $\CL$}};
\node[res]  (clcor)     at (2,0)     {Corollary~\ref{clcor}\\ \emph{clifforders}};
\node[res]  (thompson)  at (0,-1)    {Theorem~\ref{thompson}\\ \emph{Thompson}};
\node[res]  (thompcor)  at (1,-1)    {Corollary~\ref{thompcor}\\ \emph{condition (a)}};
\node[res]  (simeq)     at (2,-1)    {Corollary~\ref{simeq}\\ \emph{$\omega$-equivalence}};
\node[res]  (nilcor)    at (3,-1)    {Corollary~\ref{nilcor}\\ \emph{nilpotent $A$}};
\node[res]  (rankprop)  at (4,-1)    {Proposition~\ref{rankprop}\\ \emph{$\cC_\omega(A)=\cC_{\omega^{-1}}(B)$}};
\node[res]  (nonpow)    at (0,-2)    {Proposition~\ref{nonpow}\\ \emph{other $\varepsilon$}};
\node[main] (structure) at (1,-2)    {Theorem~\ref{structure}\\ \emph{structure}};
\node[res]  (simthm)    at (2,-2)    {Theorem~\ref{simthm}\\ \emph{$A$ similar to $\omega A$}};
\node[res]  (annih)     at (0,-3)    {Corollary~\ref{annih}\\ \emph{the space $\mathcal J(A)$}};
\node[main] (mainA)     at (1,-3)    {Theorem~\ref{mainA}\\ \emph{inclusion in $\F[A]$}};
\node[res]  (Dprop)     at (2,-3)    {Proposition~\ref{Dprop}\\ \emph{the space $\mathcal D(A)$}};
\node[main] (twins)     at (3,-3)    {Theorem~\ref{twins}\\ \emph{twins}};
\node[main,inner xsep=2.5pt] (twinsdesc) at (4,-3) {Theorem~\ref{twinsdesc}\\ \emph{all twins under (T)}};
\node[res]  (q22)       at (0,-4)    {Corollary~\ref{q22}\\ \emph{Question~\ref{Q1}}};
\node[res]  (q23)       at (1,-4)    {Corollary~\ref{q23}\\ \emph{Question~\ref{Q2}}};
\node[main] (mainB)     at (2,-4)    {Theorem~\ref{mainB}\\ \emph{$\cC_\omega(A)=\cC_\omega(B)$}};
\node[res]  (robeq)     at (0,-5.25) {Theorem~\ref{robeq}\\ \emph{Robinson (1971)}};
\node[res]  (kercent)   at (1,-5.25) {Corollary~\ref{kercent}\\ \emph{the set $\cC(\Ker\ad_A^k)$}};
\node[res]  (robinson)  at (2,-5.25) {Theorem~\ref{robinson}\\ \emph{Robinson (1965)}};
\node[res]  (annbyann)  at (3,-5.25) {Corollary~\ref{annbyann}\\ \emph{both equalities}};
\node[main,inner xsep=2.5pt] (annthm) at (4,-5.25) {Theorem~\ref{annthm}\\ \emph{$\Ker\ad_A^k=\Ann_k(B)$}};
\draw[arr] (family) -- (oddcor);
\draw[arr] (oddcor) -- (clcor);
\draw[arr] (thompson) -- (thompcor);
\draw[arr] (thompcor) -- (simeq);
\draw[arr] (simthm) -- (simeq);
\draw[arr] (simeq) -- (clcor);
\draw[arr] (simeq) -- (nilcor);
\draw[arr] (nilcor) -- (rankprop);
\draw[arr] (thompcor) -- (structure);
\draw[arr] (simthm) -- (structure);
\draw[arr] (structure) -- (mainA);
\draw[arr] (annih) -- (mainA);
\draw[arr] (mainA) -- (q22);
\draw[arr] (mainA) -- (q23);
\draw[arr] (mainA) -- (mainB);
\draw[arr] (structure) -- (nonpow);
\draw[arr] (annih) -- (nonpow);
\draw[arr] (structure) -- (Dprop);
\draw[arr] (Dprop) -- (twins);
\draw[arr] (twins) -- (twinsdesc);
\draw[arr] (robeq) -- (kercent);
\draw[arr] (robinson) -- (annbyann);
\draw[arr] (annthm) -- (annbyann);
\end{tikzpicture}
\caption{The structure of the paper. The main theorems are drawn with
thick frames.}
\label{depfig}
\end{figure}

\section{Preliminaries}
\label{prelim}

Throughout the paper, $\F$ is an arbitrary field, $\Fbar$ is an algebraic
closure of $\F$, $\F^\times=\F\setminus\{0\}$,
$\Fbar^\times=\Fbar\setminus\{0\}$, and $\F[x]$ is the ring of
polynomials in a variable $x$ with coefficients in $\F$. For $f\in\F[x]$
we write $\deg f$ for the degree of $f$, and $f'$ and $f''$ for its first
and second derivatives. The order of an element
$\omega\in\F^\times$ is the least integer $q\ge1$ with $\omega^q=1$, and it
is infinite if no such integer exists. An element of finite order $q$ is
a root of unity of order $q$; in this case $q$ is not divisible by the
characteristic of $\F$, because the polynomial $x^q-1$ has the $q$
distinct roots $1,\omega,\ldots,\omega^{q-1}$ in $\F$. The powers of
$\omega$ are the elements $\omega^i$ with integers $i$. We write
$M_{m\times n}(\F)$ for the space of $m\times n$ matrices over $\F$,
$M_n(\F)$ for $M_{n\times n}(\F)$, $\F^n$ and $\Fbar^n$ for the spaces of
column vectors with $n$ entries in $\F$ and in $\Fbar$, and $\GL_n(\F)$
for the group of invertible matrices in $M_n(\F)$. The identity matrix is $I_n$ or $I$, the
zero matrix of any size is $O$, the matrix with entry $1$ in position
$(i,j)$ and entries $0$ elsewhere is $E_{ij}$, and the transpose of $A$ is
$A^\top$. For square matrices $A_1,\ldots,A_s$ we write
$A_1\oplus\cdots\oplus A_s$ or $\diag(A_1,\ldots,A_s)$ for the block
diagonal matrix with these diagonal blocks, where a scalar stands for a
$1\times1$ matrix. The Jordan block of size $n$
with eigenvalue $\lambda$ is
$J_n(\lambda)=\lambda I_n+\sum_{i=1}^{n-1}E_{i,i+1}$.

Let $A\in M_n(\F)$. The minimal and the characteristic polynomial of $A$
are $m_A$ and $p_A$, and $\F[A]=\{f(A):f\in\F[x]\}$ is the algebra of
polynomials in $A$; its dimension $\dim\F[A]$ is $\deg m_A$. The set of
the scalar multiples of $A$ is $\F A$, and for a vector $v$ we define
$\F v$ in the same way. A nilpotent matrix $A$
has nilpotency index $m$ if $A^m=O$ and $A^{m-1}\ne O$. The matrix $A$ is
similar over $\F$ to a direct sum of companion matrices of monic
polynomials $d_1,\ldots,d_s$ of positive degree such that $d_{i+1}$
divides $d_i$ for $1\le i<s$. These polynomials are determined by $A$ and
are called the invariant factors of $A$. They satisfy $d_1=m_A$ and
$d_1\cdots d_s=p_A$, and two matrices in $M_n(\F)$ are similar over $\F$
if and only if they have the same invariant
factors~\cite[Chapter~III]{Jacobson1953}. For $\mu\in\Fbar$ the
generalized eigenspace of $A$ at $\mu$ is the kernel of $(A-\mu I)^n$ in
$\Fbar^n$. The Jordan blocks of $A$ at $\mu$ are the Jordan blocks with
eigenvalue $\mu$ in the Jordan form of $A$ over $\Fbar$. The size of the
largest of them is the multiplicity of $\mu$ as a root of $m_A$.

\begin{definition}
\label{centdef}
The centralizer of a set $\cS\subseteq M_n(\F)$ is
$\cC(\cS)=\{Y\in M_n(\F):XY=YX\text{ for all }X\in\cS\}$, and the
centralizer of a matrix $A\in M_n(\F)$ is $\cC(A)=\cC(\{A\})$.
\end{definition}

\begin{theorem}[Double centralizer theorem
{\cite[p.~113, Corollary~1]{Jacobson1953}}]
\label{dct}
For every $A\in M_n(\F)$ we have $\cC(\cC(A))=\F[A]$.
\end{theorem}

Jacobson stated the theorem for linear transformations of a vector space
over an arbitrary field. Suprunenko and
Tyshkevich~\cite[p.~33, Theorem~7]{SuprunenkoTyshkevich1968} gave it in
the equivalent form that the center of the algebra $\cC(A)$ is $\F[A]$.

\begin{lemma}[Dolinar, Guterman, Kuzma, and Oblak
{\cite[Lemma~2.4]{DGKO2013}}]
\label{polyeq}
Let $A,B\in M_n(\F)$. Then $\cC(A)\subseteq\cC(B)$ if and only if
$B\in\F[A]$. Consequently, $\cC(A)=\cC(B)$ if and only if
$\F[A]=\F[B]$.
\end{lemma}

Dolinar, Guterman, Kuzma, and Oblak~\cite[p.~2906]{DGKO2013} called this
lemma well known and noted that it follows easily from Theorem~\ref{dct}.

\begin{definition}
\label{pedef}
Let $\mathcal A$ be an associative algebra with unit over $\F$, and let
$y,z\in\mathcal A$. The elements $y$ and $z$ are polynomially equivalent
if there are polynomials $f,g\in\F[x]$ with $y=f(z)$ and $z=g(y)$. For an
integer $q\ge2$ they are $q$\nobreakdash-polynomially equivalent if there
are such polynomials of the form
\[
 f(x)=\sum_{i=0}^{r}a_ix^{qi+1},~g(x)=\sum_{i=0}^{r}b_ix^{qi+1}
\]
with $r\ge0$ and $a_i,b_i\in\F$. If $q=2$, we say that $y$ and $z$ are
odd polynomially equivalent. The relation defined in this way is called
polynomial equivalence, $q$\nobreakdash-polynomial equivalence, and odd
polynomial equivalence, respectively.
\end{definition}

For $A\in M_n(\F)$ and integers $q\ge1$, $j\ge0$ we write
$A^j\F[A^q]=\{A^jh(A^q):h\in\F[x]\}$, which is the span of the powers
$A^{qi+j}$ with integers $i\ge0$. Thus, for $q\ge2$, matrices
$A,B\in M_n(\F)$ are $q$\nobreakdash-polynomially equivalent if and only
if $B\in A\F[A^q]$ and $A\in B\F[B^q]$. By Lemma~\ref{polyeq}, two matrices are polynomially
equivalent if and only if they have the same centralizer.

The proof of Theorem~\ref{twinsdesc} uses the following theorem of
Thompson.

\begin{theorem}[Thompson {\cite[Theorem~1 and Corollary]{Thompson1969}}]
\label{thompson69}
Let $A\in M_n(\F)$. For $f\in\F[x]$, the matrix $A$ is a polynomial in
$f(A)$ if and only if $f$ is injective on the set of the roots of $m_A$ in $\Fbar$
and $f'(\alpha)\ne0$ for every multiple root $\alpha$ of $m_A$.
Consequently, a matrix $B\in M_n(\F)$ satisfies $\cC(A)=\cC(B)$ if and
only if $B=f(A)$ for a polynomial $f\in\F[x]$ with these two properties.
\end{theorem}

The next proposition gives a sufficient condition for two block diagonal
matrices to be polynomially equivalent.

\begin{proposition}
\label{blocks}
Let $A_i,B_i\in M_{n_i}(\F)$ be polynomially equivalent for
$1\le i\le s$. Suppose that the polynomials $p_{A_1},\ldots,p_{A_s}$ are
pairwise coprime and that $p_{B_1},\ldots,p_{B_s}$ are pairwise coprime.
Then $A=A_1\oplus\cdots\oplus A_s$ and $B=B_1\oplus\cdots\oplus B_s$ are
polynomially equivalent.
\end{proposition}

\begin{proof}
Choose $f_i\in\F[x]$ with $B_i=f_i(A_i)$. By the Chinese remainder
theorem there is a polynomial $h$ with $h\equiv f_i\pmod{p_{A_i}}$ for
$1\le i\le s$. Since $p_{A_i}(A_i)=O$, we obtain $h(A_i)=f_i(A_i)=B_i$
for all $i$, and hence $h(A)=B$. In the same way $A$ is a polynomial in
$B$.
\end{proof}

\begin{remark}
\label{blocksrem}
For $s\ge3$ it is not enough to assume that the greatest common divisor of
$p_{A_1},\ldots,p_{A_s}$ is $1$ and that the greatest common divisor of
$p_{B_1},\ldots,p_{B_s}$ is $1$. For example, the $1\times1$ matrices
$A_1=A_2=(0)$ and $A_3=(1)$, and $B_1=(0)$, $B_2=(1)$, and $B_3=(2)$
over $\mathbb Q$ satisfy these two conditions, and any two $1\times1$ matrices
are polynomially equivalent. However, $E_{12}$ commutes with
$A=\diag(0,0,1)$ and does not commute with $B=\diag(0,1,2)$, and therefore
$A$ and $B$ are not polynomially equivalent by Lemma~\ref{polyeq}.
\end{remark}

\begin{definition}
\label{gcdef}
Let $\omega\in\F$. The generalized centralizer of $A\in M_n(\F)$ with
respect to $\omega$ is
\[
 \cC_\omega(A)=\{X\in M_n(\F):AX=\omega XA\}.
\]
Thus $\cC_1(A)=\cC(A)$. The set $\CL(A)=\cC_{-1}(A)$ of the matrices that
anticommute with $A$ is the clifforder of $A$. For a set
$\cS\subseteq M_n(\F)$ and $\varepsilon\in\F$ we put
\[
 \cC_\varepsilon(\cS)=\{Y\in M_n(\F):XY=\varepsilon YX\text{ for all }
 X\in\cS\}.
\]
The sets $\cC_\varepsilon(\cC_\omega(A))$ with $\varepsilon\in\F^\times$
are the double generalized centralizers of $A$. Two matrices $A,B\in M_n(\F)$ are $\omega$\nobreakdash-equivalent if
$\cC_\omega(A)=\cC_\omega(B)$, and this relation is called
$\omega$\nobreakdash-equivalence.
\end{definition}

Dolinar, Guterman, Kuzma, and Markova~\cite{DGKM2018,DGKM2019} wrote
$\cC^\omega(A)$ for $\cC_\omega(A)$, and they used the name double
generalized centralizer for $\cC_\omega(\cC_\omega(A))$. We use the following properties several times.

\begin{lemma}
\label{basic}
Let $A,B\in M_n(\F)$, $P\in\GL_n(\F)$, and $\omega,\varepsilon\in\F$.
\begin{enumerate}
\item[(a)] The set $\cC_\omega(A)$ is a subspace of $M_n(\F)$, and
$\cC_\omega(P^{-1}AP)$ is the set of the matrices $P^{-1}XP$ with
$X\in\cC_\omega(A)$.
\item[(b)] If $\omega\ne0$, then $\cC_\omega(A)\subseteq\cC_\omega(B)$ if
and only if $B\in\cC_{\omega^{-1}}(\cC_\omega(A))$.
\item[(c)] If $\omega\ne0$ and $m\ge0$ is an integer with
$\omega^m\varepsilon=1$, then $A^m\in\cC_\varepsilon(\cC_\omega(A))$.
\end{enumerate}
\end{lemma}

\begin{proof}
Part (a) follows from the definition. For (b), the
inclusion $\cC_\omega(A)\subseteq\cC_\omega(B)$ means that
$BX=\omega XB$, that is, $XB=\omega^{-1}BX$, for all $X\in\cC_\omega(A)$.
For (c), let $X\in\cC_\omega(A)$. Then $A^mX=\omega^mXA^m$ by induction on
$m$, and therefore $XA^m=\omega^{-m}A^mX=\varepsilon A^mX$.
\end{proof}

For a root of unity $\omega$, $n\ge2$, and $\varepsilon\ne1$, part (c) is
contained in a proposition of
Dolinar, Guterman, Kuzma, and Markova~\cite[Proposition~9]{DGKM2019}.

The next lemma reduces many statements to algebraically closed fields.

\begin{lemma}
\label{scalars}
Let $\mathcal U$ and $\mathcal U'$ be subspaces of $M_n(\F)$, and let
$\operatorname{span}\{\mathcal U\}$ and $\operatorname{span}\{\mathcal U'\}$ be their spans
over $\Fbar$ in $M_n(\Fbar)$. Then
$\mathcal U=\operatorname{span}\{\mathcal U\}\cap M_n(\F)$,
$\operatorname{span}\{\mathcal U\cap \mathcal U'\}=\operatorname{span}\{\mathcal U\}\cap
\operatorname{span}\{\mathcal U'\}$, and $\mathcal U\subseteq \mathcal U'$ if and only if
$\operatorname{span}\{\mathcal U\}\subseteq\operatorname{span}\{\mathcal U'\}$. Moreover, let
$A\in M_n(\F)$, let
$\omega,\varepsilon\in\F$, and let $j\ge0$ and $q\ge1$ be integers. If
$\mathcal U$ is one of the spaces $\cC_\omega(A)$,
$\cC_\varepsilon(\cC_\omega(A))$, $\F[A]$, and $A^j\F[A^q]$, then $\operatorname{span}\{\mathcal U\}$ is the subspace of
$M_n(\Fbar)$ defined in the same way over $\Fbar$.
\end{lemma}

\begin{proof}
Let $u_1,\ldots,u_r$ be a basis of $\mathcal U$. If a matrix $Z\in M_n(\F)$ lies
in $\operatorname{span}\{\mathcal U\}$, then the system $\sum_{i=1}^{r}c_iu_i=Z$ in
the unknowns $c_i$ has a solution in $\Fbar$ and therefore a solution in
$\F$. Thus $\mathcal U=\operatorname{span}\{\mathcal U\}\cap M_n(\F)$. If
$\operatorname{span}\{\mathcal U\}\subseteq\operatorname{span}\{\mathcal U'\}$, then
$\mathcal U\subseteq\operatorname{span}\{\mathcal U'\}\cap M_n(\F)=\mathcal U'$, and the converse follows from the definition of the span. A family of matrices with entries in $\F$ has the same rank over
$\F$ and over $\Fbar$. Applied to bases of $\mathcal U$, $\mathcal U'$, and $\mathcal U\cap \mathcal U'$ and
to their union, this shows that $\operatorname{span}\{\mathcal U\}$,
$\operatorname{span}\{\mathcal U'\}$, $\operatorname{span}\{\mathcal U\cap \mathcal U'\}$, and
$\operatorname{span}\{\mathcal U\}+\operatorname{span}\{\mathcal U'\}$ have the dimensions
of $\mathcal U$, $\mathcal U'$, $\mathcal U\cap \mathcal U'$, and $\mathcal U+\mathcal U'$. Hence
$\operatorname{span}\{\mathcal U\}\cap\operatorname{span}\{\mathcal U'\}$ and
$\operatorname{span}\{\mathcal U\cap \mathcal U'\}$ have the same dimension. Since the second space is contained in the
first, the two spaces are equal.

Consider a system of homogeneous linear equations with coefficients in
$\F$. The reduced row echelon form of its coefficient matrix has entries
in $\F$. It gives vectors with entries in $\F$ that form a basis of the
solution space over $\F$ and also a basis of the solution space over
$\Fbar$. Hence the span
of the solution space over $\F$ is the solution space over $\Fbar$. This
applies to $\cC_\omega(A)$. If $X_1,\ldots,X_k$ is a basis of
$\cC_\omega(A)$ over $\F$, then it is a basis of the corresponding space
over $\Fbar$, and the equations $X_iY=\varepsilon YX_i$ with $1\le i\le k$ define
$\cC_\varepsilon(\cC_\omega(A))$ over $\F$ and the corresponding space over
$\Fbar$. Finally, the powers $A^i$ with $i\ge0$ span $\F[A]$
over $\F$ and $\Fbar[A]$ over $\Fbar$, and the powers $A^{j+qi}$ with
$i\ge0$ span $A^j\F[A^q]$ over $\F$ and $A^j\Fbar[A^q]$ over $\Fbar$.
\end{proof}

\begin{proposition}
\label{kron}
Let $A,B\in M_n(\F)$ and $\omega\in\F$, and let
$\operatorname{vec}(X)\in\F^{n^2}$ be the column vector obtained by
writing the columns of $X\in M_n(\F)$ one below the other. Then
$X\in\cC_\omega(A)$ if and only if
$(I_n\otimes A-\omega A^\top\otimes I_n)\operatorname{vec}(X)=0$, where
$\otimes$ is the Kronecker product. Moreover,
$\cC_\omega(A)=\cC_\omega(B)$ if and only if there is a matrix
$P\in\GL_{n^2}(\F)$ with
\[
 I_n\otimes A-\omega A^\top\otimes I_n
 =P(I_n\otimes B-\omega B^\top\otimes I_n).
\]
\end{proposition}

\begin{proof}
The first assertion follows from
$\operatorname{vec}(AX)=(I_n\otimes A)\operatorname{vec}(X)$ and
$\operatorname{vec}(XA)=(A^\top\otimes I_n)\operatorname{vec}(X)$. For the
second, let $M,M'\in M_{n^2}(\F)$. If $M=PM'$ with $P$ invertible, then
the systems $Mv=0$ and $M'v=0$ in the unknown vector $v\in\F^{n^2}$ have
the same solutions. Conversely, if the two
systems have the same solutions, then $M$ and $M'$ have the same row
space, which is the space of the linear forms that are zero on the common
solution set. Hence $M$ and $M'$ have the same reduced row echelon form
$R$, so that $M=P_1R$ and $M'=P_2R$ with invertible $P_1,P_2$, and
$M=P_1P_2^{-1}M'$.
\end{proof}

Alahmadi, Glasby, Praeger, Sol\'e, and
Yildiz~\cite[Proposition~1.1]{AGPSY2017} gave the first assertion with
$\operatorname{vec}(X)$ defined in the same way. If the entries of $X$ are
ordered by rows, the matrix of the map $X\mapsto AX-\omega XA$ is
$A\otimes I_n-\omega I_n\otimes A^\top$. For $\omega=1$ this matrix was
given by Khan~\cite[p.~315]{Khan1964}, and Roth~\cite[p.~484]{Roth1936}
used its transpose for row vectors. For every $\omega$, Alahmadi,
Glasby, and Praeger~\cite[Lemma~3.2]{AGP2017} used the transpose
$A^\top\otimes I_n-\omega I_n\otimes A$ in the same way.

The proofs of Theorems~\ref{mainB} and~\ref{twins} use the following
lemma.

\begin{lemma}
\label{generic}
Let $A,Y\in M_n(\F)$, and let $\Omega$ be a finite subset of $\F$ with
$|\Omega|$ elements. If
$\F$ has more than $|\Omega|n^2+1$ elements, then there is a nonzero
$t\in\F$ with $\dim\cC_\varepsilon(A+tY)\le\dim\cC_\varepsilon(A)$ for
all $\varepsilon\in\Omega$.
\end{lemma}

\begin{proof}
For $\varepsilon\in\Omega$ let $r_\varepsilon$ be the rank of the linear
map $X\mapsto AX-\varepsilon XA$ of $M_n(\F)$, whose kernel is
$\cC_\varepsilon(A)$. In a fixed basis of $M_n(\F)$, the matrix of this
map has a nonzero minor of size $r_\varepsilon$. For $t\in\F$, the entries
of the matrix of the map $X\mapsto(A+tY)X-\varepsilon X(A+tY)$ in this
basis are polynomials of degree at most $1$ in $t$, so that the
corresponding minor of this matrix is a polynomial in $t$ of degree at
most $r_\varepsilon\le n^2$ that is nonzero at $t=0$. These $|\Omega|$
polynomials have at most $|\Omega|n^2$ roots together, and $\F$ has more
than $|\Omega|n^2$ nonzero elements. Hence there is $t\ne0$ at which all
of them are nonzero. For this $t$ and every $\varepsilon\in\Omega$ the map
$X\mapsto(A+tY)X-\varepsilon X(A+tY)$ has rank at least $r_\varepsilon$,
so that its kernel $\cC_\varepsilon(A+tY)$ has dimension at most
$n^2-r_\varepsilon=\dim\cC_\varepsilon(A)$.
\end{proof}

\section{The generalized centralizers for the powers of a root of unity}
\label{famsec}

Lemma~\ref{polyeq} describes the matrices $B$ with
$\cC(A)\subseteq\cC(B)$. Theorem~\ref{family} describes in the same way
the matrices $B$ with $\cC_{\omega^j}(A)\subseteq\cC_{\omega^j}(B)$ for
the integers $j$ with $0\le j<q$, where $\omega$ is a root of unity of
order $q$. The proof uses a block diagonal matrix, which for
$q=2$ is the matrix $A\oplus(-A)$.

\begin{lemma}
\label{cover}
Let $\omega\in\F^\times$ be a root of unity of order $q$, let
$A\in M_n(\F)$, let $l$ be an integer, and let
$\widetilde A=A\oplus\omega A\oplus\cdots\oplus\omega^{q-1}A\in
M_{qn}(\F)$. A matrix $X=(X_{ij})_{i,j=0}^{q-1}$ with blocks
$X_{ij}\in M_n(\F)$ lies in $\cC_{\omega^l}(\widetilde A)$ if and only if
$X_{ij}\in\cC_{\omega^{j-i+l}}(A)$ for all $i$ and $j$. In particular,
for $q=2$ and $\widetilde A=A\oplus(-A)$,
\[
 \cC(\widetilde A)=\begin{bmatrix}\cC(A)&\CL(A)\\ \CL(A)&\cC(A)
 \end{bmatrix},~
 \CL(\widetilde A)=\begin{bmatrix}\CL(A)&\cC(A)\\ \cC(A)&\CL(A)
 \end{bmatrix},
\]
where a block matrix of sets stands for the set of the block matrices
whose blocks lie in the sets in the same positions.
\end{lemma}

\begin{proof}
The block in position $(i,j)$ of $\widetilde AX$ is $\omega^iAX_{ij}$, and
the block in position $(i,j)$ of $\omega^lX\widetilde A$ is
$\omega^{l+j}X_{ij}A$.
\end{proof}

\begin{proof}[Proof of Theorem~\ref{family}]
Suppose that $\cC_{\omega^j}(A)\subseteq\cC_{\omega^j}(B)$ for the
integers $j$ with $0\le j<q$, and let $\widetilde A$ and $\widetilde B$ be the matrices of
Lemma~\ref{cover} for $A$ and $B$. By this lemma with $l=0$ we have
$\cC(\widetilde A)\subseteq\cC(\widetilde B)$. Lemma~\ref{polyeq} gives a polynomial
$g(x)=\sum_{m=0}^{N}g_mx^m$ with $\widetilde B=g(\widetilde A)$, that is,
$\omega^iB=g(\omega^iA)$ for $0\le i<q$. We multiply this equation by
$\omega^{-i}$ and add the results:
\[
 qB=\sum_{i=0}^{q-1}\omega^{-i}g(\omega^iA)
 =\sum_{m=0}^{N}g_m\Bigl(\sum_{i=0}^{q-1}\omega^{i(m-1)}\Bigr)A^m .
\]
If $q$ divides $m-1$, the inner sum is $q$. Otherwise
$\omega^{m-1}\ne1$ is a root of $(x^q-1)/(x-1)=\sum_{i=0}^{q-1}x^i$, and
the inner sum is $0$. Since $q$ is invertible in $\F$, we obtain
$B=\sum g_mA^m$, where the sum is over the indices $m\le N$ with
$m\equiv1\pmod q$. Hence $B\in A\F[A^q]$.

Conversely, let $B=\sum_{i=0}^{r}c_iA^{qi+1}$, let $j$ be an integer with
$0\le j<q$, and let $X\in\cC_{\omega^j}(A)$. Then
$A^{qi+1}X=\omega^{j(qi+1)}XA^{qi+1}=\omega^jXA^{qi+1}$ for every $i$,
and hence $BX=\omega^jXB$. The last assertion follows by exchanging $A$
and $B$.
\end{proof}

\begin{corollary}
\label{oddcor}
Suppose that the characteristic of $\F$ is not $2$, and let
$A,B\in M_n(\F)$. Then $\cC(A)\subseteq\cC(B)$ and
$\CL(A)\subseteq\CL(B)$ if and only if $B$ is an odd polynomial in $A$,
that is, $B\in A\F[A^2]$. Consequently, the equalities $\cC(A)=\cC(B)$ and
$\CL(A)=\CL(B)$ hold if and only if $A$ and $B$ are odd polynomially
equivalent.
\end{corollary}

\begin{proof}
This is Theorem~\ref{family} for $\omega=-1$ and $q=2$.
\end{proof}

The next example applies Corollary~\ref{oddcor} to matrices $A$ that are
not similar to $-A$.

\begin{example}
\label{fiveex}
Let $\F=\mathbb Q$,
\[
 A=\begin{bmatrix}1&1&&&\\&1&&&\\&&-1&-1&\\&&&-1&-1\\&&&&-1
 \end{bmatrix},~
 B=\begin{bmatrix}2&1&&&\\&2&&&\\&&-2&-1&\\&&&-2&-1\\&&&&-2
 \end{bmatrix}.
\]
The polynomials
\[
 f_0(x)=\frac{1}{128}(3x^4+8x^3-24x^2+32x+48),~
 g_0(x)=\frac18(-3x^4-4x^3+6x^2+20x-3)
\]
satisfy $A=f_0(B)$ and $B=g_0(A)$, so that $A$ and $B$ are polynomially
equivalent. We claim that $A$ and $B$ are odd polynomially equivalent. Put
$N_2=J_2(0)$
and $N_3=J_3(0)$, so that $A=(I_2+N_2)\oplus(-I_3-N_3)$ and
$B=(2I_2+N_2)\oplus(-2I_3-N_3)$. The odd polynomials
\[
 f(x)=\frac{1}{256}(16x+40x^3-3x^5),~
 g(x)=\frac18(23x-10x^3+3x^5)
\]
satisfy $f(2)=1$, $f'(2)=1$, and $f''(2)=0$, and $g(1)=2$, $g'(1)=1$,
and $g''(1)=0$. Since $N_2^2=O$ and $N_3^3=O$, Taylor's formula gives
$f(2I_2+N_2)=I_2+N_2$ and $f(2I_3+N_3)=I_3+N_3$, and in the same way
$g(I_2+N_2)=2I_2+N_2$ and $g(I_3+N_3)=2I_3+N_3$. Because $f$ and $g$ are
odd, it follows that $A=f(B)$ and $B=g(A)$. Thus $A$ and $B$ are odd
polynomially equivalent, and $\cC(A)=\cC(B)$ and $\CL(A)=\CL(B)$ by
Corollary~\ref{oddcor}. The matrix $A$ has one Jordan block of size $2$
at $1$ and one Jordan block of size $3$ at $-1$. Hence $A$ is not similar
to $-A$, and $\CL(A)$ contains no invertible matrix (see
Theorem~\ref{simthm}). In the same way, the matrices
$A'=J_2(1)\oplus(-1)\oplus(-1)$ and $B'=J_2(2)\oplus(-2)\oplus(-2)$
satisfy $A'=\frac14B'+\frac1{16}B'^3$ and $B'=\frac52A'-\frac12A'^3$. Thus
$\cC(A')=\cC(B')$ and $\CL(A')=\CL(B')$, although $A'$ is not similar to
$-A'$.
\end{example}

\begin{remark}
\label{allj}
In Theorem~\ref{family} it suffices to assume the inclusions for the
integers $j$ with $0\le j\le q/2$. Indeed, the inclusion for $j=0$ gives
$B\in\F[A]$ by Lemma~\ref{polyeq}. Choose $S\in\GL_n(\F)$ with
$SAS^{-1}=A^\top$; such a matrix exists because $A$ and $A^\top$ have the
same invariant factors. Then $SBS^{-1}=B^\top$. For $\lambda\in\F^\times$
a matrix $X$ satisfies $AX=\lambda XA$ if and only if $Y=S^{-1}X^\top S$
satisfies $AY=\lambda^{-1}YA$, and the same holds for $B$ in place of
$A$. Hence the map $X\mapsto S^{-1}X^\top S$ sends $\cC_\lambda(A)$ onto
$\cC_{\lambda^{-1}}(A)$ and $\cC_\lambda(B)$ onto $\cC_{\lambda^{-1}}(B)$,
and the inclusions for $j$ and for $q-j$ are equivalent. The same holds
for the equalities.

None of the inclusions with $0\le j\le q/2$ can be omitted. For $j=0$ let
$A=J_2(1)$ and $B=E_{21}$. If $\lambda\ne1$, the equation $AX=\lambda XA$ reads $(1-\lambda)X=\lambda XN-NX$ with $N=E_{12}$. Since the map
$X\mapsto\lambda XN-NX$ is nilpotent, only $X=O$ satisfies this equation.
Thus the inclusions hold for $0<j<q$, while $A\in\cC(A)$ does not commute
with $B$, and $B\notin A\F[A^q]$, because $B$ is not upper triangular.
For an integer $j_0$ with $0<j_0\le q/2$ let $A=\diag(1,\omega^{j_0})$
and $B=I_2$. Then
$\cC_{\omega^j}(B)=\{O\}$ for $0<j<q$, and $\cC_{\omega^j}(A)\ne\{O\}$
exactly for the $j$ with $j\equiv0$ or $j\equiv\pm j_0\pmod q$. Hence the
inclusions hold exactly for the $j$ with $j\not\equiv\pm j_0\pmod q$,
while $B\notin A\F[A^q]=\F A$.
\end{remark}

\section{\texorpdfstring{Matrices whose minimal polynomial is balanced with respect to $\omega$}{Matrices whose minimal polynomial is balanced with respect to ω}}
\label{simsec}

For $\omega\in\F^\times$ this section treats the matrices $A\in M_n(\F)$
for which $\cC_\omega(A)$ contains an invertible matrix, and more generally
the matrices $A$ whose minimal polynomial is not changed when $A$ is
replaced by $\omega A$. For these matrices we determine the sets
$\cC_\varepsilon(\cC_\omega(A))$ with $\varepsilon\in\F^\times$, and for
the matrices that are similar to $\omega A$ also the matrices $B$ with
$\cC_\omega(A)=\cC_\omega(B)$.

\begin{definition}
\label{baldef}
Let $\omega\in\F^\times$. A monic polynomial $d\in\F[x]$ is balanced with
respect to $\omega$ if $d(\omega x)=\omega^{\deg d}d(x)$. A polynomial
that is balanced with respect to $-1$ is called balanced, and a matrix
$A\in M_n(\F)$ is called balanced if all its invariant factors are
balanced.
\end{definition}

For a monic polynomial $d\in\F[x]$ and $\omega\in\F^\times$ we put
$d^\omega(x)=\omega^{\deg d}d(\omega^{-1}x)$.

\begin{lemma}
\label{balroots}
Let $\omega\in\F^\times$ and let $d\in\F[x]$ be monic.
\begin{enumerate}
\item[(a)] The polynomial $d^\omega$ is monic, and for every $\mu\in\Fbar$
the multiplicity of $\omega\mu$ as a root of $d^\omega$ is the
multiplicity of $\mu$ as a root of $d$. The polynomial $d$ is balanced
with respect to $\omega$ if and only if $d^\omega=d$, that is, if and only
if for every $\mu\in\Fbar$ the elements $\mu$ and $\omega\mu$ have the same
multiplicity as roots of $d$.
\item[(b)] If $d$ is the minimal polynomial of $A\in M_n(\F)$, then
$d^\omega$ is the minimal polynomial of $\omega A$.
\end{enumerate}
\end{lemma}

\begin{proof}
(a) If $d(x)=\prod_{i=1}^{m}(x-\mu_i)$ with $\mu_i\in\Fbar$, then
$d^\omega(x)=\prod_{i=1}^{m}(x-\omega\mu_i)$. The substitution of $\omega x$ for $x$ shows that the equation
$d^\omega=d$ is equivalent to
$d(\omega x)=\omega^{\deg d}d(x)$.

(b) A polynomial $g$ satisfies $g(\omega A)=O$ if and only if $g(\omega x)$
is a multiple of $d$, that is, if and only if $g$ is a multiple of
$d(\omega^{-1}x)$. The monic polynomial of least degree with this property
is $d^\omega$.
\end{proof}

\begin{theorem}
\label{simthm}
Let $\omega\in\F^\times$ and $A\in M_n(\F)$. The following statements are
equivalent.
\begin{enumerate}
\item[(i)] The generalized centralizer $\cC_\omega(A)$ contains an
invertible matrix.
\item[(ii)] The matrix $A$ is similar to $\omega A$ over $\F$.
\item[(iii)] Every invariant factor of $A$ is balanced with respect to
$\omega$.
\end{enumerate}
In particular, if $A$ is nilpotent, then $A$ is similar to $\omega A$
over $\F$.
\end{theorem}

\begin{proof}
A matrix $P\in\GL_n(\F)$ satisfies $AP=\omega PA$ if and only if
$P^{-1}AP=\omega A$. Hence (i) and (ii) are equivalent.

Let $d$ be a monic polynomial of degree $m$, and let $C$ be its companion
matrix. There is a vector $v$ such that $v,Cv,\ldots,C^{m-1}v$ is a
basis. Then $v,\omega Cv,\ldots,(\omega C)^{m-1}v$ is a basis, so that the
minimal polynomial of $\omega C$ has degree $m$. By
Lemma~\ref{balroots}(b), this minimal polynomial is $d^{\omega}$, and
$\omega C$ is similar to the companion matrix of $d^{\omega}$. Let
$d_1,\ldots,d_s$ be the invariant factors of $A$. Then $\omega A$ is
similar to the direct sum of the companion matrices of
$d_1^{\omega},\ldots,d_s^{\omega}$, and $d_{i+1}^{\omega}$ divides
$d_i^{\omega}$. Hence these are the invariant factors of $\omega A$.
Therefore $A$ is similar to $\omega A$ if and only if $d_i^{\omega}=d_i$
for all $i$, which means that every $d_i$ is balanced with respect to
$\omega$. Thus (ii) and (iii) are equivalent. The invariant factors of a
nilpotent matrix are powers of $x$, which are balanced with respect to
every $\omega\in\F^\times$, and this gives the last assertion.
\end{proof}

The elementary divisors of a matrix are the powers of monic irreducible
polynomials that occur in the factorizations of its invariant factors.
A power that occurs in several invariant factors is counted once for each
of them. Kurosaki~\cite[pp.~25--27]{Kurosaki1941} stated that the
elementary divisors of $\omega A$ are the polynomials $d^\omega$ for the
elementary divisors $d$ of $A$. This statement implies the equivalence of
(ii) and (iii). He also noted the equivalence of (i) and
(ii) for invertible matrices over an arbitrary field, and he proved the
last assertion of Theorem~\ref{simthm} by an explicit invertible element
of $\cC_\omega(A)$ for $A$ in Jordan form. For $\omega=-1$ the
equivalence of (i) and (iii) says that $\CL(A)$ contains an invertible
matrix if and only if $A$ is balanced.

Corollary~\ref{thompcor} below describes the double generalized
centralizers of the matrices whose minimal polynomial is balanced with
respect to $\omega$. It
rests on the following theorem of Thompson.

\begin{theorem}[Thompson {\cite[Theorem~1]{Thompson1965}}]
\label{thompson}
Let $A_1,C_1\in M_m(\F)$ and $A_2,C_2\in M_n(\F)$, and suppose that $A_1$
and $A_2$ have the same minimal polynomial. If every matrix
$X\in M_{m\times n}(\F)$ with $A_1X=XA_2$ satisfies $C_1X=XC_2$, then
there is a polynomial $\varphi\in\F[x]$ with $C_1=\varphi(A_1)$ and
$C_2=\varphi(A_2)$.
\end{theorem}

\begin{corollary}
\label{thompcor}
Let $\omega,\varepsilon\in\F^\times$, and let $A\in M_n(\F)$ be a matrix
whose minimal polynomial is balanced with respect to $\omega$. Then
$\cC_\varepsilon(\cC_\omega(A))$ is the span over $\F$ of the powers $A^m$
with integers $m\ge0$ and $\omega^m\varepsilon=1$. In particular, if
$\omega$ has finite order $q$ and $j$ is an integer with $0\le j<q$, then
$\cC_{\omega^{-j}}(\cC_\omega(A))=A^j\F[A^q]$, and
$\cC_\varepsilon(\cC_\omega(A))=\{O\}$ if $\varepsilon$ is not a power of
$\omega$.
\end{corollary}

\begin{proof}
Let $S$ be the span of the powers $A^m$ with $m\ge0$ and
$\omega^m\varepsilon=1$. Then $S\subseteq\cC_\varepsilon(\cC_\omega(A))$ by
Lemma~\ref{basic}(c). Let $Y\in\cC_\varepsilon(\cC_\omega(A))$ and
$k=\deg m_A$. By Lemma~\ref{balroots}, the matrices $A$ and $\omega A$
have the same minimal polynomial $m_A$. A matrix $X$
satisfies $AX=X(\omega A)$ if and only if $X\in\cC_\omega(A)$, and then
$(\varepsilon Y)X=XY$. By Theorem~\ref{thompson} with $A_1=A$,
$A_2=\omega A$, $C_1=\varepsilon Y$, and $C_2=Y$, there is a polynomial
$\varphi$ with $\varepsilon Y=\varphi(A)$ and $Y=\varphi(\omega A)$. We
may replace $\varphi$ by its remainder modulo $m_A$, because $m_A(A)=O$
and $m_A(\omega A)=O$. Write $\varphi(x)=\sum_{m=0}^{k-1}f_mx^m$. Then
\[
 \sum_{m=0}^{k-1}f_mA^m=\varepsilon Y=\varepsilon\varphi(\omega A)
 =\sum_{m=0}^{k-1}\varepsilon\omega^mf_mA^m .
\]
Since $I,A,\ldots,A^{k-1}$ are linearly independent, $f_m=0$ unless
$\omega^m\varepsilon=1$, and $Y=\sum_{m=0}^{k-1}f_m\omega^mA^m$ lies in
$S$. If
$\omega$ has finite order $q$ and $\varepsilon=\omega^{-j}$, then
$\omega^m\varepsilon=1$ means $m\equiv j\pmod q$, and $S=A^j\F[A^q]$.
\end{proof}

If $A$ is similar to $\omega A$, then $A$ and $\omega A$ have the same
minimal polynomial, and Corollary~\ref{thompcor} applies. For a
nilpotent matrix $A$, $\omega\notin\{0,1\}$, and
$\varepsilon\in\{\omega,\omega^{-1}\}$, the corollary is a theorem of
Dolinar, Guterman, Kuzma, and Markova~\cite[Theorem~11]{DGKM2019} over
algebraically closed fields, and for
$A=I_{n_1}\oplus\omega I_{n_2}\oplus\cdots\oplus\omega^{q-1}I_{n_q}$ with
$n_1,\ldots,n_q\ge1$, $\omega$ of order $q$, and $\varepsilon=\omega^{-1}$
it is their Theorem~12. With $\omega$ of order $3$ and $A=\diag(1,\omega,\omega^2)$, the
corollary gives $\cC_\omega(\cC_\omega(A))=A^2\F[A^3]$, which is the line
spanned by $\diag(\omega^2,\omega,1)$. If $\omega\in\F^\times$ satisfies
$\omega^2\ne1$ and $\omega^3\ne1$, then $A=\diag(1,\omega,\omega^2)$
satisfies $\cC_\omega(A)=\F E_{21}+\F E_{32}$, and comparing
entries shows that $\cC_\omega(\cC_\omega(A))$ is spanned by
$\diag(\omega^2,\omega,1)$ and $E_{31}$. Thus, for $\omega\ne-1$, the
matrices displayed in the last two items of Example~13 of Dolinar
et~al.~\cite{DGKM2019} are exchanged.

\begin{corollary}
\label{simeq}
Let $\omega\in\F\setminus\{0,1\}$ have order $q$, let $A\in M_n(\F)$ be
similar to $\omega A$, and let $B\in M_n(\F)$.
\begin{enumerate}
\item[(a)] If $q$ is finite, then
$\cC_\omega(A)\subseteq\cC_\omega(B)$ if and only if $B\in A\F[A^q]$, and
$\cC_\omega(A)=\cC_\omega(B)$ if and only if $A$ and $B$ are
$q$\nobreakdash-polynomially equivalent.
\item[(b)] If $q$ is infinite, then
$\cC_\omega(A)\subseteq\cC_\omega(B)$ if and only if $B=cA$ with
$c\in\F$, and $\cC_\omega(A)=\cC_\omega(B)$ if and only if $B=cA$ with
$c\in\F^\times$.
\end{enumerate}
\end{corollary}

\begin{proof}
By Lemma~\ref{basic}(b) and Corollary~\ref{thompcor} with
$\varepsilon=\omega^{-1}$, the inclusion
$\cC_\omega(A)\subseteq\cC_\omega(B)$ holds if and only if $B$ lies in the
span of the powers $A^m$ with $m\ge0$ and $\omega^{m-1}=1$. Since
$\omega\ne1$, this span is $A\F[A^q]$ if $q$ is finite and $\F A$ if $q$
is infinite. Suppose that $\cC_\omega(A)=\cC_\omega(B)$. Then
$\cC_\omega(B)$ contains an invertible matrix, so that $B$ is similar to
$\omega B$ by Theorem~\ref{simthm}, and the first part applies with $A$
and $B$ exchanged. Conversely, if $A\in B\F[B^q]$, then
$A\in\cC_{\omega^{-1}}(\cC_\omega(B))$ by Lemma~\ref{basic}(c), and hence
$\cC_\omega(B)\subseteq\cC_\omega(A)$ by Lemma~\ref{basic}(b). This
proves (a). In (b) we obtain $B=cA$ and $A=c'B$
with $c,c'\in\F$. If $A\ne O$, then $cc'=1$. If $A=O$, then $B=O$ and
$B=cA$ for every $c\in\F^\times$. Conversely,
$\cC_\omega(cA)=\cC_\omega(A)$ for $c\in\F^\times$.
\end{proof}

\begin{corollary}
\label{clcor}
Suppose that the characteristic of $\F$ is not $2$, and let
$A,B\in M_n(\F)$.
\begin{enumerate}
\item[(a)] If $A$ and $B$ are odd polynomially equivalent, then
$\CL(A)=\CL(B)$.
\item[(b)] If $\CL(A)=\CL(B)$ and $\CL(A)$ contains an invertible matrix,
then $A$ and $B$ are odd polynomially equivalent.
\item[(c)] If $A$ is balanced, then $\CL(A)=\CL(B)$ if and only if $A$ and
$B$ are odd polynomially equivalent.
\end{enumerate}
\end{corollary}

\begin{proof}
Part (a) follows from Corollary~\ref{oddcor}. For (b) and (c) we use
Theorem~\ref{simthm}, by which $\CL(A)$ contains an invertible matrix if
and only if $A$ is similar to $-A$, that is, if and only if $A$ is
balanced. The assertions then follow from Corollary~\ref{simeq}(a) with
$\omega=-1$ and $q=2$.
\end{proof}

\begin{remark}
\label{clrem}
The assumption in Corollary~\ref{clcor}(b) that $\CL(A)$ contains an
invertible matrix cannot be dropped. Over $\mathbb Q$ the matrices
$A=\diag(1,\frac23,\frac12)$ and $B=I_3$ satisfy $\CL(A)=\CL(B)=\{O\}$,
and they are not polynomially equivalent. This assumption is not
necessary for odd polynomial equivalence. The matrices $A=\diag(1,2)$ and
$B=\diag(3,4)$ over $\mathbb Q$ satisfy $\CL(A)=\CL(B)=\{O\}$,
$A=\frac{5}{42}B+\frac{1}{42}B^3$, and
$B=\frac{10}{3}A-\frac13A^3$.
\end{remark}

For nilpotent matrices we obtain the following.

\begin{corollary}
\label{nilcor}
Let $A\in M_n(\F)$ be a nonzero nilpotent matrix with nilpotency index
$m$, and let $\omega\in\F\setminus\{0,1\}$ have order $q$.
\begin{enumerate}
\item[(a)] Let $B\in M_n(\F)$. If $q$ is finite, then
$\cC_\omega(A)=\cC_\omega(B)$ if and only if
$B=\sum_{i=0}^{r}c_iA^{qi+1}$ with $c_i\in\F$ and $c_0\ne0$, where $r$ is
the largest integer with $qr+1<m$. In particular, if $q\ge m-1$, then
$\cC_\omega(A)=\cC_\omega(B)$ if and only if $B=cA$ with
$c\in\F^\times$. The same holds if $q$ is infinite.
\item[(b)] There is a matrix $B\in M_n(\F)$ with
$\cC_\omega(A)=\cC_{\omega^{-1}}(B)$ if and only if $q=2$.
\end{enumerate}
\end{corollary}

\begin{proof}
(a) By Theorem~\ref{simthm}, $A$ is similar to $\omega A$. Let $q$ be
finite. By Corollary~\ref{simeq}(a) the equality
$\cC_\omega(A)=\cC_\omega(B)$ holds if and only if $B\in A\F[A^q]$ and
$A\in B\F[B^q]$. Let $B=Ah(A^q)$ with $h\in\F[x]$ and $c_0=h(0)$. If
$c_0=0$, then $B\in A^{q+1}\F[A]$, and $A\in B\F[B^q]$ would give
$A=A^{q+1}u(A)$ with $u\in\F[x]$. Then $A(I-A^qu(A))=O$ with
$I-A^qu(A)$ invertible, which contradicts $A\ne O$. If $c_0\ne0$, then
$h(A^q)$ is invertible, and $B^q=A^qh(A^q)^q$ with
$h(A^q)^q\in\F[A^q]$ invertible. Hence $A^q$ and $B^q$ are nilpotent with the same nilpotency
index, and $\F[B^q]\subseteq\F[A^q]$ is an inclusion of spaces of the
same dimension. By the Cayley--Hamilton theorem the inverse of $h(A^q)$
is a polynomial in $h(A^q)$. Therefore $h(A^q)^{-1}\in\F[A^q]=\F[B^q]$
and $A=Bh(A^q)^{-1}\in B\F[B^q]$. The case of infinite $q$ is
Corollary~\ref{simeq}(b).

(b) If $q=2$, then $\omega^{-1}=\omega$ and $B=A$ has the required
property. Suppose that $\cC_\omega(A)=\cC_{\omega^{-1}}(B)$. For
$X\in\cC_\omega(A)$ we have $BX=\omega^{-1}XB$, that is, $XB=\omega BX$,
and hence $B\in\cC_\omega(\cC_\omega(A))$. By Corollary~\ref{thompcor},
$B$ lies in the span of the powers $A^k$ with $\omega^{k+1}=1$. If $q$ is
infinite, there is no such $k\ge0$, so that $B=O$,
$\cC_\omega(A)=\cC_{\omega^{-1}}(O)=M_n(\F)$, and $I\in\cC_\omega(A)$
gives $A=\omega A$ and $A=O$, a contradiction. Let $q$ be finite. Then
$B\in A^{q-1}\F[A^q]$, and $B$ is nilpotent. By
Corollary~\ref{thompcor} for the matrix $B$ and the scalars
$\omega^{-1}$ and $\varepsilon=\omega^{-1}$, the set
$\cC_{\omega^{-1}}(\cC_{\omega^{-1}}(B))$ is the span of the powers $B^k$
with $\omega^{-k-1}=1$, which is $B^{q-1}\F[B^q]$. Since $AX=\omega XA$
for all $X\in\cC_{\omega^{-1}}(B)$, we have
$A\in\cC_{\omega^{-1}}(\cC_{\omega^{-1}}(B))$. Thus
$A\in B^{q-1}\F[B^q]\subseteq A^{(q-1)^2}\F[A]$. If $q\ge3$, then
$(q-1)^2\ge2$, and $A=A^2u(A)$ with $u\in\F[x]$ gives $A=O$ as in (a).
Hence $q=2$.
\end{proof}

For arbitrary matrices the next proposition says when the equation of
part (b) has a solution $B$.

\begin{proposition}
\label{rankprop}
Let $\omega\in\F\setminus\{0,1\}$ with $\omega^2\ne1$, and let
$A\in M_n(\F)$. There is a matrix $B\in M_n(\F)$ with
$\cC_\omega(A)=\cC_{\omega^{-1}}(B)$ if and only if the ranks of $A$ and
$A^2$ are equal.
\end{proposition}

\begin{proof}
The space $\F^n$ is the direct sum of the kernel of $A^n$ and the image
of $A^n$. Both are mapped into themselves by $A$; the restriction of $A$
to the kernel is nilpotent, and its restriction to the image is
invertible. Hence there is $P\in\GL_n(\F)$ with
$P^{-1}AP=A_0\oplus A_1$, where $A_0\in M_r(\F)$ is nilpotent and
$A_1\in M_{n-r}(\F)$ is invertible, and one of the two blocks may be
absent. By Lemma~\ref{basic}(a) we may assume that $A=A_0\oplus A_1$. The
ranks of $A$ and $A^2$ are equal if and only if $A_0=O$.

We follow the proof of Theorem~2.2 of Alahmadi, Glasby, and
Praeger~\cite{AGP2017}. Let $X=(X_{ab})_{a,b=0}^{1}$ be partitioned like
$A$. If $X\in\cC_\omega(A)$, then $A_0X_{01}=\omega X_{01}A_1$, and by induction
$X_{01}=\omega^{-k}A_0^kX_{01}A_1^{-k}$ for all $k\ge1$, so that
$X_{01}=O$. In the same way
$X_{10}=\omega^kA_1^{-k}X_{10}A_0^k=O$. Hence
\begin{equation}
 \cC_\omega(A)=\cC_\omega(A_0)\oplus\cC_\omega(A_1),
 \label{fitting}
\end{equation}
where the right side is the set of the matrices $X_{00}\oplus X_{11}$
with
$X_{00}\in\cC_\omega(A_0)$ and $X_{11}\in\cC_\omega(A_1)$.

Suppose that $A_0=O$, and put $B=O_r\oplus A_1^{-1}$, where $O_r$ is the
zero matrix in $M_r(\F)$. A matrix $X_{11}$
satisfies $A_1X_{11}=\omega X_{11}A_1$ if and only if
$A_1^{-1}X_{11}=\omega^{-1}X_{11}A_1^{-1}$. By~\eqref{fitting} for $B$
and $\omega^{-1}$, we obtain
$\cC_{\omega^{-1}}(B)=M_r(\F)\oplus\cC_{\omega^{-1}}(A_1^{-1})
=\cC_\omega(O_r)\oplus\cC_\omega(A_1)=\cC_\omega(A)$.

Suppose that $A_0\ne O$ and that $\cC_\omega(A)=\cC_{\omega^{-1}}(B)$.
Since $A_0$ is nilpotent, $\cC_\omega(A_0)$ contains an invertible matrix
$K$ by Theorem~\ref{simthm}, and $K\oplus O\in\cC_{\omega^{-1}}(B)$.
Write $B=(B_{ab})$ in blocks. In positions $(1,0)$ and
$(0,1)$ the equation $B(K\oplus O)=\omega^{-1}(K\oplus O)B$ gives
$B_{10}K=O$ and $KB_{01}=O$, and hence $B=B_{00}\oplus B_{11}$. A matrix $Y\oplus O$ lies in
$\cC_\omega(A)$ if and only if $Y\in\cC_\omega(A_0)$, and it lies in
$\cC_{\omega^{-1}}(B)$ if and only if $Y\in\cC_{\omega^{-1}}(B_{00})$.
Therefore $\cC_\omega(A_0)=\cC_{\omega^{-1}}(B_{00})$, which is
impossible by Corollary~\ref{nilcor}(b), because $\omega^2\ne1$.
\end{proof}

\begin{example}
\label{jnex}
Let $A=J_n(0)$ and $\omega\in\F^\times$. A comparison of the entries of $AX$ and $\omega XA$ shows that $\cC_\omega(A)$ consists of the upper triangular
matrices $X=(x_{ij})$ with $x_{i+1,j+1}=\omega x_{ij}$ for
$1\le i\le j<n$. Thus $\dim\cC_\omega(A)=n$, and $X$ is determined by its
first row. Let $n=5$ and let $\omega$ have order $3$. By
Corollary~\ref{nilcor}(a), the matrices $B$ with
$\cC_\omega(A)=\cC_\omega(B)$ are the matrices $B=bA+cA^4$ with
$b\in\F^\times$ and $c\in\F$.
\end{example}

\section{Double generalized centralizers of arbitrary matrices}
\label{dgcsec}

For $\omega,\varepsilon\in\F^\times$ and a matrix $A$ that is not similar
to $\omega A$, the set $\cC_\varepsilon(\cC_\omega(A))$ need not consist of polynomials in $A$.
Theorem~\ref{structure} describes this set for every matrix over an
algebraically closed field. Theorem~\ref{mainA} determines the matrices
for which it is contained in $\F[A]$ when $\varepsilon$ is a power of
$\omega$, and Proposition~\ref{nonpow} does so for the other
$\varepsilon$. Until
Theorem~\ref{structure} the field $\F$ is algebraically closed,
$\omega\in\F\setminus\{0,1\}$, and $A\in M_n(\F)$. We use the following
notation.

Matrices in $M_n(\F)$ are identified with linear maps of $V=\F^n$. For a
linear map of a vector space of finite dimension, the generalized
centralizer, the minimal polynomial, the algebra of polynomials in the
map, and the nilpotency index are defined in the same way as for
matrices. For such a map $N$ and a vector $v$ of the space we write
$\F[N]v$ for the set of the vectors $g(N)v$ with $g\in\F[x]$. Let
$\Lambda$ be the set of the nonzero eigenvalues of $A$, let $V_0$ be the
generalized eigenspace of $A$ at $0$, which may be zero, and let $V_\mu$
be the generalized eigenspace at $\mu\in\Lambda$. Thus
$V=V_0\oplus\bigoplus_{\mu\in\Lambda}V_\mu$. We put $V_\mu=\{0\}$ for
$\mu\in\F^\times\setminus\Lambda$. Let $N_0$ be the restriction of $A$ to
$V_0$ and $N_\mu$ the restriction of $A-\mu I$ to $V_\mu$; these maps are
nilpotent. For $\mu\in\F^\times$ let $a(\mu)$ be the size of the largest
Jordan block of $A$ at $\mu$, which is the nilpotency index of $N_\mu$,
and let $a(\mu)=0$ if $\mu\notin\Lambda$. We write $a(0)$ for the
nilpotency index of $N_0$, with $a(0)=0$ if $V_0=\{0\}$. A linear map $Y$
of $V$ is written in blocks $Y_{\sigma\tau}\colon V_\tau\to V_\sigma$ with
$\sigma,\tau\in\{0\}\cup\Lambda$, where $Y_{\sigma\tau}$ is the
restriction of $Y$ to $V_\tau$ followed by the projection to $V_\sigma$.
Let $\Lambda_+$ be the set of the $\sigma\in\Lambda$ with
$a(\sigma)>a(\omega\sigma)$, and let $\Lambda_-$ be the set of the
$\tau\in\Lambda$ with $a(\tau)>a(\omega^{-1}\tau)$. We call the sets
$\Lambda\cap\{\omega^i\mu:i\in\mathbb Z\}$ with $\mu\in\Lambda$ the
classes of $\Lambda$.

\begin{lemma}
\label{DU}
Let $\sigma,\tau\in\Lambda$. The space $N_\sigma^{a(\omega\sigma)}(V_\sigma)$
is nonzero if and only if $\sigma\in\Lambda_+$, and it is $V_\sigma$ if
$\omega\sigma\notin\Lambda$. The space $\Ker N_\tau^{a(\omega^{-1}\tau)}$
is different from $V_\tau$ if and only if $\tau\in\Lambda_-$, and it is
$\{0\}$ if $\omega^{-1}\tau\notin\Lambda$. Moreover, the following holds.
\begin{enumerate}
\item[(a)] The set $\Lambda_+$ is empty if and only if $\Lambda_-$ is
empty, and this is the case if and only if $a(\omega\mu)=a(\mu)$ for all
$\mu\in\F^\times$.
\item[(b)] If $\Lambda_+=\Lambda_-=\{\nu\}$, then $a(\omega\mu)=a(\mu)$ for
all $\mu\in\F^\times\setminus\{\nu,\omega^{-1}\nu\}$.
\item[(c)] If $A$ satisfies condition (b) of Theorem~\ref{mainA} with
the element $\nu$, then $\Lambda_+=\Lambda_-=\{\nu\}$.
\end{enumerate}
\end{lemma}

\begin{proof}
The assertions on the two spaces hold because $a(\sigma)$ and $a(\tau)$
are the nilpotency indices of $N_\sigma$ and $N_\tau$, and because
$a(\mu)=0$ for every $\mu\in\F^\times\setminus\Lambda$. If
$a(\omega\mu)=a(\mu)$ for all $\mu$, then $\Lambda_+$ and $\Lambda_-$ are
empty. Suppose that $\Lambda_+$ is empty. Then
$a(\omega\sigma)\ge a(\sigma)\ge1$ for all $\sigma\in\Lambda$. Thus
multiplication by $\omega$ maps the finite set $\Lambda$ into itself, and
it is a bijection of $\Lambda$. Then
$\sum_{\sigma\in\Lambda}a(\omega\sigma)=\sum_{\sigma\in\Lambda}a(\sigma)$,
and hence $a(\omega\sigma)=a(\sigma)$ for all $\sigma\in\Lambda$. If
$\mu\notin\Lambda$, then $\omega\mu\notin\Lambda$, and
$a(\omega\mu)=a(\mu)=0$. If $\Lambda_-$ is empty, the same argument with
$\omega^{-1}$ in place of $\omega$ applies. This proves (a).

Let $\Lambda_+=\Lambda_-=\{\nu\}$ and
$\mu\in\F^\times\setminus\{\nu,\omega^{-1}\nu\}$. If $\mu\in\Lambda$,
then $\mu\notin\Lambda_+$ gives $a(\omega\mu)\ge a(\mu)\ge1$, so that
$\omega\mu\in\Lambda$, and $\omega\mu\notin\Lambda_-$ gives
$a(\omega\mu)\le a(\mu)$. If $\mu\notin\Lambda$, then
$\omega\mu\notin\Lambda$, because otherwise $\omega\mu\in\Lambda_-$ and
$\omega\mu=\nu$. In both cases $a(\omega\mu)=a(\mu)$. This proves (b).

Suppose now that $A$ satisfies condition (b) of Theorem~\ref{mainA} with
the element $\nu$. Then $\nu\in\Lambda$ and
$a(\omega\nu)=a(\omega^{-1}\nu)<a(\nu)$, so that $\nu$ lies in
$\Lambda_+$ and in $\Lambda_-$. Let $\mu\in\Lambda\setminus\{\nu\}$. If
$\mu=\omega^{-1}\nu$, then $a(\omega\mu)=a(\nu)>a(\mu)$, and otherwise
$a(\omega\mu)=a(\mu)$ by condition (b) of Theorem~\ref{mainA}. Hence
$\mu\notin\Lambda_+$. If
$\mu=\omega\nu$, then $a(\omega^{-1}\mu)=a(\nu)>a(\mu)$. Otherwise
$\omega^{-1}\mu$ is different from $\nu$ and $\omega^{-1}\nu$, and
condition (b) for $\omega^{-1}\mu$ gives $a(\mu)=a(\omega^{-1}\mu)$.
Hence $\mu\notin\Lambda_-$.
\end{proof}

\begin{lemma}
\label{decomp}
A linear map $X$ of $V$ lies in $\cC_\omega(A)$ if and only if the
following two conditions hold.
\begin{enumerate}
\item[(a)] We have $X(V_0)\subseteq V_0$, and the restriction of $X$ to
$V_0$ lies in $\cC_\omega(N_0)$.
\item[(b)] For every $\mu\in\Lambda$ we have
$X(V_\mu)\subseteq V_{\omega\mu}$, and the restriction
$X_\mu\colon V_\mu\to V_{\omega\mu}$ of $X$ satisfies
$N_{\omega\mu}X_\mu=\omega X_\mu N_\mu$.
\end{enumerate}
\end{lemma}

\begin{proof}
Let $X\in\cC_\omega(A)$ and $\mu\in\F$. Then
$(A-\omega\mu I)X=\omega X(A-\mu I)$, and by induction
$(A-\omega\mu I)^kX=\omega^kX(A-\mu I)^k$ for $k\ge1$. For $k=n$ this
shows that $X$ maps the generalized eigenspace at $\mu$ into the kernel
of $(A-\omega\mu I)^n$, which is $V_{\omega\mu}$. If $X$ maps $V_\mu$
into $V_{\omega\mu}$ for some $\mu\in\Lambda$, then on $V_\mu$ the
equation $AX=\omega XA$ reads
$(\omega\mu+N_{\omega\mu})X_\mu=\omega X_\mu(\mu+N_\mu)$, which is
$N_{\omega\mu}X_\mu=\omega X_\mu N_\mu$. On $V_0$ it reads
$N_0X=\omega XN_0$.
\end{proof}

Kurosaki~\cite[p.~25]{Kurosaki1941} stated the inclusions in (a) and
(b), in a form that applies over every field. In a basis in which
$A$ has Jordan form, Lemma~\ref{decomp} is the description of
$\cC_\omega(A)$ by Dolinar, Guterman, Kuzma, and
Markova~\cite[Proposition~7]{DGKM2019}, who also wrote down the matrices
of the restrictions of $X$ in (a) and (b).

\begin{lemma}
\label{intertw}
Let $U$ and $U'$ be nonzero vector spaces of finite dimension over a
field $\F$, let $N$ and $N'$ be nilpotent linear maps of $U$ and $U'$
with nilpotency indices $\alpha$ and $\beta$, let $\omega\in\F^\times$,
and let $H$ be the space of the linear maps $X\colon U\to U'$ with
$N'X=\omega XN$.
\begin{enumerate}
\item[(a)] The intersection of the kernels of the maps $X\in H$ is
$N^\beta(U)$, and the sum of the images $X(U)$ with $X\in H$ is
$\Ker N'^\alpha$.
\item[(b)] Let $Y$ and $Y'$ be linear maps of $U$ and $U'$. Then
$XY=Y'X$ for all $X\in H$ if and only if there is a polynomial
$f\in\F[x]$ such that $Y-f(\omega N)$ maps $U$ into $N^\beta(U)$ and
$Y'-f(N')$ is zero on $\Ker N'^\alpha$.
\end{enumerate}
\end{lemma}

\begin{proof}
By the Jordan form of nilpotent maps there are decompositions
$U=\bigoplus_{i=1}^{k}U_i$ and $U'=\bigoplus_{j=1}^{l}U'_j$ with
$U_i=\F[N]v_i$ and $U'_j=\F[N']v'_j$, where $N^{a_i}v_i=0$,
$N^{a_i-1}v_i\ne0$, $N'^{b_j}v'_j=0$, and $N'^{b_j-1}v'_j\ne0$. Then
$\alpha=\max_ia_i$ and $\beta=\max_jb_j$. We use the following facts,
valid for every $g\in\F[x]$ and every integer $m\ge0$. The vector
$g(N)v_i$ lies in $N^m(U_i)$ if and only if $x^{\min(a_i,m)}$ divides
$g$, and in particular $g(N)v_i=0$ if and only if $x^{a_i}$ divides $g$.
Moreover, $\Ker N'^m\cap U'_j=N'^{\max(0,b_j-m)}(U'_j)$.

Every $X\in H$ satisfies $Xg(N)=g(\omega^{-1}N')X$ for $g\in\F[x]$, and
$N'^{a_i}Xv_i=\omega^{a_i}XN^{a_i}v_i=0$. Conversely, for an index $i$
and a vector $w\in U'$ with $N'^{a_i}w=0$ there is exactly one $X\in H$
with $Xv_i=w$ and $Xv_{i'}=0$ for $i'\ne i$, namely the map with
$X(g(N)v_i)=g(\omega^{-1}N')w$ that is zero on the other summands. This
map is well defined, because $g(N)v_i=0$ implies that $x^{a_i}$ divides
$g$, so that $g(\omega^{-1}N')w=0$. For indices $i$ and $j$ let
$X_{ij}\in H$ be the map of this kind with
$X_{ij}v_i=N'^{\max(0,b_j-a_i)}v'_j$. By these facts, the image of
$X_{ij}$ is $\Ker N'^{a_i}\cap U'_j$. The restriction of $X_{ij}$ to $U_i$
has the kernel $N^{b_j}(U_i)$, because
$X_{ij}g(N)v_i=g(\omega^{-1}N')N'^{\max(0,b_j-a_i)}v'_j$, and this vector
is zero if and only if $x^{b_j-\max(0,b_j-a_i)}$ divides $g$, that is, if
and only if $x^{\min(a_i,b_j)}$ divides $g$.

Figure~\ref{chainfig} shows two such maps, for $a_i=2$, $b_j=4$ and for
$a_i=4$, $b_j=2$. Each dot is one of the vectors $N^tv_i$ with
$0\le t<a_i$ or $N'^tv'_j$ with $0\le t<b_j$, and the gray arrows show $N$
and $N'$. A black arrow ends at the image under $X_{ij}$ of the vector
where it starts, up to a nonzero factor, and $X_{ij}$ maps to zero the
vectors of $U_i$ from which no black arrow starts. In each panel the dash
dotted frame encloses the image $\Ker N'^{a_i}\cap U'_j$ of $X_{ij}$. In
the right panel the dashed frame encloses the kernel $N^{b_j}(U_i)$ of
the restriction of $X_{ij}$ to $U_i$, and in the left panel this kernel
is zero.

\begin{figure}[H]
\centering
\includegraphics[width=0.92\textwidth]{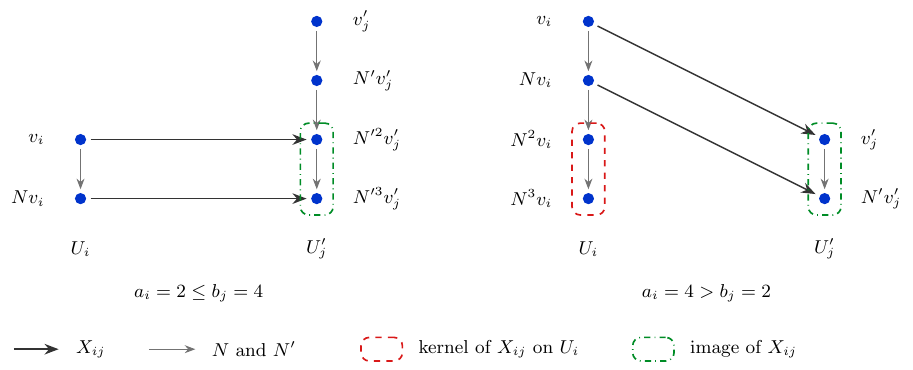}
\caption{The maps $X_{ij}$ in the proof of Lemma~\ref{intertw}.}
\label{chainfig}
\end{figure}

(a) Let $X\in H$. If $v=N^\beta u$, then
$Xv=\omega^{-\beta}N'^\beta Xu=0$. Since $N'^{a_i}Xv_i=0$ and
$a_i\le\alpha$, the image $X(U)=\sum_{i=1}^{k}\F[N']Xv_i$ lies in
$\Ker N'^\alpha$. Conversely, let $v=\sum_{i=1}^{k}g_i(N)v_i$ lie in the kernel
of every element of $H$, and choose $j_0$ with $b_{j_0}=\beta$. Then
$X_{ij_0}v=X_{ij_0}g_i(N)v_i=0$ gives
$g_i(N)v_i\in N^{\beta}(U_i)$ for every $i$, and hence
$v\in N^\beta(U)$. For the images choose $i_0$ with $a_{i_0}=\alpha$.
The images of the maps $X_{i_0j}$ are the spaces
$\Ker N'^\alpha\cap U'_j$, and their sum over $j$ is $\Ker N'^\alpha$.

(b) Suppose that $Y-f(\omega N)$ maps $U$ into $N^\beta(U)$ and that
$Y'-f(N')$ is zero on $\Ker N'^\alpha$. For $X\in H$ we have
$X(Y-f(\omega N))=O$ and $(Y'-f(N'))X=O$ by (a), and
$Xf(\omega N)=f(N')X$. Hence $XY=Y'X$.

Conversely, suppose that $XY=Y'X$ for all $X\in H$. If $v$ lies in the
kernel of every $X\in H$, then $XYv=Y'Xv=0$ for all $X\in H$, so that
$Yv$ lies in this kernel as well. Hence $Y$ maps $N^\beta(U)$ into
itself by (a). In the same way $Y'X(U)=XY(U)\subseteq X(U)$ for
$X\in H$, and $Y'$ maps $\Ker N'^\alpha$ into itself. Let
$W=U/N^\beta(U)$ and $W'=\Ker N'^\alpha$, let $\pi\colon U\to W$ be the
quotient map and $\iota\colon W'\to U'$ the inclusion, let $\bar N$ and
$\bar Y$ be the maps of $W$ induced by $N$ and $Y$, and let $N''$ and
$Y''$ be the restrictions of $N'$ and $Y'$ to $W'$. By (a), every $X\in H$
has the form $X=\iota Z\pi$ with a linear map $Z\colon W\to W'$ that
satisfies $N''Z=Z(\omega\bar N)$, and for every such map $Z$ the map
$\iota Z\pi$ lies in $H$. Since $XY=\iota Z\bar Y\pi$ and
$Y'X=\iota Y''Z\pi$, we obtain $Y''Z=Z\bar Y$ for all linear maps
$Z\colon W\to W'$ with $N''Z=Z(\omega\bar N)$.

We claim that $\omega\bar N$ and $N''$ have the same minimal polynomial
$x^{\min(\alpha,\beta)}$. By the facts listed at the beginning of the
proof, the vector $N^mv_i$ lies in $N^\beta(U_i)$ if and only if
$x^{\min(a_i,\beta)}$ divides $x^m$. Hence $N^m(U_i)\subseteq
N^\beta(U_i)$ if and only if $m\ge\min(a_i,\beta)$, so that
$(\omega\bar N)^m=\omega^m\bar N^m$ is zero if and only if
$m\ge\min(\alpha,\beta)$. The space $W'$ is the direct sum
of the spaces $\Ker N'^\alpha\cap U'_j=N'^{\max(0,b_j-\alpha)}(U'_j)$,
on which $N'$ has nilpotency index $\min(b_j,\alpha)$, so that $N''$ has
nilpotency index $\min(\alpha,\beta)$. This proves the claim. By
Theorem~\ref{thompson} with $A_1=N''$, $A_2=\omega\bar N$, $C_1=Y''$, and
$C_2=\bar Y$, there is a polynomial $f\in\F[x]$ with
$Y''=f(N'')$ and $\bar Y=f(\omega\bar N)$. The first equation says that
$Y'-f(N')$ is zero on $\Ker N'^\alpha$, and the second says that
$Y-f(\omega N)$ maps $U$ into $N^\beta(U)$.
\end{proof}

\begin{theorem}
\label{structure}
Let $\F$ be algebraically closed, $\omega\in\F\setminus\{0,1\}$,
$\varepsilon\in\F^\times$, and $A\in M_n(\F)$. A linear map $Y$ of $V$
lies in $\cC_\varepsilon(\cC_\omega(A))$ if and only if the following
three conditions hold.
\begin{enumerate}
\item[(a)] For all $\sigma,\tau\in\Lambda$ we have $Y_{0\tau}=O$ and
$Y_{\sigma0}=O$, and $Y_{00}$ lies in the span of the powers $N_0^m$
with $m\ge0$ and $\omega^m\varepsilon=1$.
\item[(b)] For all $\sigma,\tau\in\Lambda$ with $\sigma\ne\tau$, the
block $Y_{\sigma\tau}$ maps $V_\tau$ into
$N_\sigma^{a(\omega\sigma)}(V_\sigma)$ and is zero on
$\Ker N_\tau^{a(\omega^{-1}\tau)}$.
\item[(c)] For every $\mu\in\Lambda$ with $\omega\mu\in\Lambda$ there is
a polynomial $f_\mu\in\F[x]$ such that
$Y_{\mu\mu}-f_\mu(\omega N_\mu)$ maps $V_\mu$ into
$N_\mu^{a(\omega\mu)}(V_\mu)$ and
$\varepsilon Y_{\omega\mu,\omega\mu}-f_\mu(N_{\omega\mu})$ is zero on
$\Ker N_{\omega\mu}^{a(\mu)}$.
\end{enumerate}
\end{theorem}

\begin{proof}
By Lemma~\ref{decomp}, $\cC_\omega(A)$ is the direct sum of the space
$H_0$ of the maps that are zero on all $V_\mu$ with $\mu\in\Lambda$ and
restrict to an element of $\cC_\omega(N_0)$ on $V_0$, and of the spaces
$H_\mu$, $\mu\in\Lambda$, of the maps $X$ that are zero on $V_\tau$ for
$\tau\ne\mu$ and map $V_\mu$ into $V_{\omega\mu}$ with
$N_{\omega\mu}X=\omega XN_\mu$ on $V_\mu$. Here $H_\mu=\{O\}$ if
$\omega\mu\notin\Lambda$. Hence
$Y\in\cC_\varepsilon(\cC_\omega(A))$ if and only if $XY=\varepsilon YX$
for all $X$ in each of these spaces.

Let $\mu\in\Lambda$ with $\omega\mu\in\Lambda$ and $X\in H_\mu$. The
nonzero blocks of $XY$ are among the maps
$XY_{\mu\tau}\colon V_\tau\to V_{\omega\mu}$, and the nonzero blocks of
$YX$ are among the maps
$Y_{\sigma,\omega\mu}X\colon V_\mu\to V_\sigma$, where $\sigma$ and
$\tau$ run through $\{0\}\cup\Lambda$. A comparison of the blocks shows that $XY=\varepsilon YX$ if and only if
$XY_{\mu\tau}=O$ for $\tau\ne\mu$,
$Y_{\sigma,\omega\mu}X=O$ for $\sigma\ne\omega\mu$, and
$XY_{\mu\mu}=\varepsilon Y_{\omega\mu,\omega\mu}X$. By
Lemma~\ref{intertw}(a) for $N=N_\mu$ and $N'=N_{\omega\mu}$, these
conditions hold for all $X\in H_\mu$ if and only if $Y_{\mu\tau}$ maps
into $N_\mu^{a(\omega\mu)}(V_\mu)$ for $\tau\ne\mu$,
$Y_{\sigma,\omega\mu}$ is zero on $\Ker N_{\omega\mu}^{a(\mu)}$ for
$\sigma\ne\omega\mu$, and
$XY_{\mu\mu}=\varepsilon Y_{\omega\mu,\omega\mu}X$ for all $X\in H_\mu$.
By Lemma~\ref{intertw}(b) the last condition is the existence of $f_\mu$
as in (c).

In the same way, $XY=\varepsilon YX$ for all $X\in H_0$ if and only if
$XY_{0\tau}=O$ for $\tau\ne0$, $Y_{\sigma0}X=O$ for $\sigma\ne0$, and
$XY_{00}=\varepsilon Y_{00}X$ for all $X\in\cC_\omega(N_0)$. If
$V_0=\{0\}$, then every $Y$ satisfies these conditions and condition (a).
Let $V_0\ne\{0\}$. The
nilpotent map $N_0$ is similar to $\omega N_0$, so that
$\cC_\omega(N_0)$ contains an invertible map by Theorem~\ref{simthm}.
Hence the first two conditions mean $Y_{0\tau}=O$ and $Y_{\sigma0}=O$,
and by Corollary~\ref{thompcor} the third condition means that $Y_{00}$
lies in the span of the powers $N_0^m$ with $\omega^m\varepsilon=1$.

Together, the conditions for $H_0$ and for the spaces $H_\mu$ are (a),
(c), and for
$\sigma\ne\tau$ in $\Lambda$ the conditions that $Y_{\sigma\tau}$ maps
into $N_\sigma^{a(\omega\sigma)}(V_\sigma)$ if $\omega\sigma\in\Lambda$
and is zero on $\Ker N_\tau^{a(\omega^{-1}\tau)}$ if
$\omega^{-1}\tau\in\Lambda$. By Lemma~\ref{DU} these two spaces are
$V_\sigma$ if $\omega\sigma\notin\Lambda$ and $\{0\}$ if
$\omega^{-1}\tau\notin\Lambda$, so that these conditions are (b).
\end{proof}

\begin{corollary}
\label{annih}
In the situation of Theorem~\ref{structure}, the intersection of the
kernels of the elements of $\cC_\omega(A)$ is
$\bigoplus_{\sigma\in\Lambda}N_\sigma^{a(\omega\sigma)}(V_\sigma)$, and
the sum of their images is
$V_0\oplus\bigoplus_{\tau\in\Lambda}\Ker N_\tau^{a(\omega^{-1}\tau)}$.
Let $\mathcal J(A)$ be the set of the linear maps $Z$ of $V$ with
$ZX=XZ=O$ for all $X\in\cC_\omega(A)$. Then $\mathcal J(A)$ consists of
the maps $Z$ that are zero on $V_0$ and on
$\Ker N_\tau^{a(\omega^{-1}\tau)}$ for all $\tau\in\Lambda$ and map $V$
into $\bigoplus_{\sigma\in\Lambda}N_\sigma^{a(\omega\sigma)}(V_\sigma)$.
Moreover, $\mathcal J(A)\subseteq\cC_\varepsilon(\cC_\omega(A))$, and
\[
 \dim\mathcal J(A)=\Bigl(\sum_{\sigma\in\Lambda}
 \dim N_\sigma^{a(\omega\sigma)}(V_\sigma)\Bigr)
 \Bigl(\sum_{\tau\in\Lambda}\dim N_\tau^{a(\omega^{-1}\tau)}(V_\tau)\Bigr).
\]
\end{corollary}

\begin{proof}
By Lemma~\ref{decomp}, the elements of $\cC_\omega(A)$ are the maps $X$
that map $V_0$ into $V_0$ and each $V_\sigma$ with $\sigma\in\Lambda$ into
$V_{\omega\sigma}$ and whose restrictions satisfy the conditions of that
lemma, and these restrictions can be chosen independently. Their images
lie in different summands of $V=V_0\oplus\bigoplus_{\mu\in\Lambda}V_\mu$.
Hence a vector of $V$ is mapped to zero by every $X\in\cC_\omega(A)$ if
and only if each of its components in $V_0$ and in the spaces $V_\sigma$
has this property, and the sum of the images is the direct sum of the
sums of the images of the restrictions. If $V_0\ne\{0\}$, then
Lemma~\ref{intertw}(a) for $N=N'=N_0$ gives the two assertions on $V_0$,
and if $V_0=\{0\}$, there is nothing to prove on $V_0$.
Lemma~\ref{intertw}(a) for $N=N_\sigma$ and $N'=N_{\omega\sigma}$ with
$\omega\sigma\in\Lambda$ gives the two assertions on the spaces
$V_\sigma$ with $\omega\sigma\in\Lambda$ and on the spaces $V_\tau$ with
$\omega^{-1}\tau\in\Lambda$. If $\omega\sigma\notin\Lambda$,
then every element of $\cC_\omega(A)$ is zero on $V_\sigma$, which is
$N_\sigma^{a(\omega\sigma)}(V_\sigma)$ by Lemma~\ref{DU}. If
$\omega^{-1}\tau\notin\Lambda$, then the images have no component in
$V_\tau$, and $\Ker N_\tau^{a(\omega^{-1}\tau)}=\{0\}$ by Lemma~\ref{DU}.
This proves the first sentence.

We have $XZ=O$ for all $X\in\cC_\omega(A)$ if and only if $Z$ maps $V$
into the intersection of the kernels of these $X$, and $ZX=O$ for all
$X$ if and only if $Z$ is zero on the sum of their images. This gives the
description of $\mathcal J(A)$. The inclusion $\mathcal J(A)\subseteq\cC_\varepsilon(\cC_\omega(A))$
holds by the definition. The dimension of $\mathcal J(A)$ is the product
of the dimension of the intersection of the kernels and the codimension
of the sum of the images, and the codimension of
$\Ker N_\tau^{a(\omega^{-1}\tau)}$ in $V_\tau$ is the dimension of
$N_\tau^{a(\omega^{-1}\tau)}(V_\tau)$.
\end{proof}

Apart from the requirement in (b) that $A$ has exactly one Jordan block
of size $a(\nu)$ at $\nu$, the two conditions of Theorem~\ref{mainA}
concern only the function $a$. For $\omega$ of order $4$ and a
nonzero element $\nu$, Figure~\ref{orbitfig} shows values of $a$ on the
set $\{\nu,\omega\nu,\omega^2\nu,\omega^3\nu\}$. If $A$ has no other nonzero eigenvalues, the values
in the left panel satisfy condition (a), those in the middle panel
satisfy condition (b) when $A$ has a single Jordan block of size $3$ at
$\nu$, and those in the right panel satisfy neither condition. The
elements of the sets $\Lambda_+$ and $\Lambda_-$ are marked by a dashed
and a dash dotted circle, respectively.

\begin{figure}[H]
\centering
\includegraphics[width=\textwidth]{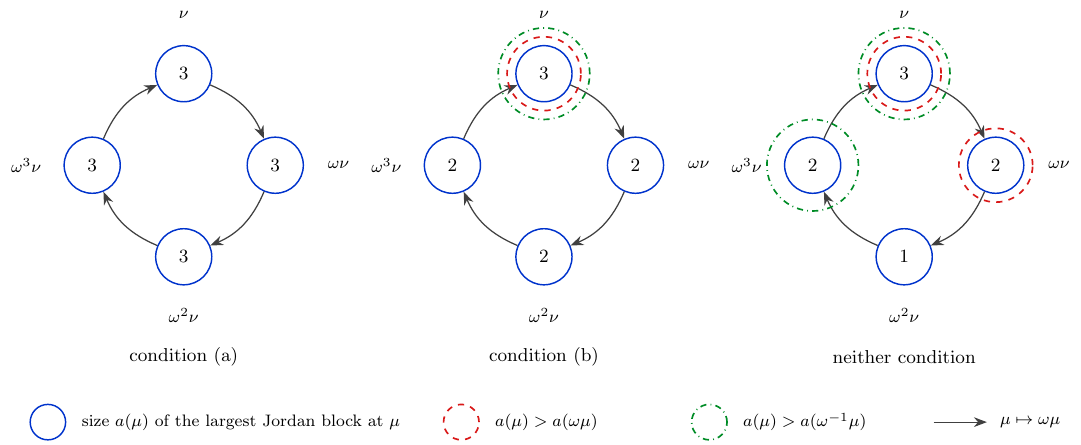}
\caption{Conditions (a) and (b) of Theorem~\ref{mainA} and the sets
$\Lambda_+$ and $\Lambda_-$ for $\omega$ of order $4$ (see
Lemma~\ref{DU}).}
\label{orbitfig}
\end{figure}

The following lemma gives the conditions at the element $\nu$ in
condition (b) of Theorem~\ref{mainA}.

\begin{lemma}
\label{nilpoly}
Let $N$ be a nilpotent linear map of a vector space $U$ of finite
dimension over a field $\F$, let $c\ge1$ be its nilpotency index, and
let $\ell_1$ and $\ell_2$ be integers with $0\le \ell_1,\ell_2\le c-1$. Suppose that
every linear map of $U$ that is zero on $\Ker N^{\ell_1}$ and maps $U$ into
$N^{\ell_2}(U)$ is a polynomial in $N$. Then $N$ has exactly one Jordan
block of size $c$, and $\ell_1=\ell_2=c-1$.
\end{lemma}

\begin{proof}
Choose $v\in U$ with $N^{c-1}v\ne0$ and a complement $C$ of $\F v$ in $U$
that contains $\Ker N^{c-1}$. Let $Z_1$ be the linear map of $U$ with
$Z_1v=N^{c-1}v$ that is zero on $C$. It is zero on
$\Ker N^{\ell_1}\subseteq\Ker N^{c-1}$ and maps $U$ into
$N^{c-1}(U)\subseteq N^{\ell_2}(U)$, so that $Z_1=g(N)$ with $g\ne0$.
Write $g(x)=x^kh(x)$ with $h(0)\ne0$. Then $N^kv=h(N)^{-1}N^{c-1}v$ is
a nonzero vector in the kernel of $N$, which gives $k=c-1$. Hence
$N^{c-1}=h(N)^{-1}Z_1$ has rank $1$, and $N$ has exactly one Jordan block
of size $c$. Suppose that $\ell_1\le c-2$. Then $v$ and $Nv$ are linearly
independent modulo $\Ker N^{\ell_1}$, because $N^{c-2}v$ and $N^{c-1}v$ are
linearly independent. Hence there is a linear map $Z_2$ of $U$ with image
in $\F N^{c-1}v$ that is zero on $\Ker N^{\ell_1}$ and satisfies $Z_2v=0$
and $Z_2(Nv)=N^{c-1}v$. It maps $U$ into
$\F N^{c-1}v\subseteq N^{\ell_2}(U)$, because $\ell_2\le c-1$, and it commutes
with $N$, because it is a polynomial in $N$. We obtain
$N^{c-1}v=Z_2Nv=NZ_2v=0$, a contradiction. Thus $\ell_1=c-1$. Suppose that
$\ell_2\le c-2$, and let $Z_3$ be the linear map of $U$ with
$Z_3v=N^{c-2}v$ that is zero on $C$. It is zero on
$\Ker N^{\ell_1}=\Ker N^{c-1}$ and maps $U$ into
$\F N^{c-2}v\subseteq N^{\ell_2}(U)$, because $\ell_2\le c-2$. Since $Z_3$
commutes with $N$ and $Nv\in C$, we obtain $N^{c-1}v=NZ_3v=Z_3Nv=0$, a
contradiction. Thus $\ell_2=c-1$.
\end{proof}

\begin{lemma}
\label{Jpoly}
In the situation of Theorem~\ref{structure}, suppose that the space
$\mathcal J(A)$ of Corollary~\ref{annih} is contained in $\F[A]$. Then
$A$ satisfies condition (a) or condition (b) of Theorem~\ref{mainA}.
\end{lemma}

\begin{proof}
Every polynomial in $A$ maps each $V_\mu$ into itself, and for
$\mu\in\Lambda$ its restriction to $V_\mu$ is a polynomial in $N_\mu$. If
$\sigma\in\Lambda_+$ and $\tau\in\Lambda_-$ with $\sigma\ne\tau$, then by
Lemma~\ref{DU} and Corollary~\ref{annih} there is a nonzero map
$Z\in\mathcal J(A)$ that maps $V_\tau$ into $V_\sigma$ and is zero on the
other generalized eigenspaces. This is impossible, because $Z$ does not
map $V_\tau$ into itself. Hence $\Lambda_+$ is empty, or $\Lambda_-$ is empty, or
$\Lambda_+=\Lambda_-=\{\nu\}$ for some $\nu\in\Lambda$. In the first two
cases (a) holds by Lemma~\ref{DU}(a).

Let $\Lambda_+=\Lambda_-=\{\nu\}$. By Lemma~\ref{DU}(b) we have
$a(\omega\mu)=a(\mu)$ for all
$\mu\in\F^\times\setminus\{\nu,\omega^{-1}\nu\}$. Put $N=N_\nu$,
$c=a(\nu)$, $\ell_1=a(\omega^{-1}\nu)$, and $\ell_2=a(\omega\nu)$, so
that $\ell_1<c$ and $\ell_2<c$. By Corollary~\ref{annih}, if a linear map of $V_\nu$ is zero on
$\Ker N^{\ell_1}$ and maps $V_\nu$ into $N^{\ell_2}(V_\nu)$, then its
extension by zero on the other generalized eigenspaces lies in
$\mathcal J(A)$. This extension is therefore a polynomial in $A$, and its
restriction to $V_\nu$ is a polynomial in $N$. By Lemma~\ref{nilpoly}, $A$ has exactly one Jordan
block of size $c$ at $\nu$, and $\ell_1=\ell_2=c-1$. Thus (b) holds.
\end{proof}

The proof of Theorem~\ref{mainA} also uses the following count.

\begin{lemma}
\label{dimpow}
In the situation of Theorem~\ref{structure}, let $\omega$ have finite
order $q$, let $j\ge0$ be an integer, and let $z$ be the number of
integers $i\ge0$
with $qi+j<a(0)$. Then the dimension of $A^j\F[A^q]$ is $z$ plus the sum
of the largest values of $a$ on the classes of $\Lambda$, one value for
each class.
\end{lemma}

\begin{proof}
Since $A$ is invertible on each $V_\mu$ with $\mu\in\Lambda$, a
polynomial $h$ satisfies $A^jh(A^q)=O$ if and only if $N_0^jh(N_0^q)=O$
and $h((\mu I+N_\mu)^q)=O$ for all $\mu\in\Lambda$. The first condition
means that $x^z$ divides $h$. The characteristic of $\F$ does not divide
$q$, because $x^q-1$ has the $q$ distinct roots $\omega^i$ with
$0\le i<q$. Thus $q$ is invertible in $\F$, and
$(\mu I+N_\mu)^q=\mu^qI+N_\mu R$ with an invertible map $R$ that commutes
with $N_\mu$. Hence $h((\mu I+N_\mu)^q)=O$ if and only if
$(x-\mu^q)^{a(\mu)}$ divides $h$. Two elements $\mu,\mu'\in\Lambda$
satisfy $\mu^q=\mu'^q$ if and only if they lie in the same class, because
the roots of $x^q-1$ are the powers of $\omega$.
Therefore the polynomials $h$ with $A^jh(A^q)=O$ are the multiples of a
polynomial whose degree is $z$ plus the sum of the largest values of $a$
on the classes of $\Lambda$, one value for each class. The space
$A^j\F[A^q]$ is the image of the map $h\mapsto A^jh(A^q)$, so that its
dimension is this degree.
\end{proof}

We recall that for a field $\F$ that is not algebraically closed, the
Jordan blocks of
$A\in M_n(\F)$ are those of its Jordan form over $\Fbar$, and for
$\mu\in\Fbar^\times$ the number $a(\mu)$ is the size of the largest
Jordan block of $A$ at $\mu$, or $0$ if $\mu$ is not an eigenvalue of
$A$.

\begin{proof}[Proof of Theorem~\ref{mainA}]
By Lemma~\ref{scalars} we may assume that $\F$ is algebraically closed,
and we use the notation of Section~\ref{dgcsec}.

Suppose that (i) holds. By Corollary~\ref{annih} we have
$\mathcal J(A)\subseteq\cC_\varepsilon(\cC_\omega(A))\subseteq\F[A]$,
and Lemma~\ref{Jpoly} shows that (a) or (b) holds.

Suppose now that (ii) holds. Since $a(\mu)$ is the multiplicity of $\mu$
as a root of $m_A$, condition (a) says that $\mu$ and $\omega\mu$ have the
same multiplicity as roots of $m_A$ for every $\mu\in\F^\times$, and this
also holds for $\mu=0$. If (a) holds, then $m_A$ is balanced with
respect to $\omega$ by Lemma~\ref{balroots}(a), and
Corollary~\ref{thompcor} shows that
$\cC_\varepsilon(\cC_\omega(A))$ is the span of the powers $A^m$ with
integers $m\ge0$ and $\omega^m\varepsilon=1$. This span lies in $\F[A]$,
and it is $A^j\F[A^q]$ if $q$ is finite and $0\le j<q$.

Suppose that (b) holds, and let $Y\in\cC_\varepsilon(\cC_\omega(A))$.
Then $\Lambda_+=\Lambda_-=\{\nu\}$ by Lemma~\ref{DU}(c), and Theorem~\ref{structure}(a) and (b)
and Lemma~\ref{DU} give $Y_{\sigma\tau}=O$ for all $\sigma\ne\tau$ in
$\{0\}\cup\Lambda$.

Suppose first that $q$ is infinite, and let $\mu\in\Lambda$. The
elements $\omega^i\mu$ with $i\in\mathbb Z$ are distinct, and only
finitely many of them lie in $\Lambda$. If $k$ is the largest and $k'$
the smallest integer $i$ with $\omega^i\mu\in\Lambda$, then
$\omega^k\mu\in\Lambda_+$ and $\omega^{k'}\mu\in\Lambda_-$. Hence
$\omega^k\mu=\nu=\omega^{k'}\mu$, so that $k=k'=0$ and $\mu=\nu$. Thus
$\Lambda=\{\nu\}$ and $a(\nu)=a(\omega\nu)+1=1$. Since $A$ has exactly
one Jordan block of size $1$ at $\nu$, the space $V_\nu$ has dimension
$1$, and $Y_{\nu\nu}$ is a scalar. The block $Y_{00}$ is a polynomial in
$N_0$ by Theorem~\ref{structure}(a), and the Chinese remainder theorem
gives a polynomial $h$ with $Y=h(A)$.

Suppose now that $q$ is finite. Since $\varepsilon=\omega^{-j}$ depends
only on the remainder of $j$ modulo $q$, we may assume that $0\le j<q$. By
Lemma~\ref{basic}(c) we have
$A^j\F[A^q]\subseteq\cC_\varepsilon(\cC_\omega(A))$, and it suffices to
prove that
\begin{equation}
 \dim\cC_\varepsilon(\cC_\omega(A))\le\dim A^j\F[A^q].
 \label{dimbound}
\end{equation}
Let $z$ be the number of integers $i\ge0$ with $qi+j<a(0)$. By
Lemma~\ref{dimpow}, the dimension of $A^j\F[A^q]$ is $z$ plus the sum of
the largest values of $a$ on the classes of $\Lambda$, one value for each
class. Theorem~\ref{structure}(a) shows that the block $Y_{00}$ lies in
the span of the powers $N_0^{qi+j}$ with integers $i\ge0$, which has
dimension $z$. Since $Y$ is the
direct sum of its diagonal blocks, it suffices to show for every class that, as $Y$ runs through
$\cC_\varepsilon(\cC_\omega(A))$, the family of the blocks $Y_{\mu\mu}$
with $\mu$ in the class lies in a space of dimension at most the largest
value of $a$ on the class. Consider a class
that does not contain $\nu$. Then $a(\omega\mu)=a(\mu)\ge1$ for all $\mu$
in the class. Hence the class is $\{\omega^i\mu_0:0\le i<q\}$ for some
$\mu_0$, the function $a$ is constant on it, and no element of the class
lies in $\Lambda_+$ or in $\Lambda_-$. By Theorem~\ref{structure}(c) and
Lemma~\ref{DU}, for every $\mu$ in the class there is a polynomial
$f_\mu$ with $Y_{\mu\mu}=f_\mu(\omega N_\mu)$ and
$\varepsilon Y_{\omega\mu,\omega\mu}=f_\mu(N_{\omega\mu})$. Hence
$Y_{\mu\mu}$ lies in $\F[N_\mu]$, which has dimension $a(\mu_0)$, and
$Y_{\mu\mu}$ determines $f_\mu$ modulo $x^{a(\mu_0)}$ and therefore
$Y_{\omega\mu,\omega\mu}$. Thus $Y_{\mu_0\mu_0}$ determines the family of
the blocks $Y_{\mu\mu}$ with $\mu$ in the class, and this family lies in
a space of dimension at most $a(\mu_0)$.

It remains to consider the class of $\nu$. Put $c=a(\nu)$ and
$\nu_i=\omega^i\nu$ for $0\le i\le q$, so that $\nu_q=\nu$. By (b) we have
$a(\nu_1)=a(\nu_{q-1})=c-1$ and $a(\nu_{i+1})=a(\nu_i)$ for
$1\le i\le q-2$, so that $a(\nu_i)=c-1$ for $1\le i<q$. If $c=1$, the
class of $\nu$ is $\{\nu\}$ and $V_\nu$ has dimension $1$, so that
$Y_{\nu\nu}$ lies in a space of dimension $1$. Let $c\ge2$. For
$1\le i<q$ the element $\nu_i$ lies neither in $\Lambda_+$ nor in
$\Lambda_-$, because $a(\nu_{i-1})$ and $a(\nu_{i+1})$ are at least
$c-1=a(\nu_i)$. Since $a(\omega\nu)=a(\omega^{-1}\nu)=c-1$, the spaces
$N_\nu^{a(\omega\nu)}(V_\nu)$ and $\Ker N_\nu^{a(\omega^{-1}\nu)}$ are
$N_\nu^{c-1}(V_\nu)$ and $\Ker N_\nu^{c-1}$. The first has dimension $1$
and the second has codimension $1$ in $V_\nu$, because $A$ has exactly one
Jordan block of size $c$ at $\nu$. By Theorem~\ref{structure}(c) and
Lemma~\ref{DU} there are polynomials $f_i$ with
$\varepsilon Y_{\nu_{i+1}\nu_{i+1}}=f_i(N_{\nu_{i+1}})$ for
$0\le i\le q-2$ and $Y_{\nu_i\nu_i}=f_i(\omega N_{\nu_i})$ for
$1\le i\le q-1$. Hence $Y_{\nu_1\nu_1}$ lies in $\F[N_{\nu_1}]$, which
has dimension $c-1$, and it determines $f_0$ modulo $x^{c-1}$, then
$f_1,\ldots,f_{q-1}$ modulo $x^{c-1}$ and $Y_{\nu_i\nu_i}$ for
$2\le i\le q-1$. Suppose that $Y_{\nu_1\nu_1}=O$. Then $x^{c-1}$ divides
$f_0$ and $f_{q-1}$. By
Theorem~\ref{structure}(c) for $\mu=\nu_{q-1}$ the map $Y_{\nu\nu}$ is
zero on $\Ker N_\nu^{c-1}$, and by the same statement for $\mu=\nu$ it
maps $V_\nu$ into $N_\nu^{c-1}(V_\nu)$. The space of such maps has
dimension $1$. Therefore the family of the blocks $Y_{\nu_i\nu_i}$ with
$0\le i<q$ lies in a space of dimension at most $(c-1)+1=c$. This proves \eqref{dimbound}, and hence
$\cC_\varepsilon(\cC_\omega(A))=A^j\F[A^q]\subseteq\F[A]$.
\end{proof}

\begin{remark}
\label{abfield}
Conditions (a) and (b) of Theorem~\ref{mainA} can be stated over $\F$.
Since $a(\mu)$ is the multiplicity of $\mu$ as a root of $m_A$ for every
$\mu\in\Fbar^\times$, Lemma~\ref{balroots}(a) shows that condition (a) holds if and only if
$m_A$ is balanced with respect to $\omega$. Condition (b) holds if and
only if $m_A(x)=(x-\nu)h(x)$ for some $\nu\in\F^\times$ and some
polynomial $h\in\F[x]$ that is balanced with respect to $\omega$ and is
divisible by the second invariant factor of $A$ if $A$ has more than one
invariant factor.

Indeed, let $\nu\in\F^\times$ be a root of $m_A$, and put
$h=m_A/(x-\nu)$. The roots of $h$ have the same multiplicities as in
$m_A$, except that $\nu$ has multiplicity $a(\nu)-1$. By Lemma~\ref{balroots}(a), a comparison of these multiplicities at $\mu$
and at $\omega\mu$ for $\mu=\nu$, for $\mu=\omega^{-1}\nu$, and for the
other elements $\mu\in\Fbar^\times$ shows that $h$ is balanced with respect to
$\omega$ if and only if $a(\omega\nu)=a(\omega^{-1}\nu)=a(\nu)-1$ and
$a(\omega\mu)=a(\mu)$ for all
$\mu\in\Fbar^\times\setminus\{\nu,\omega^{-1}\nu\}$. The multiplicity of
$\nu$ as a root of the second invariant factor of $A$ is the size of the
second largest Jordan block of $A$ at $\nu$, or $0$ if there is only one
block. Hence the second invariant factor divides $h$ if and only if $A$
has exactly one Jordan block of size $a(\nu)$ at $\nu$. This proves the
equivalence for the elements $\nu\in\F^\times$.

It remains to show that the element $\nu$ of condition (b) lies in $\F$.
Let $g$ be the minimal polynomial of $\nu$ over $\F$, and let $s$ be the
multiplicity of $\nu$ as a root of $g$. The polynomial $g^\omega$ is
irreducible, and by Lemma~\ref{balroots}(a) it is the minimal polynomial
of $\omega\nu$ over $\F$, and $\omega\nu$ is a root of $g^\omega$ of
multiplicity $s$. The numbers $a(\nu)$ and $a(\omega\nu)$ are the products
of $s$ and the exponents of $g$ and $g^\omega$ in the factorization of
$m_A$ into irreducible polynomials. Since $a(\nu)-a(\omega\nu)=1$ by (b),
the number $s$ divides $1$, so that $s=1$ and the exponent of $g$ exceeds
that of $g^\omega$ by $1$. If $\nu'$ is a root of $g$ of multiplicity
$s'$, then in the same way $a(\nu')-a(\omega\nu')=s'>0$. By (b)
we have $a(\omega\mu)=a(\mu)$ for
$\mu\in\Fbar^\times\setminus\{\nu,\omega^{-1}\nu\}$ and
$a(\nu)>a(\omega^{-1}\nu)$, so that $a(\omega\mu)\ge a(\mu)$ for all
$\mu\in\Fbar^\times\setminus\{\nu\}$. Hence $\nu'=\nu$, so that
$g=x-\nu$ and $\nu\in\F$.
\end{remark}

The next proposition treats the scalars $\varepsilon$ that are not powers
of $\omega$.

\begin{proposition}
\label{nonpow}
Let $\F$ be a field, let $\omega\in\F\setminus\{0,1\}$, let
$A\in M_n(\F)$, and let $\varepsilon\in\F^\times$ be an element that is
not a power of $\omega$. Then
$\cC_\varepsilon(\cC_\omega(A))\subseteq\F[A]$ if and only if $A$
satisfies condition (a) of Theorem~\ref{mainA}, or $A$ satisfies
condition (b) of Theorem~\ref{mainA} and, in addition, $a(\nu)=1$ or $A$
has a Jordan block of size $a(\nu)-1$ at $\nu$. If $A$ satisfies
condition (a), then $\cC_\varepsilon(\cC_\omega(A))=\{O\}$.
\end{proposition}

\begin{proof}
By Lemma~\ref{scalars} we may assume that $\F$ is algebraically closed.
We use the notation of Section~\ref{dgcsec}. If
$\cC_\varepsilon(\cC_\omega(A))\subseteq\F[A]$, then
$\mathcal J(A)\subseteq\F[A]$ by Corollary~\ref{annih}, and condition
(a) or condition (b) holds by Lemma~\ref{Jpoly}. If (a) holds, then $m_A$ is
balanced with respect to $\omega$ by Remark~\ref{abfield}, and by
Corollary~\ref{thompcor} the set $\cC_\varepsilon(\cC_\omega(A))$ is
spanned by the powers $A^m$ with $m\ge0$ and $\omega^m\varepsilon=1$.
There are no such powers, because $\varepsilon$ is not a power of
$\omega$.

Suppose that (b) holds, put $c=a(\nu)$ and $N=N_\nu$, and let
$Y\in\cC_\varepsilon(\cC_\omega(A))$. Then $\Lambda_+=\Lambda_-=\{\nu\}$
by Lemma~\ref{DU}(c). By
Theorem~\ref{structure}(a) and (b) and Lemma~\ref{DU}, the blocks
$Y_{\sigma\tau}$ with $\sigma\ne\tau$ are zero, and $Y_{00}$ is a
polynomial in $N_0$. Let $\mu\in\Lambda\setminus\{\nu\}$. Since
$\mu\notin\Lambda_+$, we have $a(\omega\mu)\ge a(\mu)\ge1$. Thus
$\omega\mu\in\Lambda$, and Theorem~\ref{structure}(c) and
Lemma~\ref{DU} give $Y_{\mu\mu}=f_\mu(\omega N_\mu)$ for some
$f_\mu\in\F[x]$. By the Chinese
remainder theorem, $Y$ is a polynomial in $A$ if and only if $Y_{\nu\nu}$
is a polynomial in $N$. In particular, if $c=1$, then $Y\in\F[A]$, because $A$ has exactly one
Jordan block of size $c$ at $\nu$, so that $V_\nu$ has dimension $1$.

Let $c\ge2$, let $q$ be the order of $\omega$, and put
$\nu_i=\omega^i\nu$ for $i\ge0$. By (b) we have
$a(\nu_1)=c-1$, and $a(\nu_{i+1})=a(\nu_i)$ for every $i\ge1$ with
$\omega^i\ne1$ and $\omega^{i+1}\ne1$. If $q$ were infinite, this would
give $a(\nu_i)=c-1\ge1$ for all $i\ge1$, and $A$ would have infinitely
many eigenvalues. Hence $q$ is finite, $a(\nu_i)=c-1$ for $1\le i<q$, and
$\nu_q=\nu$. Moreover $\varepsilon^q\ne1$, because the $q$ powers of
$\omega$ are the only roots of $x^q-1$ in $\F$. Write
$V_\nu=\F[N]v\oplus V'$, where $\F[N]v$ is the Jordan block of size $c$
and $V'$ is the sum of the other Jordan blocks at $\nu$, which have size
at most $c-1$. Then
\[
 N^{c-1}(V_\nu)=\F N^{c-1}v,~
 \Ker N^{c-1}=N(\F[N]v)\oplus V',
\]
and $V_\nu=\F v\oplus\Ker N^{c-1}$.

By Theorem~\ref{structure}(c) for $\mu=\nu_i$ with $0\le i<q$ there are
polynomials $f_i$ as in that condition. For $1\le i<q$ the element
$\nu_i$ lies neither in $\Lambda_+$ nor in $\Lambda_-$, because
$a(\nu_{i-1})$ and $a(\nu_{i+1})$ are at least $c-1=a(\nu_i)$. By
Lemma~\ref{DU} we obtain $Y_{\nu_i\nu_i}=f_i(\omega N_{\nu_i})$ and
$\varepsilon Y_{\nu_i\nu_i}=f_{i-1}(N_{\nu_i})$ for $1\le i<q$. As
$N_{\nu_i}$ has nilpotency index $c-1$, this gives
$\varepsilon f_i(\omega x)\equiv f_{i-1}(x)\pmod{x^{c-1}}$ for
$1\le i<q$. By induction on $i$ we have
$f_i(x)\equiv\varepsilon^{-i}f_0(\omega^{-i}x)\pmod{x^{c-1}}$, and for
$i=q-1$ this reads
$f_{q-1}(x)\equiv\varepsilon^{1-q}f_0(\omega x)\pmod{x^{c-1}}$. Since
$a(\nu_{q-1})=c-1$ and every polynomial in $N$ that is divisible by
$x^{c-1}$ is zero on $\Ker N^{c-1}$, the condition for $i=q-1$ gives
\begin{equation}
 Y_{\nu\nu}w=\varepsilon^{-1}f_{q-1}(N)w=\varepsilon^{-q}f_0(\omega N)w
 ~\text{for }w\in\Ker N^{c-1},
 \label{onW}
\end{equation}
and, since $a(\nu_1)=c-1$, the condition for $i=0$ says that the map
$R=Y_{\nu\nu}-f_0(\omega N)$ sends $V_\nu$ into $N^{c-1}(V_\nu)$.

Suppose first that $A$ has a Jordan block of size $c-1$ at $\nu$, and let
$\F[N]u\subseteq V'$ be such a block, so that the polynomials $g$ with
$g(N)u=0$ are the multiples of $x^{c-1}$. Then $u\in\Ker N^{c-1}$, and
by \eqref{onW} we have $Ru=(\varepsilon^{-q}-1)f_0(\omega N)u$. This
vector lies in $N^{c-1}(V_\nu)\subseteq\F[N]v$ and in $V'$, so that it is
zero. Since $\varepsilon^q\ne1$, we obtain $f_0(\omega N)u=0$, and
therefore $x^{c-1}$ divides $f_0(\omega x)$. Hence $f_0(\omega N)$ maps
$V_\nu$ into $N^{c-1}(V_\nu)$ and is zero on $\Ker N^{c-1}$. By \eqref{onW} and the inclusion $R(V_\nu)\subseteq N^{c-1}(V_\nu)$, the
map $Y_{\nu\nu}$ is zero on
$\Ker N^{c-1}$ and maps $V_\nu$ into $N^{c-1}(V_\nu)$. The map $N^{c-1}$
has the same two properties and sends $v$ to $N^{c-1}v\ne0$. Since
$V_\nu=\F v\oplus\Ker N^{c-1}$ and $N^{c-1}(V_\nu)=\F N^{c-1}v$, the map
$Y_{\nu\nu}$ is a scalar multiple of $N^{c-1}$. Thus $Y_{\nu\nu}\in\F[N]$, and
$Y\in\F[A]$.

Suppose now that $A$ has no Jordan block of size $c-1$ at $\nu$. Then the Jordan blocks in $V'$ have size at most $c-2$, so that $N^{c-2}$ is zero on $V'$
and maps $\Ker N^{c-1}$ into $N^{c-1}(\F[N]v)=N^{c-1}(V_\nu)$. Let $Z$
be the linear map of $V_\nu$ with $Zv=0$ that is equal to $N^{c-2}$ on
$\Ker N^{c-1}$, and let $Y$
be the linear map of $V$ with the diagonal blocks
\[
 Y_{\nu\nu}=\omega^{c-2}\bigl(N^{c-2}+(\varepsilon^{-q}-1)Z\bigr),~
 Y_{\nu_i\nu_i}=\varepsilon^{-i}\omega^{(1-i)(c-2)}N_{\nu_i}^{c-2}
 ~\text{for }1\le i<q,
\]
and all other blocks zero. Put
$f_i(x)=\varepsilon^{-i}\omega^{-i(c-2)}x^{c-2}$ for $0\le i<q$, and use
the zero polynomial in condition (c) of Theorem~\ref{structure} for the
elements of $\Lambda$ outside $\{\nu_0,\ldots,\nu_{q-1}\}$. Then $Y$
satisfies conditions (a), (b), and (c) of Theorem~\ref{structure}. Indeed,
$Y_{\nu\nu}-f_0(\omega N)=\omega^{c-2}(\varepsilon^{-q}-1)Z$ maps $V_\nu$
into $N^{c-1}(V_\nu)$, we have $Y_{\nu_i\nu_i}=f_i(\omega N_{\nu_i})$ for
$1\le i<q$ and
$\varepsilon Y_{\nu_{i+1}\nu_{i+1}}=f_i(N_{\nu_{i+1}})$ for
$0\le i\le q-2$, and on $\Ker N^{c-1}$ the map $\varepsilon Y_{\nu\nu}$ is equal
to $\varepsilon^{1-q}\omega^{c-2}N^{c-2}$, which is $f_{q-1}(N)$ because
$\omega^{1-q}=\omega$. Hence $Y\in\cC_\varepsilon(\cC_\omega(A))$. Since
$Zv=0$ and $Nv\in\Ker N^{c-1}$, we have
$NY_{\nu\nu}v=\omega^{c-2}N^{c-1}v$ and
$Y_{\nu\nu}Nv=\varepsilon^{-q}\omega^{c-2}N^{c-1}v$. As $\varepsilon^q\ne1$,
the map $Y_{\nu\nu}$ does not commute with $N$. Thus
$Y_{\nu\nu}\notin\F[N]$, and $Y\notin\F[A]$.
\end{proof}

For instance, let the characteristic of $\F$ be different from $2$, let
$\omega=-1$, and let $\varepsilon\in\F\setminus\{0,1,-1\}$. The matrices
$J_2(1)\oplus(-1)$ and $J_2(1)\oplus(1)\oplus(-1)$ both satisfy condition
(b) with $\nu=1$ and $a(\nu)=2$. By Proposition~\ref{nonpow}, the
inclusion $\cC_\varepsilon(\CL(A))\subseteq\F[A]$ holds for
$A=J_2(1)\oplus(1)\oplus(-1)$ and fails for $A=J_2(1)\oplus(-1)$.

The next two corollaries answer Questions~\ref{Q1} and~\ref{Q2} for
$\omega\ne1$.  For $\omega=1$ the inclusion of
Question~\ref{Q1} holds for every $A$ by Theorem~\ref{dct}, and no matrix
$B$ satisfies the assumption of Question~\ref{Q2}, because then
$\omega\beta=\beta$ for every eigenvalue $\beta$ of $B$.
Remark~\ref{q25} concerns Question~\ref{Q4}, and Proposition~\ref{q24}
and Remark~\ref{q24rem} concern Question~\ref{Q3}.

\begin{corollary}
\label{q22}
Let $\F$ be algebraically closed, let $\omega\in\F$ be a root of unity
of order $q\ge2$, and let $A\in M_n(\F)$ be a matrix whose spectrum is
closed under multiplication by $\omega$. Then
$\cC_\omega(\cC_\omega(A))\subseteq\F[A]$ if and only if the largest
Jordan blocks of $A$ at $\mu$ and at $\omega\mu$ have the same size for
every nonzero eigenvalue $\mu$, or there is a nonzero eigenvalue $\nu$
with a single Jordan block of the largest size $c\ge2$ such that the
largest Jordan blocks at $\omega\nu,\ldots,\omega^{q-1}\nu$ have size
$c-1$ and the largest Jordan blocks at $\mu$ and at $\omega\mu$ have the
same size for every nonzero eigenvalue $\mu$ outside
$\{\nu,\omega\nu,\ldots,\omega^{q-1}\nu\}$. In both cases
$\cC_\omega(\cC_\omega(A))=A^{q-1}\F[A^q]$ and
$\cC_{\omega^{-1}}(\cC_\omega(A))=A\F[A^q]$.
\end{corollary}

\begin{proof}
This is Theorem~\ref{mainA} for $\varepsilon=\omega=\omega^{-(q-1)}$ and
for $\varepsilon=\omega^{-1}$. Since the spectrum of $A$ is closed under
multiplication by $\omega$ and by $\omega^{-1}=\omega^{q-1}$, an element
$\mu\in\F^\times$ is an eigenvalue of $A$ if and only if $\omega\mu$ is
one, so that $a(\omega\mu)=a(\mu)=0$ for every $\mu\in\F^\times$ that is
not an eigenvalue. Hence condition (a) of that theorem is the first
condition of the corollary. In condition (b) of that theorem we have
$a(\nu)\ge2$, because $\omega\nu$ is an eigenvalue of $A$, and the
values $a(\omega^i\nu)$ with $1\le i\le q-1$ are all equal to
$a(\omega\nu)=a(\nu)-1$. Conversely, the second condition of the
corollary gives $a(\omega^{i+1}\nu)=a(\omega^i\nu)$ for $1\le i\le q-2$,
and therefore condition (b).
\end{proof}

\begin{corollary}
\label{q23}
Let $\F$ be algebraically closed and let $\omega\in\F$ be a root of
unity of order $q\ge2$. Let $A\in M_n(\F)$ be a matrix whose spectrum is
closed under multiplication by $\omega$ and for which
$\cC_\omega(\cC_\omega(A))\subseteq\F[A]$. Let $B\in M_m(\F)$ with
$m\ge1$ be a matrix that has no eigenvalue in common with $A$ and such
that, for every eigenvalue $\beta$ of $B$, the element $\omega\beta$ is
not an eigenvalue of $B$. Then
$\cC_\omega(\cC_\omega(A\oplus B))\subseteq\F[A]\oplus M_m(\F)$ if and
only if the largest Jordan blocks of $A$ at $\mu$ and at $\omega\mu$
have the same size for every nonzero eigenvalue $\mu$ of $A$.
\end{corollary}

\begin{proof}
Put $G=A\oplus B$, and let $\beta$ be an eigenvalue of $B$. Then
$\beta\ne0$, since $\omega\cdot0=0$. The elements $\omega\beta$ and
$\omega^{-1}\beta$ are not eigenvalues of $B$, by the assumption on $B$
applied to $\beta$ and to $\omega^{-1}\beta$. They are not eigenvalues
of $A$, because the spectrum of $A$ is closed under multiplication by
$\omega$ and by $\omega^{-1}=\omega^{q-1}$ and does not contain $\beta$.
Hence, in the notation of Section~\ref{dgcsec} for the matrix $G$, we
have $a(\omega\beta)=a(\omega^{-1}\beta)=0$, so that
$\Ker N_\beta^{a(\omega^{-1}\beta)}=\{0\}$ by Lemma~\ref{DU}. For a
nonzero eigenvalue
$\mu$ of $A$, the numbers $a(\omega\mu)$ and $a(\omega^{-1}\mu)$ are the
same for $G$ and for $A$, because $\omega\mu$ and $\omega^{-1}\mu$ are
eigenvalues of $A$ and therefore not eigenvalues of $B$.

By Theorem~\ref{mainA}, the matrix $A$ satisfies condition (a) or
condition (b) of that theorem. Suppose that it satisfies (a). Then no
nonzero eigenvalue of $A$ lies in the sets $\Lambda_+$ and $\Lambda_-$
formed for $G$. Let $Y\in\cC_\omega(\cC_\omega(G))$. By
Theorem~\ref{structure}(a) and (b) for $G$ and Lemma~\ref{DU}, every
block of $Y$ between a
generalized eigenspace of $A$ and one of $B$ is zero. Hence
$Y=Y_1\oplus Y_2$ with $Y_2\in M_m(\F)$, and $Y_1$ satisfies the three
conditions of Theorem~\ref{structure} for $A$, so that
$Y_1\in\cC_\omega(\cC_\omega(A))\subseteq\F[A]$. Suppose now that $A$
satisfies (b), and let $\nu$ be the eigenvalue of that condition. Then
$\nu$ lies in the set $\Lambda_+$ formed for $G$, so that
$N_\nu^{a(\omega\nu)}(V_\nu)\ne\{0\}$ by Lemma~\ref{DU}. Let $Z$ be a
nonzero linear map that maps $V_\beta$ into $N_\nu^{a(\omega\nu)}(V_\nu)$
and is zero on all other generalized eigenspaces of $G$. It satisfies
the conditions of Theorem~\ref{structure} for $G$, with $f_\mu=0$ for all
$\mu$, because $\Ker N_\beta^{a(\omega^{-1}\beta)}=\{0\}$. Thus $Z\in\cC_\omega(\cC_\omega(G))$, and $Z$
does not lie in $\F[A]\oplus M_m(\F)$.
\end{proof}

\begin{example}
\label{q23ex}
Let $\omega=-1$, and suppose that the characteristic of $\F$ is neither
$2$ nor $3$. The matrix $A=J_2(1)\oplus(-1)\in M_3(\F)$ satisfies
condition (b) of Theorem~\ref{mainA} with $\nu=1$, so that
$\cC_{-1}(\cC_{-1}(A))=A\F[A^2]$, and its spectrum $\{1,-1\}$ is closed
under multiplication by $-1$. The matrix $B=(2)\in M_1(\F)$ has no
eigenvalue in common with $A$, and $-2$ is not an eigenvalue of $B$. For
$G=A\oplus B\in M_4(\F)$ we find that $\cC_{-1}(G)$ is spanned by $E_{32}$ and $E_{13}$. The matrix
$E_{14}$ satisfies $E_{14}X=XE_{14}=O$ for $X\in\{E_{32},E_{13}\}$.
Hence $E_{14}\in\cC_{-1}(\cC_{-1}(G))$, and $E_{14}$ does not lie in
$\F[A]\oplus M_1(\F)$. Thus the answer to Question~\ref{Q2} is negative
in general.
\end{example}

\begin{remark}
\label{q25}
Question~\ref{Q4} asks for the matrices $A$ with
$\cC_\omega(\cC_\omega(A))=\F[A]$. For $\omega=1$ the equality holds for
every $A$ by Theorem~\ref{dct}. For $n\ge2$ and $\omega\ne1$ there is no
such matrix in $M_n(\F)$. Indeed, if $I\in\cC_\omega(\cC_\omega(A))$,
then $X=\omega X$ for all $X\in\cC_\omega(A)$, so that
$\cC_\omega(A)=\{O\}$ and $\cC_\omega(\cC_\omega(A))=M_n(\F)$, which has
dimension $n^2>n\ge\dim\F[A]$. For $n=1$ and $\omega\ne1$ the equality
holds exactly for $A\ne O$.
\end{remark}

\begin{proposition}
\label{q24}
Let $\F$ be algebraically closed, let $\omega\in\F$ be a root of unity
of order $q\ge2$, let $\lambda\in\F^\times$, and let $A=D+N\in M_n(\F)$
be a matrix in Jordan form with diagonal part $D$ and nilpotent part $N$
whose set of eigenvalues is
$\{\lambda,\omega\lambda,\ldots,\omega^{q-1}\lambda\}$.
\begin{enumerate}
\item[(a)] The matrix $I$ does not lie in $\cC_\omega(\cC_\omega(A))$.
For all integers $r\ge0$ and $s$ the matrix $D^sN^r$ lies in
$\cC_{\omega^{-s-r}}(\cC_\omega(A))$.
\item[(b)] Suppose that the largest Jordan blocks of $A$ at all
eigenvalues have the same size $c$. Then the matrices $D^{1-r}N^r$ with
$0\le r<c$ form a basis of $\cC_{\omega^{-1}}(\cC_\omega(A))$, and the
matrices $D^{-1-r}N^r$ with $0\le r<c$ form a basis of
$\cC_\omega(\cC_\omega(A))$.
\end{enumerate}
\end{proposition}

\begin{proof}
(a) Since $\lambda$ and $\omega\lambda$ are eigenvalues of $A$,
Lemmas~\ref{decomp} and~\ref{intertw}(a) give $\cC_\omega(A)\ne\{O\}$,
and $I\notin\cC_\omega(\cC_\omega(A))$ as in Remark~\ref{q25}. Let $X\in\cC_\omega(A)$. By Lemma~\ref{decomp}, $X$
maps $V_\mu$ into $V_{\omega\mu}$ for every eigenvalue $\mu$, and
$NX=\omega XN$. Since $D$ is the multiplication by $\mu$ on $V_\mu$ for
every eigenvalue $\mu$, we obtain $DX=\omega XD$. Hence
$XD^sN^r=\omega^{-s-r}D^sN^rX$.

(b) By Theorem~\ref{mainA} we have
$\cC_{\omega^{-1}}(\cC_\omega(A))=A\F[A^q]$ and
$\cC_\omega(\cC_\omega(A))=A^{q-1}\F[A^q]$, and these spaces have
dimension $c$ by Lemma~\ref{dimpow}. By (a) they contain the matrices $D^{1-r}N^r$ and
$D^{-1-r}N^r$, respectively. For fixed integers $s_0,\ldots,s_{c-1}$ the
matrices
$D^{s_r}N^r$ with $0\le r<c$ are linearly independent: apply a linear
combination of them to a vector $v\in V_\lambda$ with $N^{c-1}v\ne0$ and use
the fact that $v,Nv,\ldots,N^{c-1}v$ are linearly independent.
\end{proof}

\begin{remark}
\label{q24rem}
Question~\ref{Q3} asks for which sizes of the Jordan blocks of a matrix
$A$ as in Proposition~\ref{q24} the equality
$\cC_\omega(\cC_\omega(A))=\F D+\F[N]$ holds. By
Proposition~\ref{q24}(a) this equality never holds, because
$I\in\F[N]$. If $\F[N]$ is replaced by $\F N$, then by
Proposition~\ref{q24}(b), under the assumption made there,
$\cC_{\omega^{-1}}(\cC_\omega(A))=\F D+\F N$ if and only if
$c\le2$, and for $q=2$ the same is true for
$\cC_\omega(\cC_\omega(A))$. For $q\ge3$ the matrix $D$ does not lie in
$\cC_\omega(\cC_\omega(A))$, because $XD=\omega DX$ and $DX=\omega XD$
for a nonzero $X\in\cC_\omega(A)$ give $\omega^2=1$.
\end{remark}

The description of the double generalized centralizers gives
Theorem~\ref{mainB} on $\omega$\nobreakdash-equivalence, which extends Corollary~\ref{simeq}(a) and
Corollary~\ref{clcor}(b) and~(c).

\begin{proof}[Proof of Theorem~\ref{mainB}]
(a) By Lemma~\ref{basic}(b), the inclusion
$\cC_\omega(A)\subseteq\cC_\omega(B)$ holds if and only if
$B\in\cC_{\omega^{-1}}(\cC_\omega(A))$. By Theorem~\ref{mainA} with
$j=1$ this set is contained in $\F[A]$, and it is $A\F[A^q]$ if $q$ is
finite. Let $q$ be finite. If $A$ and $B$ are
$q$\nobreakdash-polynomially equivalent, then
$\cC_\omega(A)\subseteq\cC_\omega(B)$, and
$A\in B\F[B^q]\subseteq\cC_{\omega^{-1}}(\cC_\omega(B))$ by
Lemma~\ref{basic}(c), so that $\cC_\omega(B)\subseteq\cC_\omega(A)$ by
Lemma~\ref{basic}(b). Conversely, let $\cC_\omega(A)=\cC_\omega(B)$. Then
$B\in A\F[A^q]$, and it remains to show that $A\in B\F[B^q]$. For this
it suffices to prove that $B$ satisfies condition (ii) of
Theorem~\ref{mainA}, because then the first part applies with $A$ and
$B$ exchanged. By Lemma~\ref{scalars} we may assume that $\F$ is
algebraically closed. We use the notation of Section~\ref{dgcsec}
for $A$, and we mark the corresponding objects for $B$ by an upper index
$B$. The space
$\mathcal J(A)$ of Corollary~\ref{annih} depends only on
$\cC_\omega(A)$, so that $\mathcal J(B)=\mathcal J(A)$.

If $A$ satisfies (a), then $\mathcal J(A)=\{O\}$, and one of the two
factors in the dimension formula of Corollary~\ref{annih} for $B$ is
zero. By Lemma~\ref{DU}, the terms of the first factor are positive
exactly for $\sigma\in\Lambda_+^B$ and those of the second exactly for
$\tau\in\Lambda_-^B$, so that $\Lambda_+^B$ or $\Lambda_-^B$ is empty,
and $B$ satisfies (a) by Lemma~\ref{DU}(a). Let $A$ satisfy (b) with
the eigenvalue $\nu$, and put $c=a(\nu)$. Then $\Lambda_+=\Lambda_-=\{\nu\}$
by Lemma~\ref{DU}(c), and $a(\omega\nu)=a(\omega^{-1}\nu)=c-1$. Hence the space
$N_\nu^{a(\omega\nu)}(V_\nu)=N_\nu^{c-1}(V_\nu)$ has dimension $1$, the
space $\Ker N_\nu^{a(\omega^{-1}\nu)}=\Ker N_\nu^{c-1}$ has codimension
$1$ in $V_\nu$, and by Corollary~\ref{annih} and Lemma~\ref{DU} the space
$\mathcal J(A)$ is spanned by the map $Z_0$ that is $N_\nu^{c-1}$ on $V_\nu$ and zero on the other
generalized eigenspaces of $A$. This map is a polynomial in $A$. Since
$B\in\F[A]$, the map $Z_0$ commutes with $B$ and maps every generalized
eigenspace of $B$ into itself. As $\dim\mathcal J(B)=1$, both factors in
the dimension formula of Corollary~\ref{annih} for $B$ are $1$, and
Lemma~\ref{DU} shows that $\Lambda_+^B=\{\sigma\}$ and
$\Lambda_-^B=\{\tau\}$ for some $\sigma$ and $\tau$. By the same
corollary, the nonzero map $Z_0\in\mathcal J(B)$ is zero on the
generalized eigenspaces of $B$ other than $V^B_\tau$ and maps $V^B_\tau$
into $V^B_\sigma$. Hence $\sigma=\tau$. Let $N=N^B_\tau$ and
$c'=a^B(\tau)$. The spaces $N^k(V^B_\tau)$ with $0\le k\le c'$ decrease
strictly, and $N^{c'-1}(V^B_\tau)\ne\{0\}$. Since $\tau\in\Lambda_+^B$,
the number $a^B(\omega\tau)$ is smaller than $c'$, and the image of $V^B_\tau$ under $N^{a^B(\omega\tau)}$ has dimension $1$. Hence
$a^B(\omega\tau)=c'-1$, and $N^{c'-1}$ has rank $1$, so that $B$ has
exactly one Jordan block of size $c'$ at $\tau$. In the same way the
kernels of $N^k$ increase strictly for $0\le k\le c'$. Since
$\tau\in\Lambda_-^B$ and the kernel of $N^{a^B(\omega^{-1}\tau)}$ has codimension $1$ in
$V^B_\tau$, we obtain $a^B(\omega^{-1}\tau)=c'-1$. With Lemma~\ref{DU}(b) we see that
$B$ satisfies (b).

(b) By Theorem~\ref{mainA} with $j=1$ there is a matrix
$Y\in\cC_{\omega^{-1}}(\cC_\omega(A))$ that is not a polynomial in $A$.
For $t\in\F$ put $B_t=A+tY$. Since $A$ and $Y$ lie in
$\cC_{\omega^{-1}}(\cC_\omega(A))$, we have
$\cC_\omega(A)\subseteq\cC_\omega(B_t)$ for all $t$ by
Lemma~\ref{basic}(b). By Lemma~\ref{generic} with $\Omega=\{\omega\}$,
there is $t\ne0$ with
$\dim\cC_\omega(B_t)\le\dim\cC_\omega(A)$, and hence
$\cC_\omega(B_t)=\cC_\omega(A)$. The matrix $B_t$ is not a polynomial in
$A$, because $t\ne0$ and $Y\notin\F[A]$.
\end{proof}

\begin{example}
\label{clex}
Suppose that the characteristic of $\F$ is not $2$, let
$\lambda\in\F^\times$, and let
$A=J_3(\lambda)\oplus(-J_4(\lambda))\in M_7(\F)$. The matrix
$-J_4(\lambda)$ is similar to $J_4(-\lambda)$. Hence $a(\lambda)=3$ and
$a(-\lambda)=4$, and $A$ has exactly one Jordan block of size $4$ at
$-\lambda$. Thus $A$ satisfies condition (b) of Theorem~\ref{mainA} with
$\omega=-1$ and $\nu=-\lambda$, while $A$ is not balanced. By Theorem~\ref{mainB}(a), a matrix
$B\in M_7(\F)$ satisfies $\CL(A)=\CL(B)$ if and only if $A$ and $B$ are
odd polynomially equivalent, and by
Theorem~\ref{mainA}, $\cC_{-1}(\cC_{-1}(A))=A\F[A^2]$.
\end{example}

\begin{example}
\label{adjex}
Let $n\ge2$, and let $A\in M_n(\mathbb R)$ be the adjacency matrix of a
graph with $n$ vertices in which no vertex is joined to itself and two
vertices are joined by at most one edge. The rows and the columns of $A$
are indexed by the vertices, and the entry at a pair of vertices is $1$
if they are joined by an edge and $0$ otherwise. Since $A$ is symmetric, it is diagonalizable,
and all its Jordan blocks have size $1$. Hence, for $\omega=-1$,
condition (a) of Theorem~\ref{mainA} says that the set of the nonzero
eigenvalues of $A$ is closed under multiplication by $-1$, and condition
(b) says that this set has this property after the removal of one
eigenvalue $\nu$ of multiplicity $1$ for which $-\nu$ is not an
eigenvalue. By Theorem~\ref{mainB}, every matrix $B\in M_n(\mathbb R)$
with $\CL(B)=\CL(A)$ is odd polynomially equivalent to $A$ if and only if
$A$ satisfies (a) or (b). Suppose that the vertices split into two sets
such that every edge joins the two sets. Then the diagonal matrix with
the entry $1$ at the vertices of the first set and $-1$ at the vertices
of the second set is its own inverse, and multiplying $A$ by it on both
sides gives $-A$. Hence $A$ is similar to $-A$, and (a) holds. The rook's
graph on a $4\times4$ board has the $16$ squares as vertices, and two
squares are joined by an edge if they lie in the same row or in the same
column. Its adjacency matrix is $Q\otimes I_4+I_4\otimes Q$, where $Q$
is the adjacency matrix of the complete graph on four vertices, in which
any two distinct vertices are joined by an edge. The eigenvalues of $Q$
are $3$ and $-1$. If $Qu=\alpha u$ and $Qu'=\alpha'u'$, then
$(Q\otimes I_4+I_4\otimes Q)(u\otimes u')=(\alpha+\alpha')(u\otimes u')$,
and these vectors span $\mathbb R^{16}$ when $u$ and $u'$ run through a
basis of eigenvectors of $Q$. Hence the eigenvalues are $6$,
$2$, and $-2$, the first with multiplicity $1$, and (b) holds with
$\nu=6$. Since $-6$ is not an eigenvalue, (a) fails, and therefore the
vertices of the rook's graph cannot be split into two sets such that
every edge joins the two sets. The adjacency matrix of the
complete graph on three vertices has the eigenvalues $2$, $-1$, and $-1$,
and it satisfies neither condition.
\end{example}

\begin{corollary}
\label{singular}
Let $\omega\in\F\setminus\{0,1\}$ and $A,B\in M_n(\F)$ with
$\cC_\omega(A)\subseteq\cC_\omega(B)$. If $A$ is singular, then $B$ is
singular.
\end{corollary}

\begin{proof}
By Lemma~\ref{scalars} we may assume that $\F$ is algebraically closed.
By Lemma~\ref{basic}(b) we have
$B\in\cC_{\omega^{-1}}(\cC_\omega(A))$. By Theorem~\ref{structure}(a)
with $\varepsilon=\omega^{-1}$, the matrix $B$ maps $V_0$ into itself,
and its restriction to $V_0$ lies in the span of the powers $N_0^m$ with
$\omega^{m-1}=1$. The exponents $m$ of these powers are at least $1$, because
$\omega^{-1}\ne1$. Hence this restriction is nilpotent. Since $V_0\ne\{0\}$, the matrix $B$ is singular.
\end{proof}

\begin{example}
\label{dgkm14}
Let $\omega$ be a root of unity of order $3$ in an algebraically closed
field $\F$, and let
$A=J_3(1)\oplus J_2(\omega)\oplus J_1(\omega^2)\in M_6(\F)$, which is
Example~14 of Dolinar et~al.~\cite{DGKM2019}. Here $a(1)=3$, $a(\omega)=2$, and
$a(\omega^2)=1$, so that neither condition of Theorem~\ref{mainA} holds.
With the standard basis $e_1,\ldots,e_6$ of $\F^6$, the spaces
$N_\sigma^{a(\omega\sigma)}(V_\sigma)$ of Corollary~\ref{annih} are
$N_1^2(V_1)=\F e_1$, $N_\omega(V_\omega)=\F e_4$, and
$N_{\omega^2}^3(V_{\omega^2})=\{0\}$, and the spaces
$\Ker N_\tau^{a(\omega^{-1}\tau)}$ are $\Ker N_1=\F e_1$,
$\Ker N_\omega^3=V_\omega$, and $\Ker N_{\omega^2}^2=V_{\omega^2}$. By
Corollary~\ref{annih}, the space $\mathcal J(A)$ is spanned by
$E_{12},E_{13},E_{42},E_{43}$. Theorem~\ref{structure} with
$\varepsilon=\omega$ shows that $\cC_\omega(\cC_\omega(A))$ is the
direct sum of $\mathcal J(A)$ and of the space of the matrices
\[
 y(\omega I_3\oplus I_2\oplus\omega^2I_1)
 +y'(\omega^2E_{23}+E_{45})
\]
with $y,y'\in\F$. Dolinar et~al.~\cite[Example~14]{DGKM2019} gave the
same description. The matrices $E_{42}$ and $E_{43}$ are
not polynomials in $A$.
\end{example}

Figure~\ref{annihfig} illustrates Corollary~\ref{annih} for the matrix of
Example~\ref{dgkm14}. A dot is one of the vectors $e_1,\ldots,e_6$, the
dots of one generalized eigenspace form a column, and the gray arrows
show the maps $N_\mu$. The space $\cC_\omega(A)$ has the basis
$E_{53}+\omega^{-1}E_{42}$, $E_{43}$, $E_{65}$, $E_{16}$, as we find by comparing entries. In the left panel a black arrow from a vector ends at
the image of this vector under one of these four matrices, up to a
nonzero factor. The dashed circles mark the vectors $e_1$ and $e_4$,
which span the intersection of the kernels of the elements of
$\cC_\omega(A)$, and the dash dotted circles mark the vectors $e_1$,
$e_4$, $e_5$, and $e_6$, which span the sum of their images. The dotted
arrows in the right panel show the basis $E_{12}$, $E_{13}$, $E_{42}$,
$E_{43}$ of $\mathcal J(A)$. In agreement with Corollary~\ref{annih}, they
start at $e_2$ and $e_3$, which lie outside the sum of the images, and
end at $e_1$ and $e_4$.

\begin{figure}[H]
\centering
\includegraphics[width=\textwidth]{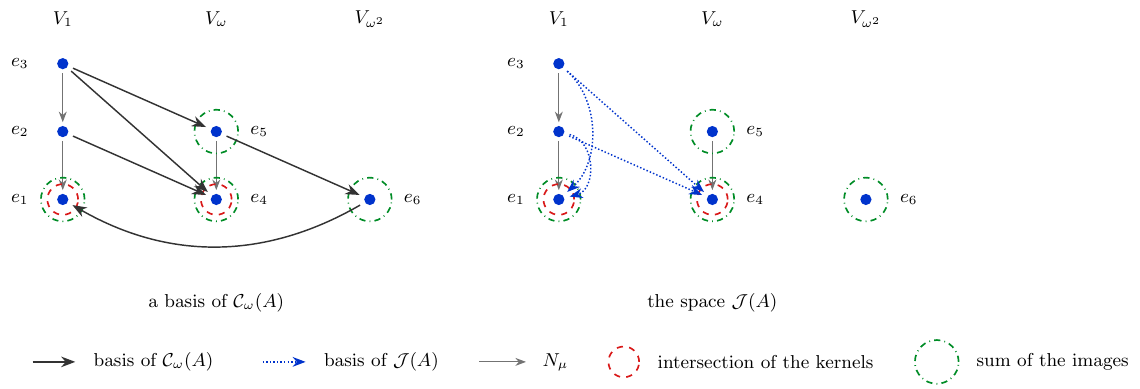}
\caption{The generalized centralizer $\cC_\omega(A)$ and the space
$\mathcal J(A)$ of Corollary~\ref{annih} for the matrix $A$ of
Example~\ref{dgkm14}.}
\label{annihfig}
\end{figure}

\begin{example}
\label{twoex}
Let $\omega\in\F\setminus\{0,1\}$ with $\omega^2\ne1$, and let
$A=\diag(1,\omega)\in M_2(\F)$. Then $\cC_\omega(A)=\F E_{21}$, and
$\cC_{\omega^{-1}}(\cC_\omega(A))$ consists of the matrices
\[
 B=\begin{bmatrix}\omega^{-1}\beta&0\\ \gamma&\beta\end{bmatrix}
\]
with $\beta,\gamma\in\F$. For $\beta\ne0$ the eigenvalues of $B$ are
$\omega^{-1}\beta$ and $\beta$, and $\cC_\omega(B)=\F E_{21}$. Thus the
matrices $B$ with $\beta\ne0$ are the matrices that are
$\omega$\nobreakdash-equivalent to $A$, and those with $\gamma\ne0$ are
not polynomials in $A$. Here $a(1)=a(\omega)=1$ and
$a(\omega^2)=a(\omega^{-1})=0$, so that $A$ does not satisfy condition
(ii) of Theorem~\ref{mainA}. Since $\cC_\omega(cA)=\cC_\omega(A)$ for
$c\in\F^\times$, the same description holds for every nonzero multiple
$cA$ of $A$. For instance, if $\omega$ has order $6$, then the matrices
$\diag(\omega^2,\omega^3)=\omega^2A$ and
$\diag(\omega^3,\omega^4)=\omega^3A$ both have the generalized
centralizer $\F E_{21}$. They are multiples of each other, whereas the
matrices $B$ above with $\beta\gamma\ne0$ have the same generalized
centralizer and are not polynomials in them.
\end{example}

\section{\texorpdfstring{The $\omega$\protect\nobreakdash-commuting digraph}{The ω-commuting digraph}}
\label{graphsec}

For $\omega\in\F\setminus\{0,1\}$ and a vertex $A$ of a directed graph
defined by Raja and Vaezpour~\cite{RajaVaezpour2011}, the nonzero
elements of $\cC_\omega(A)$ and of $\cC_{\omega^{-1}}(A)$ are the
successors and the predecessors of $A$ (see Lemma~\ref{vertex}). We
translate Theorems~\ref{mainA}
and~\ref{mainB} and Proposition~\ref{rankprop} into statements about
these sets. We find a condition (T) on the Jordan form of a vertex $A$
under which all twins of $A$ are polynomials in $A$, show that (T) is
necessary for this over fields with more than $2n^2+1$ elements, and
describe over every field the twins of the vertices that satisfy (T).

\begin{definition}[Raja and Vaezpour {\cite[p.~781]{RajaVaezpour2011}}]
\label{digraphdef}
Let $\omega\in\F$ and $n\ge2$. The $\omega$\nobreakdash-commuting digraph
$\Gamma_\omega(M_n(\F))$ is the directed graph whose vertices are the
matrices $A\in M_n(\F)$ that are not scalar matrices and for which $A$
and $\omega A$ have a common eigenvalue, and which has an edge from a
vertex $A$ to a vertex $B$ if $AB=\omega BA$.
\end{definition}

Raja and Vaezpour called this directed graph the
$\omega$\nobreakdash-commuting graph of $M_n(\F)$. The name
$\omega$\nobreakdash-commuting digraph is due to
Markova~\cite[p.~315]{Markova2026}, who wrote $\Delta_\omega(M_n(\F))$
for it. We take the common eigenvalue in $\Fbar$, as Raja and
Vaezpour~\cite[p.~788]{RajaVaezpour2011} did in the proof of their
Lemma~2.20.  Raja and Vaezpour~\cite[p.~781]{RajaVaezpour2011} did not
say whether there is an edge from a vertex to itself. They noted that
$\Gamma_1(M_n(\F))$ is the commuting graph of $M_n(\F)$, which has the
nonscalar matrices as vertices and an edge between any two distinct
commuting matrices, so that no vertex is joined to itself. We allow such
edges, as Markova~\cite[Definition~2.1]{Markova2026} did, and call them
loops. For $\omega\ne1$ there is a loop at $A$ if and only if $A^2=O$. We use this convention in the last sentence of Lemma~\ref{vertex}, in
Corollary~\ref{nbhd}, and, for vertices $A$ with $A^2=O$, in
Theorem~\ref{twinsdesc} and Corollary~\ref{twinsq}. Theorem~\ref{twins} holds with and without loops,
and so does Theorem~\ref{twinsdesc} for vertices $A$ with $A^2\ne O$
(see Remark~\ref{loopless}).

\begin{definition}
\label{twindef}
Let $\omega\in\F$, let $n\ge2$, and let $A$ be a vertex of
$\Gamma_\omega(M_n(\F))$. A vertex $B$ is a successor of $A$ if there is an edge from $A$ to $B$, and a predecessor
of $A$ if there is an edge from $B$ to $A$. Two vertices are twins if
they have the same successors and the same predecessors.
\end{definition}

For $\omega=1$ every vertex has a loop, so that the successors and the
predecessors of a vertex $A$ are the nonscalar elements of $\cC(A)$. Since every centralizer contains the scalar matrices, two
vertices of $\Gamma_1(M_n(\F))$ are twins if and only if they have the
same centralizer, that is, if and only if each of them is a polynomial in
the other (see Lemma~\ref{polyeq}). Dolinar, Guterman, Kuzma, and
Oblak~\cite[p.~454]{DGKO2014} remarked that the nonscalar elements of
$\cC(A)$ are $A$ and the vertices joined to $A$ in the commuting graph of
$M_n(\F)$.

For $\omega\notin\{0,1,-1\}$, Markova~\cite[p.~229]{MarkovaSEMR2026}
noted that the successors of a vertex $A$ are the nonzero elements of
$\cC_\omega(A)$ and that its predecessors are the nonzero elements of
$\cC_{\omega^{-1}}(A)$. The first assertion of part (a) of
Lemma~\ref{vertex} below follows from a statement of Kurosaki.

\begin{theorem}[Kurosaki {\cite[p.~24]{Kurosaki1941}}]
\label{kurosaki}
Let $A\in M_n(\F)$ and $c\in\F$. The equation $PA=cAP$ has a nonzero
solution $P\in M_n(\F)$ if and only if the characteristic polynomials of
$A$ and $cA$ have a common irreducible factor.
\end{theorem}

\begin{lemma}
\label{vertex}
Let $\omega\in\F\setminus\{0,1\}$, $n\ge2$, and $A\in M_n(\F)$.
\begin{enumerate}
\item[(a)] We have $\cC_\omega(A)\ne\{O\}$ if and only if $A$ and
$\omega A$ have a common eigenvalue in $\Fbar$. Hence $A$ is a vertex of
$\Gamma_\omega(M_n(\F))$ if and only if $A\ne O$ and
$\cC_\omega(A)\ne\{O\}$.
\item[(b)] If $A$ and $B$ are nonzero matrices with $AB=\omega BA$, then
$A$ and $B$ are vertices.
\end{enumerate}
Consequently, the successors of a vertex $A$ are the nonzero elements of
$\cC_\omega(A)$, its predecessors are the nonzero elements of
$\cC_{\omega^{-1}}(A)$, and two vertices $A$ and $B$ are twins if and only
if $\cC_\omega(A)=\cC_\omega(B)$ and
$\cC_{\omega^{-1}}(A)=\cC_{\omega^{-1}}(B)$.
\end{lemma}

\begin{proof}
(a) For $c=\omega^{-1}$ the equation $PA=cAP$ reads $AP=\omega PA$, so
that its solutions $P$ are the elements of $\cC_\omega(A)$. By
Theorem~\ref{kurosaki}, $\cC_\omega(A)\ne\{O\}$ if
and only if $p_A$ and the characteristic polynomial of $\omega^{-1}A$
have a common irreducible factor, that is, if and only if $A$ and
$\omega^{-1}A$ have a common eigenvalue in $\Fbar$. This is the case if
and only if $\omega A$ and $A$ have a common eigenvalue. The second
assertion of (a) holds because, for $\omega\ne1$, a nonzero scalar matrix
$\lambda I$ and $\omega\lambda I$ have no common eigenvalue.

(b) Since $B\in\cC_\omega(A)$ and $B\ne O$, part (a) shows that $A$ is a
vertex. Since $A\in\cC_{\omega^{-1}}(B)$ and $A\ne O$, part (a) for
$\omega^{-1}$ shows that $B$ and $\omega^{-1}B$ have a common eigenvalue
$\lambda$. Then $\omega\lambda$ is an eigenvalue of $\omega B$ and of
$B$, so that $B$ is a vertex.

By (b), every nonzero element of $\cC_\omega(A)$ is a successor of a
vertex $A$, and a matrix $B$ satisfies $BA=\omega AB$ if and only if
$B\in\cC_{\omega^{-1}}(A)$. This gives the last assertions.
\end{proof}

\begin{corollary}
\label{nbhd}
Let $\omega\in\F\setminus\{0,1\}$ have order $q$, let $n\ge2$, and let
$A$ be a vertex of $\Gamma_\omega(M_n(\F))$.
\begin{enumerate}
\item[(a)] Every successor of $A$ is a successor of a vertex $B$ if and
only if $B\in\cC_{\omega^{-1}}(\cC_\omega(A))$. If $q$ is finite and $A$
satisfies condition (ii) of Theorem~\ref{mainA}, then these vertices $B$
are the nonzero elements of $A\F[A^q]$.
\item[(b)] Suppose that $q$ is finite and that $A$ satisfies condition
(ii) of Theorem~\ref{mainA}. Then a vertex $B$ has the same successors
as $A$ if and only if $A$ and $B$ are $q$\nobreakdash-polynomially
equivalent, and the same holds for the predecessors.
\item[(c)] Suppose that $A$ does not satisfy condition (ii) of
Theorem~\ref{mainA} and that $\F$ has more than $n^2+1$ elements. Then
some vertex that is not a polynomial in $A$ has the same successors as
$A$, and some vertex that is not a polynomial in $A$ has the same
predecessors as $A$.
\item[(d)] If $\omega^2\ne1$, then the successors of $A$ are the
predecessors of some vertex if and only if the ranks of $A$ and $A^2$
are equal.
\end{enumerate}
\end{corollary}

\begin{proof}
Condition (ii) of Theorem~\ref{mainA} holds for $\omega$ if and only if
it holds for $\omega^{-1}$, which also has order $q$. Indeed, replacing $\mu$ by $\omega^{-1}\mu$
shows that $a(\omega\mu)=a(\mu)$ for all $\mu\in\Fbar^\times$ if and only
if $a(\omega^{-1}\mu)=a(\mu)$ for all $\mu\in\Fbar^\times$. If condition
(b) holds for $\omega$ with the element $\nu$, then
$a(\omega^{-1}\nu)=a(\omega\nu)=a(\nu)-1$, and for
$\mu\notin\{\nu,\omega\nu\}$ the element $\omega^{-1}\mu$ lies outside
$\{\omega^{-1}\nu,\nu\}$, so that
$a(\mu)=a(\omega\cdot\omega^{-1}\mu)=a(\omega^{-1}\mu)$. Thus condition (b)
holds for $\omega^{-1}$ with the same element $\nu$, and the converse
follows by exchanging $\omega$ and $\omega^{-1}$.

(a) By Lemma~\ref{vertex}, every successor of $A$ is a successor of $B$
if and only if $\cC_\omega(A)\subseteq\cC_\omega(B)$, which means
$B\in\cC_{\omega^{-1}}(\cC_\omega(A))$ by Lemma~\ref{basic}(b). Under the
assumptions of the second sentence this set is $A\F[A^q]$ by
Theorem~\ref{mainA} with $j=1$. A nonzero element $B$ of this set
satisfies $\cC_\omega(B)\supseteq\cC_\omega(A)\ne\{O\}$, and it is a
vertex by Lemma~\ref{vertex}(a).

(b) By Lemma~\ref{vertex}, $A$ and $B$ have the same successors if and
only if $\cC_\omega(A)=\cC_\omega(B)$, and Theorem~\ref{mainB}(a)
applies. They have the same predecessors if and only if
$\cC_{\omega^{-1}}(A)=\cC_{\omega^{-1}}(B)$, and Theorem~\ref{mainB}(a)
applies to $\omega^{-1}$, because $A$ satisfies condition (ii) for
$\omega^{-1}$.

(c) Theorem~\ref{mainB}(b) gives a matrix $B\notin\F[A]$ with
$\cC_\omega(B)=\cC_\omega(A)$. Then $B\ne O$, and $B$ is a vertex by
Lemma~\ref{vertex}(a), because $\cC_\omega(A)\ne\{O\}$. Since $A$ does
not satisfy condition (ii) for $\omega^{-1}$ either,
Theorem~\ref{mainB}(b) for $\omega^{-1}$ gives a matrix $B'\notin\F[A]$
with $\cC_{\omega^{-1}}(B')=\cC_{\omega^{-1}}(A)$. Because $A$ is a
vertex, $A$ and $\omega A$ have a common eigenvalue $\lambda$. Then
$\omega^{-1}\lambda$ is a common eigenvalue of $A$ and $\omega^{-1}A$, so
that
$\cC_{\omega^{-1}}(A)\ne\{O\}$ by Lemma~\ref{vertex}(a) for
$\omega^{-1}$. A nonzero matrix $X$ in this space satisfies
$XB'=\omega B'X$, and $B'\ne O$, so that $B'$ is a vertex by
Lemma~\ref{vertex}(b). It has the same predecessors as $A$.

(d) By Lemma~\ref{vertex}, the successors of $A$ are the predecessors of
a vertex $B$ if and only if $\cC_\omega(A)=\cC_{\omega^{-1}}(B)$. If such a
vertex exists, the ranks of $A$ and $A^2$ are equal by
Proposition~\ref{rankprop}. Conversely, if the ranks are equal,
Proposition~\ref{rankprop} gives a matrix $B$ with
$\cC_\omega(A)=\cC_{\omega^{-1}}(B)$. We have $B\ne O$, because
$\cC_{\omega^{-1}}(O)$ contains $I$, while $I\notin\cC_\omega(A)$ since
$A\ne O$ and $\omega\ne1$. A nonzero matrix $X\in\cC_\omega(A)$ satisfies
$BX=\omega^{-1}XB$, that is, $XB=\omega BX$, so that $B$ is a vertex by
Lemma~\ref{vertex}(b).
\end{proof}

If $\omega=-1$, the relation $AB=\omega BA$ is symmetric, and twins are
vertices with the same successors.  Recall that
$a(\mu)$ denotes the size of the largest Jordan block of $A$ at
$\mu\in\Fbar^\times$, and $a(\mu)=0$ if $\mu$ is not an eigenvalue of
$A$.

\begin{theorem}
\label{twins}
Let $\omega\in\F\setminus\{0,1\}$, let $n\ge2$, let $A$ be a vertex of
$\Gamma_\omega(M_n(\F))$, and consider the following condition.
\begin{enumerate}
\item[(T)] There is at most one $\nu\in\Fbar^\times$ with
$a(\nu)>a(\omega\nu)$ and $a(\nu)>a(\omega^{-1}\nu)$, and if $\nu$ is
such an element, then $A$ has exactly one Jordan block of size $a(\nu)$
at $\nu$, and $\max\{a(\omega\nu),a(\omega^{-1}\nu)\}=a(\nu)-1$.
\end{enumerate}
If $A$ satisfies (T), then every twin of $A$ is a polynomial in $A$. If
$A$ does not satisfy (T) and $\F$ has more than $2n^2+1$ elements, then
$A$ has a twin that is not a polynomial in $A$.
\end{theorem}

We call an element $\nu\in\Fbar^\times$ with $a(\nu)>a(\omega\nu)$ and
$a(\nu)>a(\omega^{-1}\nu)$ an element as in (T). For $\omega$ of order $4$
and a nonzero element $\nu$, Figure~\ref{tfig} shows values of $a$ on the
set $\{\nu,\omega\nu,\omega^2\nu,\omega^3\nu\}$, drawn as in
Figure~\ref{orbitfig}. The elements as
in (T) are the elements marked by both circles. Suppose that $A$ has no
other nonzero eigenvalues. In the left panel, (T) holds if $A$ has a
single Jordan block of size $2$ at $\nu$, while condition (ii) of
Theorem~\ref{mainA} fails, because $a(\omega\mu)\ne a(\mu)$ for all four
elements $\mu$ of the set. In the middle panel there are two elements as
in (T), and in the right panel the only element $\nu$ as in (T) has
$\max\{a(\omega\nu),a(\omega^{-1}\nu)\}<a(\nu)-1$, so that (T) fails in
both.

\begin{figure}[t]
\centering
\includegraphics[width=\textwidth]{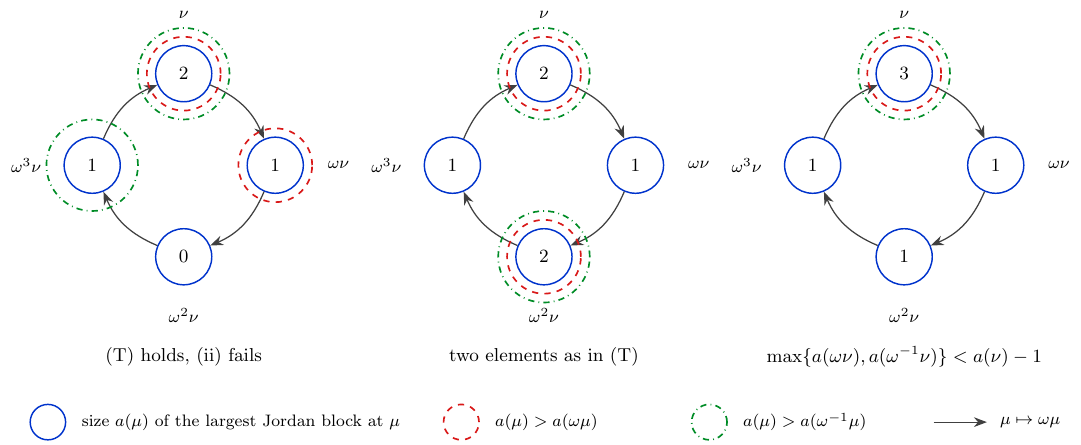}
\caption{Condition (T) of Theorem~\ref{twins} for $\omega$ of order
$4$.}
\label{tfig}
\end{figure}

For $\omega\in\F\setminus\{0,1\}$ and $A\in M_n(\F)$ we put
$\mathcal D(A)=\cC_{\omega^{-1}}(\cC_\omega(A))\cap
\cC_\omega(\cC_{\omega^{-1}}(A))$. By Lemma~\ref{vertex} and
Lemma~\ref{basic}(b), applied to $\omega$ and to $\omega^{-1}$, every
twin of a vertex $A$ lies in $\mathcal D(A)$. The proof of
Theorem~\ref{twins} rests on the following description of this space.

\begin{proposition}
\label{Dprop}
Let $\F$ be algebraically closed, let $\omega\in\F\setminus\{0,1\}$, and
let $A\in M_n(\F)$. Then $\mathcal D(A)\subseteq\F[A]$ if and only if $A$
satisfies condition (T) of Theorem~\ref{twins}.
\end{proposition}

\begin{proof}
We use the notation of Section~\ref{dgcsec}. The elements as in (T) are
the elements of $\Lambda_+\cap\Lambda_-$. For
$\sigma\in\Lambda$ put
$\ell(\sigma)=\max\{a(\omega\sigma),a(\omega^{-1}\sigma)\}$. Then
$N_\sigma^{\ell(\sigma)}(V_\sigma)$ is the intersection of
$N_\sigma^{a(\omega\sigma)}(V_\sigma)$ and
$N_\sigma^{a(\omega^{-1}\sigma)}(V_\sigma)$, and
$\Ker N_\sigma^{\ell(\sigma)}$ is the sum of
$\Ker N_\sigma^{a(\omega\sigma)}$ and
$\Ker N_\sigma^{a(\omega^{-1}\sigma)}$. Each of the conditions
$N_\sigma^{\ell(\sigma)}(V_\sigma)\ne\{0\}$ and
$\Ker N_\sigma^{\ell(\sigma)}\ne V_\sigma$ holds exactly when
$\ell(\sigma)<a(\sigma)$, that is, when $\sigma\in\Lambda_+\cap\Lambda_-$.

Let $Y$ be a linear map of $V$. We apply Theorem~\ref{structure} to
$\omega$ and $\varepsilon=\omega^{-1}$ and to $\omega^{-1}$ and
$\varepsilon=\omega$. Since $\omega^m\omega^{-1}=1$ and $\omega^{-m}\omega=1$
are equivalent for every integer $m\ge0$, conditions (a) of the two descriptions say that $Y_{0\tau}=O$ and $Y_{\sigma0}=O$ for
$\sigma,\tau\in\Lambda$ and that $Y_{00}$ lies in the span of the powers
$N_0^m$ with $m\ge0$ and $\omega^{m-1}=1$. Conditions (b) of the two
descriptions
together say that the block $Y_{\sigma\tau}$ with $\sigma\ne\tau$ in
$\Lambda$ maps $V_\tau$ into $N_\sigma^{\ell(\sigma)}(V_\sigma)$ and is
zero on $\Ker N_\tau^{\ell(\tau)}$. Thus $Y$ lies in
$\mathcal D(A)$ if and only if its blocks have these properties and
condition (c) of Theorem~\ref{structure} holds for $\omega$ and
$\varepsilon=\omega^{-1}$ and for $\omega^{-1}$ and $\varepsilon=\omega$.

Suppose that (T) holds, and let $Y\in\mathcal D(A)$. Then
$\Lambda_+\cap\Lambda_-$ has at most one element, so that $Y_{\sigma\tau}=O$ for all $\sigma\ne\tau$ in $\Lambda$,
and $Y_{00}$ is a polynomial in $N_0$. Let $\mu\in\Lambda$ lie outside
$\Lambda_+\cap\Lambda_-$. If $\mu\notin\Lambda_+$, then
$a(\omega\mu)\ge a(\mu)\ge1$, so that $\omega\mu\in\Lambda$, and
condition (c) for $\omega$ and Lemma~\ref{DU} show that $Y_{\mu\mu}$ is a
polynomial in $N_\mu$. If $\mu\notin\Lambda_-$,
condition (c) for $\omega^{-1}$ shows this in the same way. Let
$\Lambda_+\cap\Lambda_-=\{\nu\}$, put $c=a(\nu)$ and $N=N_\nu$, and let
$\F[N]v$ be the Jordan block of size $c$ at $\nu$. If $c=1$, then $V_\nu$
has dimension $1$. Let $c\ge2$, and suppose that $a(\omega\nu)=c-1$; the
case $a(\omega^{-1}\nu)=c-1$ is the same with $\omega$ and $\omega^{-1}$
exchanged. Then $\omega\nu\in\Lambda$ and
$N^{a(\omega\nu)}(V_\nu)=N^{c-1}(V_\nu)=\F N^{c-1}v$. Condition (c) for
$\omega$ at $\mu=\nu$ gives a polynomial $f$ such that
$R=Y_{\nu\nu}-f(\omega N)$ maps $V_\nu$ into $\F N^{c-1}v$ and
$\omega^{-1}Y_{\omega\nu,\omega\nu}=f(N_{\omega\nu})$, because
$\Ker N_{\omega\nu}^{a(\nu)}=\Ker N_{\omega\nu}^{c}=V_{\omega\nu}$.
Condition (c) for $\omega^{-1}$ at $\mu=\omega\nu$ gives a polynomial $g$
with $Y_{\omega\nu,\omega\nu}=g(\omega^{-1}N_{\omega\nu})$, because
$N_{\omega\nu}^{a(\nu)}(V_{\omega\nu})=N_{\omega\nu}^{c}(V_{\omega\nu})
=\{0\}$, and with $\omega Y_{\nu\nu}=g(N)$ on
$\Ker N^{a(\omega\nu)}=\Ker N^{c-1}$. Since $N_{\omega\nu}$
has nilpotency index $c-1$, the two descriptions of
$Y_{\omega\nu,\omega\nu}$ give
$g(x)\equiv\omega f(\omega x)\pmod{x^{c-1}}$. Every polynomial in $N$
that is divisible by $x^{c-1}$ is zero on $\Ker N^{c-1}$, so that
$Y_{\nu\nu}=\omega^{-1}g(N)=f(\omega N)$ on $\Ker N^{c-1}$. Thus $R$ is zero on $\Ker N^{c-1}$ and maps $V_\nu$ into
$\F N^{c-1}v$. The map $N^{c-1}$ has the same two properties and sends
$v$ to $N^{c-1}v\ne0$. Since $V_\nu=\F v\oplus\Ker N^{c-1}$, the map $R$
is a scalar multiple of $N^{c-1}$, and $Y_{\nu\nu}\in\F[N]$. Hence every
diagonal block of $Y$ is a polynomial in the corresponding nilpotent map,
and the Chinese remainder theorem gives $Y\in\F[A]$.

Suppose now that (T) fails. If $\sigma\ne\tau$ are elements of
$\Lambda_+\cap\Lambda_-$, then a nonzero linear map that sends $V_\tau$
into $N_\sigma^{\ell(\sigma)}(V_\sigma)$ and is zero on
$\Ker N_\tau^{\ell(\tau)}$ and on the other generalized eigenspaces
satisfies the conditions on the blocks of an element of
$\mathcal D(A)$, with all polynomials in condition (c) equal to zero. It lies in $\mathcal D(A)$ and is not a polynomial in $A$,
because it does not map $V_\tau$ into itself. Otherwise
$\Lambda_+\cap\Lambda_-=\{\nu\}$ for some $\nu$. Put $c=a(\nu)$ and
$N=N_\nu$, so that $\ell(\nu)<c$. Let $Z$ be a linear map of $V_\nu$
that is zero on $\Ker N^{\ell(\nu)}$ and maps $V_\nu$ into $N^{\ell(\nu)}(V_\nu)$, and extend $Z$ by zero on the
other generalized eigenspaces. Condition (c) for $\omega$ involves the block $Z$ at
$\mu=\nu$, where it requires that $Z$ map $V_\nu$ into
$N^{a(\omega\nu)}(V_\nu)\supseteq N^{\ell(\nu)}(V_\nu)$, and at
$\mu=\omega^{-1}\nu$, where it requires that $\omega^{-1}Z$ be zero on
$\Ker N^{a(\omega^{-1}\nu)}\subseteq\Ker N^{\ell(\nu)}$. In the same way
condition (c) for
$\omega^{-1}$ holds. Thus $Z$ lies in $\mathcal D(A)$, with all
polynomials in condition (c) equal to zero. If all these maps lie in
$\F[A]$, then their restrictions to $V_\nu$ are polynomials in $N$, and
Lemma~\ref{nilpoly} with $\ell_1=\ell_2=\ell(\nu)$ shows that $A$ has exactly
one Jordan block of size $c$ at $\nu$ and that $\ell(\nu)=c-1$, so that
(T) holds. Hence $\mathcal D(A)\not\subseteq\F[A]$ if (T) fails.
\end{proof}

\begin{proof}[Proof of Theorem~\ref{twins}]
Every twin of $A$ lies in $\mathcal D(A)$ by Lemmas~\ref{vertex}
and~\ref{basic}(b). By Lemma~\ref{scalars}, the span of
$\mathcal D(A)$ over $\Fbar$ is the space defined in the same way over
$\Fbar$, and $\mathcal D(A)\subseteq\F[A]$ holds if and only if the
corresponding inclusion holds over $\Fbar$. By
Proposition~\ref{Dprop}, this is the case if and only if $A$ satisfies
(T). If (T) holds, every twin of $A$ is therefore a polynomial in $A$.

Suppose that (T) fails and that $\F$ has more than $2n^2+1$ elements, and
choose $Y\in\mathcal D(A)$ with $Y\notin\F[A]$. For $t\in\F$ put
$B_t=A+tY$. Since $A$ and $Y$ lie in $\mathcal D(A)$, so does $B_t$. Lemma~\ref{basic}(b) for $\omega$ and
for $\omega^{-1}$ gives $\cC_\omega(A)\subseteq\cC_\omega(B_t)$ and
$\cC_{\omega^{-1}}(A)\subseteq\cC_{\omega^{-1}}(B_t)$. By
Lemma~\ref{generic} with $\Omega=\{\omega,\omega^{-1}\}$, there is
$t\ne0$ with
$\dim\cC_\omega(B_t)\le\dim\cC_\omega(A)$ and
$\dim\cC_{\omega^{-1}}(B_t)\le\dim\cC_{\omega^{-1}}(A)$, so that the two
inclusions are equalities. The matrix $B_t$ is not a polynomial in $A$,
because $t\ne0$ and $Y\notin\F[A]$. In particular $B_t\ne O$, and $B_t$
is a vertex by Lemma~\ref{vertex}(a), because
$\cC_\omega(B_t)=\cC_\omega(A)\ne\{O\}$. By Lemma~\ref{vertex}, $B_t$ is a
twin of $A$.
\end{proof}

\begin{remark}
\label{twinrem}
Condition (ii) of Theorem~\ref{mainA} implies condition (T). Under
condition (a) there is no element $\nu$ as in (T). Under condition (b),
the element $\nu$ of condition (b) is the only element $\mu$ with
$a(\mu)>a(\omega\mu)$ and $a(\mu)>a(\omega^{-1}\mu)$, because
$a(\omega\mu)=a(\mu)$ for $\mu\notin\{\nu,\omega^{-1}\nu\}$ and
$a(\omega\mu)=a(\nu)>a(\mu)$ for $\mu=\omega^{-1}\nu$. The other
requirements of (T) at $\nu$ are part of condition (b). If $\omega=-1$,
conditions (ii) and (T) are equivalent. Indeed, then
$\omega\mu=\omega^{-1}\mu=-\mu$ for every $\mu$, so that every $\mu$ that
is not an element as in (T) satisfies
$a(\mu)\le a(-\mu)$. If neither $\mu$ nor $-\mu$ is such an element, then
$a(\mu)=a(-\mu)$. Hence (T) gives condition (a) if there is no element as
in (T), and condition (b) with the same $\nu$ otherwise. This agrees with
Corollary~\ref{nbhd}(b) and~(c). When $\omega^2\ne1$, condition (T) is
weaker. The matrix $A=\diag(1,\omega)$ of Example~\ref{twoex} shows this.
It does not satisfy condition (ii), and some vertices that are not
polynomials in $A$ have the same successors as $A$. However,
$a(\omega\cdot1)=a(1)$ and $a(\omega^{-1}\cdot\omega)=a(\omega)$, so that
no element $\nu$ as in (T) exists. In fact, $\cC_\omega(A)=\F E_{21}$ and
$\cC_{\omega^{-1}}(A)=\F E_{12}$, and a matrix $Y$ with
$E_{21}Y=\omega^{-1}YE_{21}$ and $E_{12}Y=\omega YE_{12}$ is a multiple of
$A$. Hence the twins of $A$ are its nonzero multiples.
\end{remark}

The following lemma determines the polynomials in $A$ that lie in
$\cC_{\omega^{-1}}(\cC_\omega(A))$ and in
$\cC_\omega(\cC_{\omega^{-1}}(A))$. It holds for every matrix $A$.

\begin{lemma}
\label{Dpoly}
Let $\omega\in\F\setminus\{0,1\}$, $A\in M_n(\F)$, and $f\in\F[x]$. Then
$f(A)$ lies in $\cC_{\omega^{-1}}(\cC_\omega(A))$ if and only if the
monic greatest common divisor of $m_A(x)$ and $m_A(\omega x)$ divides
$f(\omega x)-\omega f(x)$, and the same holds for
$\cC_\omega(\cC_{\omega^{-1}}(A))$. In particular, $f(A)$ lies in
$\mathcal D(A)$ if and only if this divisibility holds.
\end{lemma}

\begin{proof}
If $X\in\cC_\omega(A)$, then $AX=X(\omega A)$, so that
$A^mX=X(\omega A)^m$ for all $m\ge0$ and $f(A)X=Xf(\omega A)$. Hence
$Xf(A)=\omega^{-1}f(A)X$ if and only if $X(f(\omega A)-\omega f(A))=O$.
We claim that for every polynomial $g\in\F[x]$ the equation $Xg(A)=O$
holds for all $X\in\cC_\omega(A)$ if and only if the monic greatest
common divisor of $m_A(x)$ and $m_A(\omega x)$ divides $g$. With
$g(x)=f(\omega x)-\omega f(x)$ this gives the assertion for
$\cC_{\omega^{-1}}(\cC_\omega(A))$.

Both conditions of the claim are unchanged if $\F$ is replaced by
$\Fbar$. For the first condition this follows from Lemma~\ref{scalars},
because the condition is linear in $X$, and for the second it holds
because minimal polynomials and greatest common divisors do not change
under extensions of the field. We may therefore assume that $\F$ is
algebraically closed, and we use the notation of Section~\ref{dgcsec}.
The polynomial $m_A(x)$ is the product of the factors
$(x-\lambda)^{a(\lambda)}$ over the eigenvalues $\lambda$ of $A$, and
$m_A(\omega x)$ is a nonzero constant times the product of the factors
$(x-\omega^{-1}\lambda)^{a(\lambda)}$. Therefore the monic greatest
common divisor of $m_A(x)$ and $m_A(\omega x)$ is the product of the
factors $(x-\lambda)^{\min\{a(\lambda),a(\omega\lambda)\}}$ over the
eigenvalues $\lambda$ of $A$.

By Corollary~\ref{annih}, the vectors that every $X\in\cC_\omega(A)$ maps
to zero form the space $\bigoplus_{\mu\in\Lambda}N_\mu^{a(\omega\mu)}(V_\mu)$.
For $g=0$ both conditions of the claim hold, so let $g\ne0$. The map $g(A)$ maps $V_0$
and each $V_\mu$ into itself, and it acts on $V_0$ as $g(N_0)$ and on
$V_\mu$ as $g(\mu I+N_\mu)$. Hence $Xg(A)=O$ for all $X\in\cC_\omega(A)$
if and only if $g(N_0)=O$ and
$g(\mu I+N_\mu)(V_\mu)\subseteq N_\mu^{a(\omega\mu)}(V_\mu)$ for all
$\mu\in\Lambda$. The first condition means that $x^{a(0)}$ divides $g$.
For the second, let $\mu\in\Lambda$, and let $k$ be the largest integer
such that $(x-\mu)^k$ divides $g$. Then $g(\mu I+N_\mu)$ is the product of
$N_\mu^k$ and an invertible map that commutes with $N_\mu$, so that
$g(\mu I+N_\mu)(V_\mu)=N_\mu^k(V_\mu)$. Since the spaces $N_\mu^i(V_\mu)$
with $0\le i\le a(\mu)$ decrease strictly, the inclusion holds if and
only if $k\ge\min\{a(\mu),a(\omega\mu)\}$, that is, if and only if
$(x-\mu)^{\min\{a(\mu),a(\omega\mu)\}}$ divides $g$. This proves the
claim.

For $\cC_\omega(\cC_{\omega^{-1}}(A))$, let $X\in\cC_{\omega^{-1}}(A)$.
Then $f(A)X=Xf(\omega^{-1}A)$, so that $Xf(A)=\omega f(A)X$ if and only
if $X(f(A)-\omega f(\omega^{-1}A))=O$. By the claim for $\omega^{-1}$ in
place of $\omega$, this holds for all $X\in\cC_{\omega^{-1}}(A)$ if and
only if the monic greatest common divisor of $m_A(x)$ and
$m_A(\omega^{-1}x)$ divides $f(x)-\omega f(\omega^{-1}x)$. After the
substitution of $\omega x$ for $x$, this condition says that the monic
greatest common divisor of $m_A(\omega x)$ and $m_A(x)$ divides
$f(\omega x)-\omega f(x)$.
\end{proof}

Under condition (T) the next theorem describes all twins of $A$.

\begin{theorem}
\label{twinsdesc}
Let $\omega\in\F\setminus\{0,1\}$, let $n\ge2$, and let $A$ be a vertex
of $\Gamma_\omega(M_n(\F))$ that satisfies condition (T) of
Theorem~\ref{twins}. A matrix $B$ is a twin of $A$ if and only if
$B=f(A)$ for a polynomial $f\in\F[x]$ with the following three
properties.
\begin{enumerate}
\item[(P1)] The monic greatest common divisor of $m_A(x)$ and
$m_A(\omega x)$ divides $f(\omega x)-\omega f(x)$.
\item[(P2)] If $\lambda$ and $\mu$ are eigenvalues of $A$ in $\Fbar$
with $f(\lambda)=\omega f(\mu)$, then $\lambda=\omega\mu$.
\item[(P3)] The derivative $f'$ is nonzero at every multiple root of
$m_A$.
\end{enumerate}
Every twin of $A$ satisfies (T) and is polynomially equivalent to $A$.
\end{theorem}

The eigenvalues $\lambda$ and $\mu$ in (P2) may be equal. Since
$\omega\ne1$, condition (P2) with $\lambda=\mu$ says that $f(\lambda)\ne0$
for every nonzero eigenvalue $\lambda$ of $A$.

\begin{proof}
Let $B$ be a twin of $A$. By Theorem~\ref{twins}, $B=f(A)$ for a
polynomial $f\in\F[x]$. The matrix $B$ lies in $\mathcal D(A)$ by Lemmas~\ref{vertex}
and~\ref{basic}(b), so that
$f$ satisfies (P1) by Lemma~\ref{Dpoly}. For (P2) and (P3) we may assume that $\F$ is
algebraically closed, because the equalities
$\cC_\omega(B)=\cC_\omega(A)$ and
$\cC_{\omega^{-1}}(B)=\cC_{\omega^{-1}}(A)$ remain true over $\Fbar$ by
Lemma~\ref{scalars}. We use the notation of Section~\ref{dgcsec}. For an
eigenvalue $\lambda$ of $A$ let $\widetilde N_\lambda$ be the restriction of
$B-f(\lambda)I$ to $V_\lambda$. Since $B$ acts on $V_\lambda$ as
$f(\lambda I+N_\lambda)$, the map $\widetilde N_\lambda$ is a polynomial in
$N_\lambda$ without constant term, and in particular $\widetilde N_\lambda$ is
nilpotent.

Suppose that $f(\lambda)=\omega f(\mu)$ for eigenvalues $\lambda$ and
$\mu$ with $\lambda\ne\omega\mu$. Choose a nonzero vector $u\in V_\lambda$
with $\widetilde N_\lambda u=0$. Since $\widetilde N_\mu$ is nilpotent, the sum of
$\widetilde N_\mu(V_\mu)$ and the generalized eigenspaces of $A$ other than $V_\mu$
is a proper subspace of $V$. Hence there is a linear map $X$ of $V$ with
image $\F u$ whose kernel contains this sum. Then $BX=f(\lambda)X$ and
$XB=f(\mu)X$, so that $BX=\omega XB$, that is,
$X\in\cC_\omega(B)=\cC_\omega(A)$. However, $X$ maps $V_\mu$ onto
$\F u$, which is not contained in $V_{\omega\mu}$ because
$\lambda\ne\omega\mu$, and this contradicts Lemma~\ref{decomp}. Thus (P2)
holds.

Suppose now that $f'(\lambda)=0$ for a multiple root $\lambda$ of $m_A$.
Then $N_\lambda\ne O$, and $\widetilde N_\lambda$ is the product of $N_\lambda^2$ and
a polynomial in $N_\lambda$, so that
$\widetilde N_\lambda(V_\lambda)\subseteq N_\lambda^2(V_\lambda)$ and
$\Ker N_\lambda^2\subseteq\Ker \widetilde N_\lambda$. Suppose first that
$\lambda=0$ or that $\omega\lambda$ is an eigenvalue. Then $\lambda$ is a
root of $m_A(x)$ and of $m_A(\omega x)$, so that
$f(\omega\lambda)=\omega f(\lambda)$ by (P1). Choose a nonzero vector
$u\in V_{\omega\lambda}$ with $N_{\omega\lambda}u=0$. Since
$N_\lambda^2(V_\lambda)$ is a proper subspace of $N_\lambda(V_\lambda)$,
there is a linear map $X$ of $V$ with image $\F u$ whose kernel contains
$N_\lambda^2(V_\lambda)$ and the generalized eigenspaces other than
$V_\lambda$, but not $N_\lambda(V_\lambda)$. Then
$BX=f(\omega\lambda)X$, because $\widetilde N_{\omega\lambda}u=0$, and
$XB=f(\lambda)X$. Thus $X\in\cC_\omega(B)=\cC_\omega(A)$, and by
Lemma~\ref{decomp} the restriction $X_\lambda\colon V_\lambda\to
V_{\omega\lambda}$ of $X$ satisfies
$N_{\omega\lambda}X_\lambda=\omega X_\lambda N_\lambda$. However,
$N_{\omega\lambda}X_\lambda=O$, while $X_\lambda N_\lambda\ne O$ because
the kernel of $X$ does not contain $N_\lambda(V_\lambda)$, a
contradiction.

It remains to treat the case that $\lambda\ne0$ and that $\omega\lambda$
is not an eigenvalue. Then $\omega^{-1}\lambda$ is an eigenvalue. Indeed, otherwise
$a(\omega\lambda)=a(\omega^{-1}\lambda)=0<a(\lambda)$, so that $\lambda$
is an element as in (T), and (T) gives $a(\lambda)=1$, which is false
because $\lambda$ is a multiple root of $m_A$. Since $\omega^{-1}\lambda$
is a root of $m_A(x)$ and of $m_A(\omega x)$, (P1) gives
$f(\lambda)=\omega f(\omega^{-1}\lambda)$. As $N_\lambda\ne O$, there
is a vector $u\in\Ker N_\lambda^2$ with $N_\lambda u\ne0$, and there is a
linear map $X$ of $V$ with image $\F u$ whose kernel contains
$N_{\omega^{-1}\lambda}(V_{\omega^{-1}\lambda})$ and the generalized
eigenspaces other than $V_{\omega^{-1}\lambda}$. Then $\widetilde N_\lambda u=0$, so
that $BX=f(\lambda)X$, and $XB=f(\omega^{-1}\lambda)X$. Hence
$X\in\cC_\omega(B)=\cC_\omega(A)$. By Lemma~\ref{decomp}, the restriction
$X_{\omega^{-1}\lambda}\colon V_{\omega^{-1}\lambda}\to V_\lambda$ of $X$
satisfies $N_\lambda X_{\omega^{-1}\lambda}=\omega
X_{\omega^{-1}\lambda}N_{\omega^{-1}\lambda}$. Here the right side is
zero and the left side is not. Thus (P3) holds.

Now let $f$ satisfy (P1) and (P2). We show that $f$ is injective on the
eigenvalues of $A$ in $\Fbar$. Since (P1), (P2), and (T) do not change
when $\F$ is replaced by $\Fbar$, we may assume that $\F$ is
algebraically closed. If $\lambda$ and
$\omega\lambda$ are eigenvalues, then $\lambda$ is a root of $m_A(x)$ and
of $m_A(\omega x)$, and (P1) gives $f(\omega\lambda)=\omega f(\lambda)$.
In particular, $f(0)=0$ if $0$ is an eigenvalue. Let $\lambda\ne\mu$ be
eigenvalues with $f(\lambda)=f(\mu)$. If $\lambda=0$, then
$f(\mu)=0=\omega f(0)$, and (P2) gives $\mu=0$, which is false. Hence
$\lambda\ne0$, and in the same way $\mu\ne0$. If $\omega\lambda$ were an
eigenvalue, then $f(\omega\lambda)=\omega f(\lambda)=\omega f(\mu)$, and
(P2) would give $\omega\lambda=\omega\mu$. If $\omega^{-1}\lambda$ were
an eigenvalue, then $f(\mu)=f(\lambda)=\omega f(\omega^{-1}\lambda)$, and
(P2) would give $\mu=\lambda$. Both conclusions contradict
$\lambda\ne\mu$. Hence neither $\omega\lambda$ nor $\omega^{-1}\lambda$
is an eigenvalue, and $\lambda$ is an element as in (T). In the same way
$\mu$ is such an element, which contradicts (T).

Let $f$ satisfy (P1), (P2), and (P3), and put $B=f(A)$. By
Lemma~\ref{Dpoly}, $B$ lies in $\cC_{\omega^{-1}}(\cC_\omega(A))$ and in
$\cC_\omega(\cC_{\omega^{-1}}(A))$, so that
$\cC_\omega(A)\subseteq\cC_\omega(B)$ and
$\cC_{\omega^{-1}}(A)\subseteq\cC_{\omega^{-1}}(B)$ by
Lemma~\ref{basic}(b). By Lemma~\ref{scalars} it suffices to show that the
spaces defined in the same way over $\Fbar$ have equal dimensions, so let
$\F$ be algebraically closed. Since $f$ is injective on the eigenvalues of
$A$, and $f(\lambda)\ne0$ for $\lambda\ne0$ by (P2), the generalized
eigenspaces of $B$ are the spaces $V_\lambda$, with the eigenvalues
$f(\lambda)$, and the generalized eigenspace of $B$ at $0$ is $V_0$. If
$N_\lambda\ne O$, then $\lambda$ is a multiple root of $m_A$, and
$f'(\lambda)\ne0$ by (P3). Then the restriction $\widetilde N_\lambda$ of
$B-f(\lambda)I$ to $V_\lambda$ is the product of $N_\lambda$
and an invertible polynomial in $N_\lambda$, so that $\widetilde N_\lambda^k$ and
$N_\lambda^k$ have the same kernel for every $k\ge1$. It follows that
$\widetilde N_\lambda$ and $N_\lambda$ have the same Jordan form, and there is an invertible map
$P_\lambda$ of $V_\lambda$ with $\widetilde N_\lambda=P_\lambda N_\lambda
P_\lambda^{-1}$. If $N_\lambda=O$, then $\widetilde N_\lambda=O$, and we put
$P_\lambda=I$.

Let $\mu$ be an eigenvalue of $A$. If $\omega\mu$ is an eigenvalue, then
$\omega f(\mu)=f(\omega\mu)$ by (P1), and if $\omega f(\mu)=f(\lambda)$
for an eigenvalue $\lambda$, then $\lambda=\omega\mu$ by (P2). Hence
$\omega f(\mu)$ is an eigenvalue of $B$ if and only if $\omega\mu$ is an
eigenvalue of $A$, and then $\omega f(\mu)=f(\omega\mu)$. Therefore Lemma~\ref{decomp} for $A$ and for $B$ shows that
$\cC_\omega(A)$ and $\cC_\omega(B)$ both consist of the maps $X$ that map
each $V_\mu$ into $V_{\omega\mu}$ and whose restrictions
$X_\mu\colon V_\mu\to V_{\omega\mu}$ satisfy
$N_{\omega\mu}X_\mu=\omega X_\mu N_\mu$ for $\cC_\omega(A)$ and
$\widetilde N_{\omega\mu}X_\mu=\omega X_\mu \widetilde N_\mu$ for $\cC_\omega(B)$. For the
eigenvalues $\mu$ such that $\omega\mu$ is an eigenvalue, a map $X_\mu$ satisfies the condition for $\cC_\omega(B)$ if and only if
$P_{\omega\mu}^{-1}X_\mu P_\mu$ satisfies the condition for
$\cC_\omega(A)$. Consequently,
$\dim\cC_\omega(B)=\dim\cC_\omega(A)$. In the same way, (P1) and (P2)
show that $\omega^{-1}f(\mu)$ is an eigenvalue of $B$ if and only if
$\omega^{-1}\mu$ is an eigenvalue of $A$, and that then
$\omega^{-1}f(\mu)=f(\omega^{-1}\mu)$, so that
$\dim\cC_{\omega^{-1}}(B)=\dim\cC_{\omega^{-1}}(A)$. Thus
$\cC_\omega(B)=\cC_\omega(A)$ and
$\cC_{\omega^{-1}}(B)=\cC_{\omega^{-1}}(A)$. Moreover, $B\ne O$, because
$I\in\cC_\omega(O)$ and $I\notin\cC_\omega(A)$. By Lemma~\ref{vertex},
$B$ is a vertex and a twin of $A$.

Finally, let $B=f(A)$ be a twin of $A$. Then $f$ satisfies (P1), (P2),
and (P3), and we have shown that $f$ is injective on the roots of $m_A$ in
$\Fbar$. By Theorem~\ref{thompson69}, $A$ is a polynomial in $B$, so that
$A$ and $B$ are polynomially equivalent. Over $\Fbar$, the Jordan blocks
of $B$ at $f(\lambda)$ have the same sizes as those of $A$ at $\lambda$,
because $\widetilde N_\lambda$ and $N_\lambda$ have the same Jordan form. By what we
have shown, the same holds at $\omega f(\lambda)$ and $\omega\lambda$ and
at $\omega^{-1}f(\lambda)$ and $\omega^{-1}\lambda$, where a number that
is not an eigenvalue has no Jordan blocks. Hence the elements as in (T)
for $B$ are the numbers $f(\nu)$ with $\nu$ an element as in (T) for $A$,
and $B$ satisfies (T).
\end{proof}

For $\omega$ of order at most $4$, the twins of a vertex that satisfies
(T) have the following description.

\begin{corollary}
\label{twinsq}
Let $\omega\in\F$ have finite order $q\in\{2,3,4\}$, let $n\ge2$, and let
$A$ be a vertex of $\Gamma_\omega(M_n(\F))$ that satisfies condition (T)
of Theorem~\ref{twins}. Then $\mathcal D(A)=A\F[A^q]$, and the twins of
$A$ are the matrices that are $q$\nobreakdash-polynomially equivalent to
$A$.
\end{corollary}

\begin{proof}
By Lemma~\ref{basic}(c), applied to $\omega$ and $\varepsilon=\omega^{-1}$
and to $\omega^{-1}$ and $\varepsilon=\omega$, the powers $A^{qi+1}$ with
$i\ge0$ lie in $\mathcal D(A)$, so that $A\F[A^q]\subseteq\mathcal D(A)$.
By Lemma~\ref{scalars}, the spans of these two spaces over $\Fbar$ are the
spaces defined in the same way over $\Fbar$, and each of the two spaces
consists of the matrices with entries in $\F$ in its span. Hence we may
assume that $\F$ is algebraically closed, and it suffices to show that
$\dim\mathcal D(A)\le\dim A\F[A^q]$.

By Proposition~\ref{Dprop}, $\mathcal D(A)\subseteq\F[A]$, and by
Lemma~\ref{Dpoly}, a matrix $f(A)$ lies in $\mathcal D(A)$ if and only if
the monic greatest common divisor of $m_A(x)$ and $m_A(\omega x)$ divides
$f(\omega x)-\omega f(x)$. We use the notation of Section~\ref{dgcsec}.
For $f\in\F[x]$, $\lambda\in\F$, and $j\ge0$ let $c_j(\lambda)$ be the
coefficient of $(x-\lambda)^j$ in the expansion of $f$ in powers of
$x-\lambda$. The matrix
$f(A)$ acts on $V_0$ as the sum of the maps $c_j(0)N_0^j$ with
$0\le j<a(0)$, and on $V_\lambda$ with $\lambda\in\Lambda$ as the sum of
the maps $c_j(\lambda)N_\lambda^j$ with $0\le j<a(\lambda)$. Since the
powers $N_\lambda^j$ with $0\le j<a(\lambda)$ are linearly independent,
$f(A)$ determines these numbers $c_j(\lambda)$, and they determine $f(A)$.

Let $f(A)\in\mathcal D(A)$. If $\lambda$ and $\omega\lambda$ are
eigenvalues, then $(x-\lambda)^{a(\lambda)}$ divides $m_A(x)$ and
$(x-\lambda)^{a(\omega\lambda)}$ divides $m_A(\omega x)$, so that
$(x-\lambda)^{\min\{a(\lambda),a(\omega\lambda)\}}$ divides
$f(\omega x)-\omega f(x)$. Since
$f(\omega x)=\sum_{j\ge0}c_j(\omega\lambda)\omega^j(x-\lambda)^j$, the
coefficient of $(x-\lambda)^j$ in the expansion of
$f(\omega x)-\omega f(x)$ in powers of $x-\lambda$ is
$\omega^jc_j(\omega\lambda)-\omega c_j(\lambda)$, and it is zero for
$j<\min\{a(\lambda),a(\omega\lambda)\}$, that is,
\begin{equation}
 \omega^jc_j(\omega\lambda)=\omega c_j(\lambda)~\text{for }
 0\le j<\min\{a(\lambda),a(\omega\lambda)\}.
 \label{taylorrel}
\end{equation}
For $\lambda=0$ this means that $c_j(0)=0$ for every $j<a(0)$ such that
$j-1$ is not divisible by $q$. For $\lambda\ne0$ it means that
$(\omega\lambda)^{j-1}c_j(\omega\lambda)=\lambda^{j-1}c_j(\lambda)$.

Let $j\ge0$. We claim that the numbers $\mu^{j-1}c_j(\mu)$ with $\mu$ in
one class of $\Lambda$ and $a(\mu)>j$ are all equal. Let $\mu$ and $\mu'$ be two such
elements. Since $\omega$ has order $q\le4$, either
$\mu'\in\{\mu,\omega\mu,\omega^{-1}\mu\}$, or $q=4$ and
$\mu'=\omega^2\mu$. In the first case \eqref{taylorrel} gives the claim
for $\mu$ and $\mu'$. In the second case $a(\omega\mu)>j$ or
$a(\omega^{-1}\mu)>j$, since otherwise $\mu$ and $\omega^2\mu$ would be two
elements as in (T). If $a(\omega\mu)>j$, we apply \eqref{taylorrel} to
$\lambda=\mu$ and to $\lambda=\omega\mu$, and otherwise to
$\lambda=\omega^{-1}\mu$ and to $\lambda=\omega^2\mu$, for which
$\omega\lambda=\omega^3\mu=\omega^{-1}\mu$. This proves the claim.

Hence, for every class of $\Lambda$ and every $j\ge0$, the numbers
$c_j(\mu)$ with $\mu$ in the class and $a(\mu)>j$ are determined by one
of them, and such numbers exist only if $j$ is smaller than the largest
value of $a$ on the class. Together with the condition on the numbers
$c_j(0)$, this shows that the numbers $c_j(\lambda)$ of the elements
$f(A)$ of $\mathcal D(A)$ lie in a space whose dimension is at most the number of integers $i\ge0$ with $qi+1<a(0)$ plus the sum of the largest values of
$a$ on the classes of $\Lambda$, one value for each class. By
Lemma~\ref{dimpow}, this bound is the dimension of $A\F[A^q]$. Hence
$\dim\mathcal D(A)\le\dim A\F[A^q]$, and $\mathcal D(A)=A\F[A^q]$.

If $B$ is $q$\nobreakdash-polynomially equivalent to $A$, then
$\cC_{\omega^k}(B)=\cC_{\omega^k}(A)$ for $0\le k<q$ by
Theorem~\ref{family}, in particular for $k=1$ and for $k=q-1$, where
$\omega^{q-1}=\omega^{-1}$. Moreover, $B\ne O$ because $A\in B\F[B^q]$.
By Lemma~\ref{vertex}, $B$ is a vertex and a twin of $A$. Conversely, let
$B$ be a twin of $A$. Then $B\in\mathcal D(A)=A\F[A^q]$ by
Lemma~\ref{vertex} and Lemma~\ref{basic}(b). By Theorem~\ref{twinsdesc},
$B$ satisfies (T), and $A$ is a twin of $B$. The equality $\mathcal D(A)=A\F[A^q]$ with the vertex $B$ in place of $A$
gives $A\in\mathcal D(B)=B\F[B^q]$. Thus $A$ and $B$ are
$q$\nobreakdash-polynomially equivalent.
\end{proof}

For $q\ge5$ the conclusion of Corollary~\ref{twinsq} fails, as the
following example shows.

\begin{example}
\label{twinex}
Let $\omega\in\F$ have finite order $q\ge5$, and let
$A=\diag(1,\omega,\omega^3)$. It is a vertex of $\Gamma_\omega(M_3(\F))$,
because $\omega$ is an eigenvalue of $A$ and of $\omega A$. Since
$a(1)=a(\omega)=a(\omega^3)=1$ and
$\omega^2$ and $\omega^4$ are not eigenvalues of $A$, the only element
$\nu$ as in (T) is $\omega^3$, and $A$ has one Jordan block of size $1$ at
$\omega^3$. Thus $A$ satisfies (T). The monic greatest common divisor of
$m_A(x)$ and $m_A(\omega x)$ is $x-1$, so that, for a polynomial
$f\in\F[x]$, condition (P1) of Theorem~\ref{twinsdesc} says that
$f(\omega)=\omega f(1)$, and (P3) holds
because $m_A$ has no multiple roots. Condition (P2) with $\lambda=\mu=1$
gives $f(1)\ne0$. Put $\theta=f(\omega^3)/f(1)$, so that
$f(A)=f(1)\diag(1,\omega,\theta)$. For the pairs $(\lambda,\mu)$ equal to
$(\omega^3,\omega^3)$, $(\omega,\omega^3)$, $(\omega^3,1)$,
$(\omega^3,\omega)$, and $(1,\omega^3)$, condition (P2) says that $\theta$
is not $0$, $1$, $\omega$, $\omega^2$, and $\omega^{-1}$, respectively,
and the other pairs give no further condition. Hence the twins of $A$ are
the nonzero multiples of the matrices $\diag(1,\omega,\theta)$ with
$\theta\notin\{0,1,\omega,\omega^2,\omega^{-1}\}$. Since $A^q=I$, we have
$A\F[A^q]=\F A$, and the twins with $\theta\ne\omega^3$ do not lie in
$A\F[A^q]$. Such twins exist. For $q\ge6$ we may take
$\theta=\omega^4$. For $q=5$ the six values $0$, $1$, $\omega$,
$\omega^2$, $\omega^3$, and $\omega^4$ are excluded, and $\F$ has more
than six elements. Indeed, if $\F$ is finite, then $5$ divides the number
of elements of $\F^\times$, and no field has exactly six elements. These twins are not $q$\nobreakdash-polynomially
equivalent to $A$, and the space $\mathcal D(A)$ contains them and is therefore larger than
$A\F[A^q]$.
\end{example}

\begin{remark}
\label{loopless}
Theorem~\ref{twins} also holds for the digraph without loops, which has
the same vertices as $\Gamma_\omega(M_n(\F))$ and an edge from a vertex
$A$ to a vertex $B$ if $A\ne B$ and $AB=\omega BA$. By Lemma~\ref{vertex}, the successors of a
vertex $A$ in this digraph are the nonzero elements of $\cC_\omega(A)$
other than $A$, and its predecessors are the nonzero elements of
$\cC_{\omega^{-1}}(A)$ other than $A$. The first set spans
$\cC_\omega(A)$, and the second spans $\cC_{\omega^{-1}}(A)$. Indeed, $A$
lies in one of these two spaces only if $A^2=O$, and then both sets
contain $\omega A$, of which $A$ is a multiple. Hence two
vertices that are twins in this digraph have the same
generalized centralizers with respect to $\omega$ and to $\omega^{-1}$,
and they are twins in $\Gamma_\omega(M_n(\F))$ by Lemma~\ref{vertex}.
This proves the first assertion of Theorem~\ref{twins} for the digraph
without loops.

Suppose now that $A$ does not satisfy (T) and that $\F$ has more than
$2n^2+1$ elements. By Theorem~\ref{twins}, $A$ has a twin $B$ in
$\Gamma_\omega(M_n(\F))$ that is not a polynomial in $A$. A nilpotent
matrix satisfies (T), because $a(\mu)=0$ for all $\mu\in\Fbar^\times$.
Thus $A^2\ne O$. We claim that $B^2\ne O$. Indeed, if $B^2=O$, then $B$
satisfies (T), and Theorem~\ref{twins} for the vertex $B$ shows that its
twin $A$ is a polynomial in $B$, that is, a linear combination of $I$ and
$B$. Then $AB$ and $BA$ are both equal to the product of $B$ and the
coefficient of $I$ in $A$. Since $B\in\cC_\omega(B)=\cC_\omega(A)$, we
have $AB=\omega BA$, and $\omega\ne1$ and $B\ne O$ show that this
coefficient is zero. Thus $A$ is a multiple of $B$, and $A^2=O$, a
contradiction. Therefore neither $A$ nor $B$ has a loop. Their successors and their predecessors
are the same in both digraphs, and $B$ is a twin of $A$ in the digraph
without loops.

Theorem~\ref{twinsdesc} also holds for the digraph without loops if
$A^2\ne O$. Indeed, let $B$ be a twin of such a vertex $A$ in
$\Gamma_\omega(M_n(\F))$. If $B^2=O$, then $A$ is a polynomial in $B$ by
Theorem~\ref{twinsdesc}, that is, $A=\alpha I+\beta B$ with
$\alpha,\beta\in\F$. Since $B\in\cC_\omega(B)=\cC_\omega(A)$, we have
$\alpha B=AB=\omega BA=\omega\alpha B$, so that $\alpha=0$ and
$A^2=\beta^2B^2=O$, which contradicts $A^2\ne O$. Hence neither $A$ nor
$B$ has a loop, and $B$ is a twin of $A$ in the digraph without loops. If
$A^2=O$, then $A$ has no twin other than itself in the digraph without
loops. Indeed, since $A\ne O$ and $A^2=O$, we have $m_A=x^2$, so that,
for a polynomial $f\in\F[x]$, (P1) says that $f(0)=0$, (P2) holds because $0$ is the only eigenvalue of
$A$, and (P3) says that $f'(0)\ne0$. By Theorem~\ref{twinsdesc}, the
twins of $A$ in $\Gamma_\omega(M_n(\F))$ are therefore the nonzero
multiples of $A$, and in the digraph without loops $A$ is a successor of
each of these multiples other than $A$, but not of $A$ itself.
\end{remark}

\section{Higher commutators}
\label{highsec}

For $A\in M_n(\F)$ let $L_A$ and $R_A$ be the linear maps $X\mapsto AX$
and $X\mapsto XA$ of $M_n(\F)$, so that $\ad_A=L_A-R_A$ is the map
$X\mapsto AX-XA$. The maps $L_A$
and $R_A$ commute. Hence, if $\F$ has characteristic $p>0$, then
\begin{equation}
 \ad_A^{p^e}=L_A^{p^e}-R_A^{p^e}=\ad_{A^{p^e}}~\text{for every integer}~e\ge0,
 \label{frobad}
\end{equation}
as Robinson~\cite[p.~286]{Robinson1971} noted. This section concerns the
kernel of $\ad_A^k$ and the following sets.

\begin{definition}
\label{anndef}
Let $A,B\in M_n(\F)$ and let $k\ge1$ be an integer. We put
\[
 \Ker\ad_A^k=\Bigl\{X\in M_n(\F):
 \sum_{i=0}^{k}(-1)^i\binom{k}{i}A^{k-i}XA^i=O\Bigr\},
\]
and the $k$th annihilator of $B$ is the set
\[
 \Ann_k(B)=\{X\in M_n(\F):\ad_X^k(B)=O\}
 =\Bigl\{X\in M_n(\F):
 \sum_{i=0}^{k}(-1)^i\binom{k}{i}X^{k-i}BX^i=O\Bigr\}.
\]
\end{definition}

Thus $\Ker\ad_A=\cC(A)$ and $\Ann_1(B)=\cC(B)$, so that
$\Ker\ad_A\subseteq\Ann_1(B)$ means $\cC(A)\subseteq\cC(B)$. The set $\Ker\ad_A^k$ is a subspace of $M_n(\F)$. The following theorem extends
to all $k$ the implication from $\cC(A)\subseteq\cC(B)$ to $B\in\F[A]$ in
Lemma~\ref{polyeq}. For $k\ge2$ the converse fails (see
Theorem~\ref{robeq}).

\begin{theorem}[Robinson {\cite[Proposition and Theorem~2]{Robinson1965}}]
\label{robinson}
Let $A,B\in M_n(\F)$ and $k\ge1$ with
$\Ker\ad_A^k\subseteq\Ann_k(B)$. Then $B\in\F[A]$. If, moreover, $\F$
has characteristic $p>0$ and $p^e$ with $e\ge0$ is the least power of
$p$ such that $p^e\ge k$, then $B\in\F[A^{p^e}]$.
\end{theorem}

Robinson later determined the set of the matrices $B$ with
$\Ker\ad_A^k\subseteq\Ann_k(B)$. In characteristic $0$ we state his
theorem for $k\ge2$, because for $k=1$ it is the double centralizer
theorem.

\begin{theorem}[Robinson {\cite[Theorem~5.5]{Robinson1971}}]
\label{robeq}
Let $A,B\in M_n(\F)$ and let $k\ge1$ be an integer. If $\F$ has
characteristic $p>0$ and $p^e$ with $e\ge0$ is the least power of $p$
such that $p^e\ge k$, then $\Ker\ad_A^k\subseteq\Ann_k(B)$ if and only if
$B\in\F[A^{p^e}]$. If $\F$ has characteristic $0$ and $k\ge2$, then
$\Ker\ad_A^k\subseteq\Ann_k(B)$ if and only if $B\in\F[A]$ and $B$ is
diagonalizable over $\Fbar$.
\end{theorem}

Robinson wrote the last set as the set of the polynomials in the
semisimple part of $A$. In characteristic $0$ the semisimple part of $A$
is the matrix $D\in\F[A]$ that is diagonalizable over $\Fbar$ and for
which $A-D$ is nilpotent. Robinson~\cite[Lemma~3.1]{Robinson1971} stated
that it exists and is unique. If $A$ is in Jordan form, then
$D$ is the diagonal part of $A$, as in Question~\ref{Q3}. For $f\in\F[x]$,
the matrix $f(A)$ is diagonalizable over $\Fbar$ if and only if $f(A)=f(D)$, because
$f(A)-f(D)$ is nilpotent and commutes with $f(D)$. Hence the polynomials
in $D$ are the matrices in $\F[A]$ that are diagonalizable over $\Fbar$,
and Theorem~\ref{robeq} states his result in these terms. We use the
following fact on matrices
with square zero and a lemma of Robinson on diagonalizable matrices.

\begin{lemma}
\label{twofacts}
Let $B,Z\in M_n(\F)$ with $Z^2=O$. Then $\ad_Z^2(B)=-2ZBZ$ and
$\ad_Z^3(B)=O$.
\end{lemma}

\begin{proof}
We have $\ad_Z^2(B)=Z^2B-2ZBZ+BZ^2=-2ZBZ$ and
$\ad_Z^3(B)=-2(Z^2BZ-ZBZ^2)=O$.
\end{proof}

\begin{lemma}[Robinson {\cite[p.~527, Lemma]{Robinson1965}}]
\label{semisimple}
Let $X\in M_n(\F)$ be diagonalizable over $\Fbar$, let $B\in M_n(\F)$,
and let $k\ge1$ be an integer with $\ad_X^k(B)=O$. Then $XB=BX$.
\end{lemma}

Robinson called these matrices $X$ semisimple and defined them by the
condition that the minimal polynomial of $X$ is coprime to its
derivative, which means that $X$ is diagonalizable over $\Fbar$. Thus
$\Ker\ad_X^k=\Ker\ad_X$ for such $X$, and $X\in\Ann_k(B)$ if and only if
$XB=BX$. The next corollary shows that the centralizer of $\Ker\ad_A^k$
is the set of the matrices $B$ described in Theorem~\ref{robeq}. In
characteristic $0$ its proof uses another lemma of Robinson.

\begin{lemma}[Robinson {\cite[Lemma~5.1]{Robinson1971}}]
\label{robsep}
Let $\F$ have characteristic $0$, and let $A\in M_n(\F)$. Every matrix
in $\F[A]$ that is diagonalizable over $\Fbar$ commutes with every matrix
$X$ such that $\ad_A^m(X)=O$ for some integer $m\ge1$.
\end{lemma}

We state only the case of characteristic $0$ and one implication of his
lemma, in the terms of Theorem~\ref{robeq}.

\begin{corollary}
\label{kercent}
Let $A\in M_n(\F)$ and let $k\ge1$ be an integer. If $\F$ has
characteristic $p>0$, then $\cC(\Ker\ad_A^k)=\F[A^{p^e}]$, where $p^e$
is as in Theorem~\ref{robeq}. If $\F$ has characteristic $0$ and $k\ge2$,
then $\cC(\Ker\ad_A^k)$ consists of the matrices in $\F[A]$ that are
diagonalizable over $\Fbar$.
\end{corollary}

\begin{proof}
If $B$ commutes with every element of $\Ker\ad_A^k$, then $\ad_X^k(B)=O$
for all $X\in\Ker\ad_A^k$, and $B$ lies in the set described by
Theorem~\ref{robeq}. Conversely, let $\F$ have characteristic $p$. By
\eqref{frobad} we have $\Ker\ad_A^k\subseteq\Ker\ad_A^{p^e}=\cC(A^{p^e})$,
and every element of $\F[A^{p^e}]$ commutes with every element of
$\Ker\ad_A^k$. In characteristic $0$, every matrix in $\F[A]$ that is
diagonalizable over $\Fbar$ commutes with every element of
$\Ker\ad_A^k$ by Lemma~\ref{robsep}.
\end{proof}

Roth~\cite[Theorem~3]{Roth1936} showed that
$\Ker\ad_A^k\subseteq\Ker\ad_{f(A)}^k$ for every polynomial $f$ and every
$k\ge1$. His proof writes $\ad_{f(A)}$ as the composition of $\ad_A$
with a map that commutes with $\ad_A$, and it works over every field.
Hence polynomially equivalent
matrices $A$ and $B$ satisfy $\Ker\ad_A^k=\Ker\ad_B^k$ for every $k\ge1$.
The next corollary concerns the converse.

\begin{corollary}
\label{samekernel}
Let $A,B\in M_n(\F)$ and $k\ge1$.
\begin{enumerate}
\item[(a)] Suppose that $\Ker\ad_A^k=\Ker\ad_B^k$. If $\F$ has
characteristic $p>0$ and $p^e$ is the least power of $p$ with
$p^e\ge k$, then $A^{p^e}$ and $B^{p^e}$ are polynomially equivalent. If
$\F$ has characteristic $0$ and $k\ge2$, then $\F[A]$ and $\F[B]$
contain the same matrices that are diagonalizable over $\Fbar$.
\item[(b)] If $\F$ has characteristic $p>0$ and $k=p^e$ with an integer
$e\ge0$, then
$\Ker\ad_A^k=\Ker\ad_B^k$ if and only if $A^{p^e}$ and $B^{p^e}$ are
polynomially equivalent.
\end{enumerate}
\end{corollary}

\begin{proof}
Part (a) follows from Corollary~\ref{kercent} and Lemma~\ref{polyeq}.
For (b) we use
$\Ker\ad_A^{p^e}=\cC(A^{p^e})$, which holds by~\eqref{frobad}, and
Lemma~\ref{polyeq}.
\end{proof}

\begin{example}
\label{c2ex}
For $k=2$ the equality $\Ker\ad_A^k=\Ker\ad_B^k$ does not imply that
$A$ and $B$ are polynomially equivalent. Let
$N=E_{13}+E_{24}$ and $N'=E_{14}+E_{23}$ in $M_4(\F)$. Then
$N^2=N'^2=O$, and $N$ and $N'$ have the same kernel and the same image,
namely the span of the first two standard basis vectors. By
Lemma~\ref{twofacts} with $N$ in place of $Z$ and $X$ in place of
$B$, we have $\ad_N^2(X)=-2NXN$ for every $X\in M_4(\F)$. Hence
$\Ker\ad_N^2$ is $M_4(\F)$ if
$\F$ has characteristic $2$, and otherwise it consists of the matrices
$X$ with $NXN=O$, that is, of the matrices that map this span into
itself. The same holds for
$N'$. Thus $\Ker\ad_N^2=\Ker\ad_{N'}^2$, while $N'$ is not a polynomial
in $N$, because $\F[N]$ is spanned by $I$ and $N$.
\end{example}

\begin{proof}[Proof of Theorem~\ref{annthm}]
If $B$ is a scalar matrix, then $\Ann_k(B)=M_n(\F)$, and
$\Ker\ad_A^k=M_n(\F)$ if and only if $\ad_A^k=0$. We assume from now on
that $\cS=\Ker\ad_A^k=\Ann_k(B)$ and show that (i) or (ii) holds; in
case (ii) we also prove the converse. The set $\cS$ is a subspace that
contains $I$, $\cC(A)$, and $\cC(B)$.

Suppose first that $k\ge3$, or that $k=2$ and $\F$ has characteristic
$2$. By Lemma~\ref{twofacts}, $\cS$ contains every matrix $Z$ with
$Z^2=O$. For $i\ne j$ the matrices $E_{ij}$ and $E_{ii}-E_{jj}+E_{ij}-E_{ji}$
have square zero. Hence the subspace $\cS$ contains all matrices of
trace $0$. Let the characteristic of $\F$ be different from $2$. For
$i\ne j$ the matrix $E_{ii}-E_{jj}$ is diagonal and lies in $\cS$, so
that it commutes with $B$ by Lemma~\ref{semisimple}, which gives
$2b_{ij}=0$ for the entries of $B=(b_{ij})$. Thus $B$ is diagonal. The
matrix $X=E_{ij}+E_{ji}$ lies in $\cS$ and is diagonalizable, because
$X^3=X$ and the polynomial $x^3-x$ has three distinct roots.
Hence it commutes with $B$, which gives $b_{ii}=b_{jj}$. Thus $B$ is a
scalar matrix. Let the characteristic of $\F$ be $2$ and $n\ge3$. For
distinct indices $i,j,l$ the diagonal matrix $E_{ii}+E_{jj}$ has trace
$0$ and commutes with $B$ by Lemma~\ref{semisimple}, which gives
$b_{il}=0$. Thus $B$ is diagonal. The matrix $X=E_{ij}+E_{ji}+E_{ii}+E_{ll}$ has
trace $0$, and it is diagonalizable over $\Fbar$, because $X^4=X$ and
$x^4-x$ has four distinct roots in $\Fbar$. Hence $X$ commutes with $B$,
which gives $b_{ii}=b_{jj}$, and $B$ is a scalar matrix. Let the
characteristic of $\F$ be $2$ and $n=2$. The matrices of trace $0$ form
a subspace of dimension $3$ of $\cS$. The dimension of
$\cS=\Ker\ad_A^k$ is not $3$: if $A$ has trace $0$, then $A^2$ is a
scalar matrix by the Cayley--Hamilton theorem,
$\ad_A^2=\ad_{A^2}=0$, and $\cS=M_2(\F)$; if the trace of $A$ is not
$0$, then $p_A$ has two distinct roots in $\Fbar$, $A$ is
diagonalizable over $\Fbar$, and $\cS=\cC(A)$ by Lemma~\ref{semisimple},
so that $\cS$ has dimension $2$.
Therefore $\cS=M_2(\F)$. The matrices $E_{11}$ and
$E_{11}+E_{12}+E_{21}$ are diagonalizable over $\Fbar$, the second
because its characteristic polynomial $x^2+x+1$ has distinct roots, and
they commute with $B$ by Lemma~\ref{semisimple}. Hence $B$ is a scalar
matrix. In all these cases (i) holds.

Suppose now that $k=2$ and that the characteristic of $\F$ is not $2$.
Let $X\in\cS$ and $Y\in\cC(B)$. Then $X+Y\in\cS$, and since
$\ad_Y(B)=O$ and $\ad_X^2(B)=O$, we obtain
$O=\ad_{X+Y}^2(B)=\ad_Y(\ad_X(B))$. Thus $XB-BX$ commutes with every
element of $\cC(B)$, and by Theorem~\ref{dct}
\begin{equation}
 XB-BX\in\F[B]~\text{for all }X\in\cS.
 \label{inFB}
\end{equation}

Let $n\ge3$, and suppose that $B$ is not a scalar matrix. We derive a
contradiction. A matrix for which every nonzero vector is an eigenvector
is a scalar matrix. Hence there is a vector $u\in\F^n$ such that $u$ and
$Bu$ are linearly independent. Let $v\in\F^n$ be a nonzero vector with
$v^\top u=v^\top Bu=0$, and put $Z=uv^\top$. Then $Z^2=O$ and $ZBZ=O$,
so that $Z\in\Ann_2(B)=\cS$ by Lemma~\ref{twofacts}. By \eqref{inFB} the
matrix $ZB-BZ=u(v^\top B)-(Bu)v^\top$ commutes with $B$, which gives
\begin{equation}
 2(Bu)(v^\top B)=(B^2u)v^\top+u(v^\top B^2).
 \label{rankone}
\end{equation}
If $u,Bu,B^2u$ were linearly independent, \eqref{rankone} would imply
$v^\top=0$. Hence $B^2u=\alpha u+\beta Bu$ with $\alpha,\beta\in\F$, and
\eqref{rankone} becomes
$(Bu)(2v^\top B-\beta v^\top)=u(\alpha v^\top+v^\top B^2)$. Since $u$
and $Bu$ are linearly independent, we get
$v^\top B=\gamma v^\top$ with $\gamma=\beta/2$ and
$v^\top B^2=-\alpha v^\top$, hence $\alpha=-\gamma^2$ and
$(B-\gamma I)^2u=0$. Put $N=B-\gamma I$. The scalar $\gamma$ depends
only on $u$, and $v^\top N=0$ for every $v$ with
$v^\top u=v^\top Bu=0$. Therefore the image of $N$ lies in the span of
$u$ and $Bu$, which is the span of $u$ and $Nu$. Since $N^2u=0$, we
obtain $N^3=O$, so that $\gamma$ is the only eigenvalue of $B$. Let
$u'$ be a vector with $Nu'\ne0$. Since $N$ is nilpotent, $u'$ and $Nu'$
are linearly independent, and hence so are $u'$ and $Bu'$. The same argument for $u'$ in place of $u$ gives an eigenvalue
$\gamma'$ of $B^\top$ with $(B-\gamma'I)^2u'=0$, so that
$\gamma'=\gamma$, $N^2u'=0$, and $N(\F^n)\subseteq\F u'+\F Nu'$. Hence
$N^2=O$. If the rank of $N$ were at least $2$, then
$N(\F^n)=\F u'+\F Nu'$ for such a vector $u'$, and $u'$ would lie in
the image of $N$, which is contained in the kernel of $N$. Since
$N\ne O$, the matrix $N$ has rank $1$, and $N^2=O$. In a suitable basis
of $\F^n$ we have $N=E_{12}$. Since
$\Ann_2(B)=\Ann_2(N)$ and the assumption $\Ker\ad_A^2=\Ann_2(B)$ is
preserved when $A$ and $B$ are replaced by $P^{-1}AP$ and $P^{-1}BP$ with
$P\in\GL_n(\F)$, we may assume that $N=E_{12}$. The matrices $Z_1=E_{31}$ and
$Z_2=E_{23}$ satisfy $Z_i^2=O$ and $Z_iNZ_i=O$, so that
$Z_1,Z_2\in\cS$ and $X=Z_1+Z_2\in\cS$. However, $X^2=E_{21}$ and
$XNX=E_{33}$, so that
$\ad_X^2(N)=X^2N-2XNX+NX^2=E_{22}-2E_{33}+E_{11}\ne O$. This
contradiction shows that $B$ is a scalar matrix. Then
$\Ann_2(B)=M_n(\F)$ and $\ad_A^2=0$, which is (i).

Let $n=2$. For a matrix $X\in M_2(\F)$ let $\tr(X)$ be its trace, and let
$\hat X=X-\frac12\tr(X)I$, which has trace $0$, and define $\hat A$ and $\hat B$
in the same way. The Cayley--Hamilton theorem for matrices of trace $0$
gives
$\hat X^2=\frac12\tr(\hat X^2)I$ and $\hat X\hat B+\hat B\hat X=\tr(\hat X\hat B)I$. Hence
\begin{equation}
 \ad_X^2(B)=\ad_{\hat X}^2(\hat B)
 =2\tr(\hat X^2)\hat B-2\tr(\hat X\hat B)\hat X .
 \label{twobytwo}
\end{equation}
Suppose that $B$ is not a scalar matrix, so that $\hat B\ne O$. By
\eqref{twobytwo}, a matrix $X$ lies in $\Ann_2(B)$ if and only if
$\hat X\in\F \hat B$, or $\hat X\notin\F \hat B$ and
$\tr(\hat X^2)=\tr(\hat X\hat B)=0$. Suppose that $\cS$ contains a matrix $X$ of
the second kind. Then $X+B\in\cS$ is also of the second kind, and
$\tr((\hat X+\hat B)^2)=0$ gives $\tr(\hat B^2)=0$ and $\hat B^2=O$. After a
change of basis $\hat B=E_{12}$. Then $\tr(\hat X\hat B)=0$ and
$\tr(\hat X^2)=0$ show that $\hat X$ is a multiple of $E_{12}$, a
contradiction. Hence $\cS=\F I+\F B=\cC(B)$. Since
$\cC(A)\subseteq\cS$ has dimension at least $2$, we obtain
$\cC(A)=\cC(B)$ and $\F[A]=\F[B]$. Write $\hat B^2=\delta I$. By
\eqref{twobytwo} with $A$ in place of $X$ and $X$ in place of $B$, and
since $\hat A$ is a nonzero multiple of $\hat B$, the space $\Ker\ad_A^2$
consists of the matrices $X$ with
$2\tr(\hat B^2)\hat X=2\tr(\hat B\hat X)\hat B$. If $\delta=0$, this space has
dimension $3$, which contradicts $\cS=\F I+\F B$. Hence $\delta\ne0$.
If $\delta=\rho^2$ with $\rho\in\F$, then $\hat B$ has the eigenvalues
$\rho$ and $-\rho$ with eigenvectors $w_+$ and $w_-$, and the linear map
$X$ with $Xw_-=w_+$ and $Xw_+=0$ has trace $0$. Hence $\hat X=X$, and $X$
satisfies $X\notin\F \hat B$, $X^2=O$, and $X\hat B+\hat BX=O$, so that
$X\in\Ann_2(B)$ by \eqref{twobytwo}, which contradicts $\cS=\F I+\F B$. Thus $\delta$ is
not a square in $\F$, and the characteristic polynomial
$(x-\frac12\tr B)^2-\delta$ of $B$ is irreducible. This is (ii).

Conversely, suppose that (ii) holds. Then $\hat A=c\hat B$ with
$c\in\F^\times$ and $\hat B^2=\delta I$, where $\delta$ is not a square in
$\F$. By \eqref{twobytwo} for $\ad_A^2(X)$, the space $\Ker\ad_A^2$
consists of the matrices $X$ with $2\delta \hat X=\tr(\hat B\hat X)\hat B$, which is
$\F I+\F B$. If $\hat X\notin\F \hat B$ satisfies
$\tr(\hat X^2)=\tr(\hat X\hat B)=0$, then $\hat X^2=O$ and $\hat X\hat B=-\hat B\hat X$, so
that $\hat B$ maps the kernel of $\hat X$, which is a line, into itself.
Then $\hat B$ has an eigenvalue $\rho\in\F$ with $\rho^2=\delta$, a
contradiction. Hence $\Ann_2(B)=\F I+\F B=\Ker\ad_A^2$.
\end{proof}

\begin{corollary}
\label{annbyann}
Let $A,B\in M_n(\F)$.
\begin{enumerate}
\item[(a)] The following statements are equivalent: the matrices $A$
and $B$ are polynomially equivalent; $\cC(A)=\cC(B)$;
$\Ker\ad_A^{k_1}\subseteq\Ann_{k_1}(B)$ and
$\Ker\ad_B^{k_2}\subseteq\Ann_{k_2}(A)$ for some integers
$k_1,k_2\ge1$; $\Ker\ad_A^{k_1}=\Ann_{k_1}(B)$ and
$\Ker\ad_B^{k_2}=\Ann_{k_2}(A)$ for some integers $k_1,k_2\ge1$.
\item[(b)] Let $n\ge2$ and $k\ge2$. If $A$ is nilpotent and
$\Ker\ad_A^k=\Ann_k(B)$, then $B$ is a scalar matrix.
\item[(c)] Let $n\ge2$ and $k_1,k_2\ge2$. Then
$\Ker\ad_A^{k_1}=\Ann_{k_1}(B)$ and $\Ker\ad_B^{k_2}=\Ann_{k_2}(A)$ if
and only if $A$ and $B$ are scalar matrices, or $n=2$, $k_1=k_2=2$, the
characteristic of $\F$ is not $2$, and $\F[A]=\F[B]$ is a field of
degree $2$ over $\F$.
\end{enumerate}
\end{corollary}

\begin{proof}
(a) The first two statements are equivalent by Lemma~\ref{polyeq},
and they imply the fourth with $k_1=k_2=1$. The fourth statement
implies the third, and the third implies $B\in\F[A]$ and $A\in\F[B]$ by
Theorem~\ref{robinson}.

(b) By Theorem~\ref{annthm}, either $B$ is a scalar matrix or
$\F[A]=\F[B]$ and the characteristic polynomial of $B$ is irreducible of
degree $2$. In the second case $\F[B]$ is a field, which contains no
nonzero nilpotent element, so that $A=O$ and $\F[B]=\F I$, a
contradiction.

(c) If $A$ and $B$ are scalar matrices, both equalities hold by
Theorem~\ref{annthm}(i). Suppose that both equalities hold. If
condition (i) of Theorem~\ref{annthm} holds for both, then $A$ and $B$
are scalar matrices. If condition (ii) holds for one of the two
equalities, then after exchanging $A$ and $B$ and $k_1$ and $k_2$ we may assume that
it holds for the first. Then $n=2$, $k_1=2$, and
$\F[A]=\F[B]$ is a field of degree $2$, so that neither $A$ nor $B$ is a
scalar matrix; hence condition (ii) holds for the second equality, and
$k_2=2$. Conversely, if $n=2$, the
characteristic is not $2$, and $\F[A]=\F[B]$ is a field of degree $2$
over $\F$, then the characteristic polynomials of $A$ and $B$ are
irreducible, and both equalities hold for $k_1=k_2=2$ by
Theorem~\ref{annthm}(ii).
\end{proof}

\begin{example}
\label{realex}
Let $\F=\mathbb R$ and $A=B=E_{12}-E_{21}$. The characteristic
polynomial $x^2+1$ of $B$ is irreducible over $\mathbb R$. By
Theorem~\ref{annthm} we have $\Ker\ad_A^2=\Ann_2(B)$. Since $A$ is
diagonalizable over $\mathbb C$, Lemma~\ref{semisimple} gives
$\Ker\ad_A^2=\cC(A)=\mathbb RI+\mathbb RB$. Over $\mathbb C$ the
equality fails: the matrix $X=E_{11}-E_{22}+\sqrt{-1}\,(E_{12}+E_{21})$
has trace $0$ and satisfies $X^2=O$ and $XB+BX=O$, so that $\ad_X^2(B)=O$ by \eqref{twobytwo}. However, $X$ does not commute
with $A$, and over $\mathbb C$ the kernel of $\ad_A^2$ is $\cC(A)$ by
Lemma~\ref{semisimple}, so that $X$ does not lie in this kernel.
\end{example}

The next proposition concerns a nilpotent Jordan block $A$ and the map
$X\mapsto AX+XA$ in place of $\ad_A$. The kernel of this map is
$\CL(A)$.

\begin{proposition}
\label{anticomm}
Let $n\ge1$ and $k\ge2$, let $A=J_n(0)$, and suppose that the
characteristic of $\F$ is $0$ or larger than $2n$. Let $B\in M_n(\F)$ be
a matrix such that every $X\in M_n(\F)$ with
$\sum_{i=0}^{k}\binom{k}{i}A^{k-i}XA^i=O$ satisfies
$\sum_{i=0}^{k}\binom{k}{i}X^{k-i}BX^i=O$. Then $B=O$.
\end{proposition}

\begin{proof}
The first sum is $(L_A+R_A)^k(X)$. Let
$G=\diag(g_1,\ldots,g_n)$ with $g_i=(-1)^{i-1}i$. Then
$AG+GA=\sum_{i=1}^{n-1}(g_i+g_{i+1})E_{i,i+1}
=\sum_{i=1}^{n-1}(-1)^iE_{i,i+1}$, and this matrix anticommutes with
$A$. Hence $(L_A+R_A)^2(G)=O$, and $G$ satisfies the first equation. By
assumption $\sum_{i=0}^{k}\binom{k}{i}G^{k-i}BG^i=O$. The entry of this
matrix in position $(s,t)$ is $(g_s+g_t)^kb_{st}$, where $b_{st}$ is the
entry of $B$ in this position. The integer
$g_s+g_t$ is $\pm(s+t)$ if $s-t$ is even and $\pm(s-t)$ if $s-t$ is
odd. In both cases its absolute value is between $1$ and $2n$, so that
it is not zero in $\F$. Hence $b_{st}=0$ for all $s,t$.
\end{proof}

\begin{example}
\label{anticommex}
The assumption on the characteristic in Proposition~\ref{anticomm}
cannot be dropped. Let $\F$ have characteristic $3$, and let $n=4$,
$k=3$, $A=J_4(0)$, and $B=E_{14}$. Since
$L_A$ and $R_A$ commute, $(L_A+R_A)^3=L_{A^3}+R_{A^3}$, and
$A^3=E_{14}$. The matrices $X=(x_{ij})$ with $E_{14}X+XE_{14}=O$ are
those with $x_{21}=x_{31}=x_{41}=x_{42}=x_{43}=0$ and $x_{44}=-x_{11}$.
For such a matrix the first column of $X^3$ is $x_{11}^3$ times the
first standard basis vector, and the last row of $X^3$ is $-x_{11}^3$
times the transpose of the last standard basis vector. Hence
$X^3B+BX^3=x_{11}^3E_{14}-x_{11}^3E_{14}=O$, and
$\sum_{i=0}^{3}\binom3iX^{3-i}BX^i=X^3B+BX^3=O$, although $B\ne O$.
\end{example}

\section{Concluding remarks}
\label{concl}

We have shown that the generalized centralizers of a matrix $A$ with
respect to all powers of a root of unity $\omega$ of order $q\ge2$
determine $A$ up to $q$\nobreakdash-polynomial equivalence, over every
field and without any assumption on $A$ (see Theorem~\ref{family}). If $\omega\ne1$ has finite order $q$, then a single generalized centralizer
$\cC_\omega(A)$ determines $A$ in this sense if and only if $A$ satisfies condition (a) or
condition (b) of Theorem~\ref{mainA}; for the necessity we assume that
$\F$ has more than $n^2+1$ elements (see Theorem~\ref{mainB}). Whenever $\varepsilon$ is a power of $\omega$, the same conditions describe
the matrices whose double generalized centralizer
$\cC_\varepsilon(\cC_\omega(A))$ consists of polynomials in $A$
(see Theorem~\ref{mainA}). For a scalar $\varepsilon\in\F^\times$ that is not
a power of $\omega$, condition (a) gives
$\cC_\varepsilon(\cC_\omega(A))=\{O\}$, and under condition (b) the inclusion
$\cC_\varepsilon(\cC_\omega(A))\subseteq\F[A]$ holds exactly when
$a(\nu)=1$ for the element $\nu$ of condition (b) or $A$ has a Jordan
block of size $a(\nu)-1$ at $\nu$.
The inclusion fails for a matrix that satisfies neither condition (see
Proposition~\ref{nonpow}). In the $\omega$\nobreakdash-commuting
digraph of Raja and Vaezpour, Theorem~\ref{mainB} describes the vertices
with the same successors as a given vertex $A$. For
$\omega\in\F\setminus\{0,1\}$, every twin of a vertex $A$ is a
polynomial in $A$ if and only if $A$ satisfies condition (T) of
Theorem~\ref{twins}; for the necessity we assume that $\F$ has more than
$2n^2+1$ elements. If $\omega^2\ne1$, condition (T) is weaker than
condition (ii) of Theorem~\ref{mainA} (see Remark~\ref{twinrem}). For a vertex $A$ that satisfies (T), Theorem~\ref{twinsdesc} describes
all twins of $A$ over every field. They are polynomially equivalent to $A$, and for $\omega$ of
finite order $q\le4$ they are the matrices that are
$q$\nobreakdash-polynomially equivalent to $A$ (see
Corollary~\ref{twinsq}). If $\omega$ has finite order $q\ge5$, this
description fails for some vertices (see Example~\ref{twinex}). For
higher commutators and matrices $A,B\in M_n(\F)$, the equality
$\Ker\ad_A^k=\Ann_k(B)$ with $k\ge2$ and $n\ge2$ holds exactly in the two
cases of Theorem~\ref{annthm}.

We do not know whether the assumptions that $\F$ has more than $n^2+1$
elements in Theorem~\ref{mainB}(b) and more than $2n^2+1$ elements in
Theorem~\ref{twins} can be dropped. The proofs use them only to find a
nonzero $t\in\F$ for which $A+tY$ has the same generalized centralizers
as $A$, where $Y$ is a suitable matrix that is not a polynomial in $A$
(see Lemma~\ref{generic}). For $\omega=-1$ the bound $n^2+1$ suffices in
Theorem~\ref{twins}, by Remark~\ref{twinrem} and Corollary~\ref{nbhd}(c).
If $n=2$, every vertex satisfies (T), so that the second assertion of
Theorem~\ref{twins} concerns only $n\ge3$. Indeed, a vertex of
$\Gamma_\omega(M_2(\F))$ is nilpotent, or it has the simple eigenvalues
$0$ and $\lambda\ne0$, and then $\lambda$ is the only element as in (T),
or it has the eigenvalues $\mu$ and $\omega\mu$ with $\mu\ne0$, and then
no element as in (T) exists.

If $\omega\in\F\setminus\{0,1\}$ and a vertex $A$ of $\Gamma_\omega(M_n(\F))$
does not satisfy condition (T), it would be interesting to describe all
twins of $A$.

For $A\in M_n(\F)$, $\omega\in\F^\times$, and $k\ge1$ the kernel of the
$k$th power of the
map $X\mapsto AX-\omega XA$ contains $\cC_\omega(A)$. When $\omega=1$, the centralizer of this kernel is $\F[A]$ if $k=1$, and it is described by
Corollary~\ref{kercent} if $k\ge2$. If $\omega\ne1$ and $k\ge2$, it would be interesting to determine the set
of the matrices $Y$ with $XY=\omega YX$
for all $X$ in this kernel.

\bibliographystyle{amsplain0}
\bibliography{paper}

\end{document}